\documentclass[12pt]{amsart}
\usepackage{amsmath}
\usepackage{amssymb}
\usepackage{graphicx}
\usepackage[all,cmtip]{xy}

 \newtheorem{theorem}{Theorem}[section]
 \newtheorem{corollary}[theorem]{Corollary}
 \newtheorem{lemma}[theorem]{Lemma}
 \newtheorem{proposition}[theorem]{Proposition}
 \newtheorem{definition}[theorem]{Definition}
 \newtheorem{remark}{Remark}[section]

 \newtheorem{example}{Example}[section]
 \numberwithin{equation}{section}
\begin{document}

\title[Moduli Spaces and CW Structures]
{On Moduli Spaces and CW Structures Arising from Morse Theory on
Hilbert Manifolds}

\author{LIZHEN QIN}

\address{Mathematics Department of Wayne State University, Detroit, MI 48202, USA}

\email{dv6624@wayne.edu}

%\thanks{This work was completed with the support of an Izaak
 %Walton Killam Memorial Scholarship.}

%\thanks{The author was also supported in part by the Research
 %Council of Slovenia.}

%\subjclass{Primary 47A15; Secondary 46A32, 47D20}


%\date{February 15, 1995 and, in revised form, July 6, 1995.}

%\dedicatory{}

%\commby{Daniel J. Rudolph}

%%% ----------------------------------------------------------------------

\begin{abstract}
This paper proves some results on negative gradient dynamics of
Morse functions on Hilbert manifolds. It contains the compactness of
flow lines, manifold structures of certain compactified moduli
spaces, orientation formulas, and CW structures of the underlying
manifolds.
\end{abstract}
%%% ----------------------------------------------------------------------
\maketitle
%%% ----------------------------------------------------------------------

%--------------------------------------------------------------------------------------------------------------------
%--------------------------------------------------------------------------------------------------------------------
%--------------------------------------------------------------------------------------------------------------------
%--------------------------------------------------------------------------------------------------------------------
\section{Introduction}
Invented in the 1920s (see \cite{morse1} and \cite{morse2}), Morse
theory has been a crucial tool in the study of smooth manifolds. In
the past two decades, largely due to the influence of A.\ Floer,
there has been a resurgence in activity in Morse theory in its
geometrical and dynamical aspects, especially in infinite
dimensional situations. An explosion of new ideas produced many
``oral theorems'' which were apparently widely acknowledged, highly
anticipated or even frequently used. Unfortunately, the literature
has not kept pace with the oral tradition. Some previously asserted
results are still stated without proof and, having asked various
experts in the field, the author could not ascertain what is
sufficiently proved or what is even regarded as true. The purpose of
this paper is to give a self-contained and detailed treatment
proving some of these claims.
\medskip

In the simplest instance, suppose one is given a Morse function on a
finite dimensional closed smooth manifold. By choosing a Riemannian
metric, one obtains a negative gradient flow. This determines a
stratification in which two points lie in the same stratum if they
lie on the same unstable manifold. Now each such unstable manifold
(or descending manifold) is homeomorphic to an open cell, and it is
desirable to know whether this open cell can be compactified in such
a way that it becomes the image of a closed cell arising from a CW
structure on the manifold. This is one of the problems we will be
addressing. Another related problem is to consider moduli spaces of
flow lines between any pair of critical points. Using piecewise flow
lines, one obtains a compactification of these moduli spaces.  The
question in this case to decide when one obtains a manifold with
corner structure from this compactification.

In addition to the finite dimensional case, our results will
generalize in two ways. Firstly, all of our results have an infinite
dimensional version in which the underlying manifold is a complete
Hilbert manifold and the Morse function satisfies Condition (C) and
has finite index at each critical point. This situation will be
called the {\it CF case} (see Definition \ref{CF_pair}). Secondly,
we will also strengthen some results in the finite dimensional case.
For example, we will obtain a certain result about simple homotopy
type in Theorem \ref{CW}.

The main results of this paper (see Section
\ref{section_main_results}) consist of nine theorems and one example
(Example \ref{not_c1}). All theorems are considered in CF case. The
results on compactness (Theorems \ref{flow_compactness} and
\ref{point_compactness}) require no more assumptions. Theorem
\ref{point_compactness} is even true in a more general setting.
Other theorems need two additional assumptions, transversality (see
Definition \ref{transverality}) and the local triviality of the
metric (see Definition \ref{locally_trivial_metric}). When the
compactification of descending manifolds is considered (see Theorems
\ref{d(p)_manifold}, \ref{disk}, \ref{CW} and
\ref{boundary_operator} and (2) of Theorem \ref{orientation}), the
Morse function is furthermore assumed to satisfy a lower bound
condition.

The following is a brief description of our main results.

Theorems \ref{flow_compactness} and \ref{point_compactness} are  two
results on the compactness. Roughly speaking, compactness means the
space of unbroken flow lines can be compactified by adding broken
flow lines. When the underlying manifold $M$ is finite dimensional,
similar results are well-known, for example, \cite[thm.\ 2.3, p.\
798]{smale3}, \cite[prop.\ 3]{burghelea_haller} and \cite[prop.\
2.35]{schwarz}. For the infinite dimensional Floer case, there are
results in \cite{floer1}, \cite{floer2} and \cite{salamon}. The
referees for this paper also referred the author to
\cite{abbondandolo_majer1} and \cite{abbondandolo_majer2} which
prove results similar to Theorem \ref{flow_compactness} (see Remark
\ref{compactness_remark_1}). Even in the finite dimensional case,
some assumptions on $M$ (e.g. compactness, both complete metrics and
Condition (c)) are needed in order to prove such results (compare
\cite[rem.,\ p. 13]{milnor1}).

Some spaces arise naturally from the study of negative gradient
dynamics. Let $\mathcal{D}(p)$ and $\mathcal{A}(p)$ be the
descending and ascending manifolds of a critical point $p$
respectively. Assuming transversality of the dynamics (see
Definition \ref{transverality}), let $\mathcal{W}(p,q)$ be the
intersection manifold of $\mathcal{D}(p)$ and $\mathcal{A}(q)$, and
$\mathcal{M}(p,q)$ be the orbit space of $\mathcal{W}(p,q)$ with
respect to the action of the flow (see Definitions
\ref{descending_ascending_manifold} and \ref{moduli_space}). It's
well-known that these manifolds can be compactified in a standard
way (see (\ref{compactified_space})). Theorems
\ref{m(p,q)_manifold}, \ref{d(p)_manifold} and \ref{w(p,q)_manifold}
consider the manifold structures of their compactified spaces.
Denote the compactified spaces by $\overline{\mathcal{M}(p,q)}$,
$\overline{\mathcal{D}(p)}$ and $\overline{\mathcal{W}(p,q)}$. A
central problem is to equip them with smooth structures in such a
way that they are manifolds with corners that are compatible with
the given stratifications. The smooth structure of
$\overline{\mathcal{M}(p,q)}$ is useful for some geometric
constructions in Morse Theory. For example, the papers
\cite{cohen_jones_segal1}, \cite{cohen_jones_segal2} and
\cite{franks} use the moduli spaces to recover the topology of the
underlying manifold. The smooth structure of
$\overline{\mathcal{D}(p)}$ is useful for Witten Deformations, for
example, see \cite{latour} and \cite{burghelea_haller}. The papers
\cite{bismut_zhang} and \cite{laudenbach} use the ``smooth
structure" of $e(\overline{\mathcal{D}(p)})$, where $e$ is the
evaluation map defined in (3) of Theorem \ref{d(p)_manifold}. The
smooth structure of $\overline{\mathcal{W}(p,q)}$ is useful for
computing the cup product of $H^{*}(M;R)$ via Morse Theory (see
\cite[sec.\ 2.4]{austin_braam} and \cite{viterbo}). To the best of
my knowledge, when $M$ is finite dimensional, and the metric is
locally trivial (see Definition \ref{locally_trivial_metric}), the
cases of $\overline{\mathcal{M}(p,q)}$ and
$\overline{\mathcal{D}(p)}$ are solved by \cite{latour} and
\cite{burghelea_haller}. (Actually, these two papers consider closed
$1$-forms which are more general than Morse functions.) The paper
\cite{burghelea_haller} gives a quick and nice proof. However, this
problem still remains open in the general case, in particular, when
the metric is nontrivial near the critical points. This problem is
closely related to the associative gluing of broken flow lines which
is also a well-known open problem. In addition, few papers in the
literature study $\overline{\mathcal{W}(p,q)}$.

In this paper, we extend the proof in \cite{burghelea_haller} to the
infinite dimensional CF case. This also includes the case of
$\overline{\mathcal{W}(p,q)}$. Our proofs of Theorems
\ref{m(p,q)_manifold} and \ref{d(p)_manifold} largely follow
\cite{burghelea_haller}. Subsection \ref{remark_manifold_structure}
presents a detailed remark on the literature, in particular, the
relations between this paper and \cite{burghelea_haller}.

Example \ref{not_c1} is another contribution of this paper to the
above problem. It shows that even if the answer to the above problem
is positive for a general metric, there are still some remarkable
differences from the locally trivial metric case even if the
underlying manifold is compact.

Theorem \ref{orientation} is a result on orientations. Since the
descending manifolds $\mathcal{D}(p)$ are finite dimensional, we can
assign orientations to them arbitrarily. This determines naturally
the orientations of $\mathcal{M}(p,q)$, $\mathcal{W}(p,q)$ and the
compactified manifolds $\overline{\mathcal{M}(p,q)}$,
$\overline{\mathcal{D}(p)}$ and $\overline{\mathcal{W}(p,q)}$. The
$1$-strata (see Definition \ref{k_stratum}) of these compactified
manifolds have two types of orientations, boundary orientations and
product orientations (see Subsection
\ref{subsection_definition_orientation} for details). Theorem
\ref{orientation} shows the relation between these two. Some results
on the finite dimensional case can be found, for example, in
\cite{austin_braam} and \cite{latour} (see Remarks
\ref{orientation_remark} and \ref{orientation_remark_latour}). These
orientation formulas have some applications. As pointed out in
\cite[prop.\ 2.8]{austin_braam}, the formula for
$\overline{\mathcal{M}(p,q)}$ ((1) of Theorem \ref{orientation})
gives an immediate proof of $\partial^{2} = 0$ for the Thom-Smale
complex in Morse homology. The formula for
$\overline{\mathcal{D}(p)}$ ((2) of Theorem \ref{orientation}) tells
us how to apply Stokes' theorem correctly when a differential form
is integrated on $\overline{\mathcal{D}(p)}$ (compare \cite[prop.\
6]{laudenbach}). In this paper, it together with Theorem \ref{disk}
also gives a straightforward proof of Theorem
\ref{boundary_operator}. As mentioned above, the papers
\cite{austin_braam} and \cite{viterbo} compute the cup product of
$H^{*}(M;R)$ via Morse Theory. Both \cite[(2.2)]{austin_braam} and
\cite[lem.\ 2 and 3]{viterbo} neglect signs. If we do care about the
signs in their formulas, the formula for
$\overline{\mathcal{W}(p,q)}$ ((3) of Theorem \ref{orientation}) can
tell us the answer (see Remark \ref{cup_product}).

The proof of Theorem \ref{orientation} is based on subtle
computations. If the underlying manifold $M$ is finite dimensional,
then it is locally orientable, and the proof follows easily from the
geometric constructions in \cite{burghelea_haller} although the
details are possibly lengthy. However, the issue is more complicated
in the infinite dimensional case since there is no way to give $M$ a
local orientation.

Finally, we consider the problem of constructing a CW structure from
the descending manifolds of the Morse function. Suppose a Morse
function $f$ on $M$ is lower bounded. It's well known that the
descending manifolds $\mathcal{D}(p)$ for the critical points $p$
are disjoint. Let $K^{a} = \bigsqcup_{f(p) \leq a} \mathcal{D}(p)$.
A natural question is whether or not $K^{a}$ is a CW complex with
open cells $\mathcal{D}(p)$. This has been considered by Thom
(\cite{thom}), Bott (\cite[p. 104]{bott}) and Smale (\cite[p.
197]{smale2}). If the answer is positive, then Morse theory will
give a compact manifold a bona fide CW decomposition which is
stronger than the homotopical CW approximation in \cite[thm.\
3.5]{milnor1}. In order to prove this, we have to construct a
characteristic map $e: D \longrightarrow M$ such that $e$ maps the
interior $D^{\circ}$ homeomorphically onto $\mathcal{D}(p)$ for each
$p$, where $D$ is a closed disk. This has been solved by \cite[thm.\
1]{kalmbach2} and \cite[rem.\ 3]{laudenbach} when $M$ is finite
dimensional and the metric is locally trivial. In this paper, these
results will be further improved as follows.

Actually, the papers \cite{kalmbach2} and \cite{laudenbach} show
that there exists such a characteristic map. Theorem
\ref{d(p)_manifold} shows that, even in the infinite dimensional CF
case, $\mathcal{D}(p)$ can be compactified to be
$\overline{\mathcal{D}(p)}$ and there is the map $e:
\overline{\mathcal{D}(p)} \longrightarrow M$ which is explicitly
constructed. If $\overline{\mathcal{D}(p)}$ is homeomorphic to a
closed disk (this is Theorem \ref{disk}), then $K^{a}$ is a CW
complex, and what's more, the characteristic maps $e:
\overline{\mathcal{D}(p)} \longrightarrow M$ are explicit. In order
to get an elementary proof of Theorem \ref{disk}, I asked Prof. John
Milnor for help. (Actually, there is a quick but non-elementary
proof based on the Poincar\'{e} Conjecture in all dimensions, see
Remark \ref{disk_non_elementary}.) I had not known the existence of
characteristic maps had been proved by \cite{kalmbach2} and
\cite{laudenbach} at that time. Prof. Milnor helped me greatly.
First, he referred me to \cite{kalmbach2}. Second, he suggested that
we may add a vector field to $-\nabla f$ on $\mathcal{D}(p)$ to
control the limit behavior of $-\nabla f$. Motivated by his
suggestion and \cite{kalmbach2}, I found the desired proof. In
particular, the key Lemma \ref{modified_vector} fulfills his
suggestion.

In addition, Theorem \ref{disk} and Lemma \ref{modified_vector} help
us prove more results. Let $M^{a} = f^{-1}((-\infty,a])$, the paper
\cite[cor.,\ p. 543]{kalmbach1} (see also \cite[sec.\
4.5]{kalmbach2}) shows that $K^{a}$ is a strong deformation retract
of $M^{a}$ when $f$ is lower bounded and proper and $a$ is regular.
Theorem \ref{CW} shows that, in this case, $M^{a}$ even has a CW
decomposition such that $K^{a}$ expands to $ M^{a}$ by elementary
expansions. The last theorem, Theorem \ref{boundary_operator},
computes the boundary operator of the CW chain complex associated
with $K^{a}$. This relates Morse homology to a cellular chain
complex (see Remark \ref{morse_homology}). The proofs of Theorems
\ref{CW} and \ref{boundary_operator} reflect the advantage of
Theorem \ref{disk} and Lemma \ref{modified_vector}.

The outline of this paper is as follows. Section
\ref{section_preliminaries} gives some definitions, notation and
elementary results mostly used in this paper. Section
\ref{section_main_results} formulates our main results. The
subsequent sections are the proofs of the main results.

%--------------------------------------------------------------------------------------------------------------------
%--------------------------------------------------------------------------------------------------------------------
\section{Preliminaries}\label{section_preliminaries}
In this paper, we assume $M$ is a Hilbert manifold with a
\textit{complete} Riemannian metric. The completeness of the metric
is necessary for Theorem \ref{bounded} (compare \cite[rem.,\ p.
13]{milnor1}). Let $f$ be a Morse function on $M$. Denote the index
of a critical point $p$ by $\textrm{ind}(p)$. Denote $f^{-1}([a,b])$
by $M^{a,b}$. Denote $f^{-1}((-\infty, a])$ by $M^{a}$.

We need the well-known Condition (C) or Palais-Smale Condition (see
\cite{palais_smale}). \\

\noindent \textbf{Condition (C):} \textit{If $S$ is a subset of $M$
on which $f$ is bounded but on which $\| \nabla f \|$ is not bounded
away from $0$, then there is a critical point of $f$ in the closure
of $S$.} \\

Assuming this condition, its easy to prove the following results.
Good references are \cite[thm.\ 1 and 2]{palais_smale},
\cite{palais} and \cite[sec.\ 9.1]{palais_terng}.

\begin{theorem}\label{finite_number}
If $(M, f)$ satisfies Condition (C), then for all $a, b$ such that
$-\infty < a < b < +\infty$, $M^{a,b}$ contains only finite many
critical points.
\end{theorem}

We cite \cite[thm.\ (3), p. 333]{palais} as follows.

\begin{theorem}\label{bounded}
If $(M, f)$ satisfies Condition (C), $x \in M$, and $\phi_{t}(x)$ is
the maximal flow of $-\nabla f$ with initial value $x$, then
$\phi_{t}(x)$ satisfies one of the following two conditions:

(1) $f(\phi_{t}(x))$ has no lower (upper) bound; or

(2) $f(\phi_{t}(x))$ has a lower (upper) bound, $\phi_{t}(x)$ can be
defined as a function of $t$ on $[0,+\infty)$ ($(-\infty,0]$),
$\displaystyle \lim_{t \rightarrow +\infty} \phi_{t}(x)$
($\displaystyle \lim_{t \rightarrow -\infty} \phi_{t}(x)$) exists
and is a critical point of $f$.
\end{theorem}

By Theorem \ref{bounded}, we get an immediate corollary.
\begin{corollary}\label{start_terminate}
Suppose $(M,f)$ satisfies Condition (C) and $-\infty < a < b <
+\infty$. Then all flow lines in $M^{a,b}$ are from $f^{-1}(b)$ or a
critical point in $M^{a,b}$ to $f^{-1}(a)$ or a critical point in
$M^{a,b}$.
\end{corollary}

\begin{definition}\label{descending_ascending_manifold}
Let $\phi_{t}(x)$ be the flow generated by $- \nabla f$ with initial
value $x$. Suppose $p$ is a critical point. Define the descending
manifold of $p$ to be $\mathcal{D}(p) = \{ x \in M \mid
\displaystyle \lim_{t \rightarrow - \infty} \phi_{t}(x) = p \}$.
Define the ascending manifold of $p$ to be $\mathcal{A}(p) = \{ x
\in M \mid \displaystyle \lim_{t \rightarrow + \infty} \phi_{t}(x) =
p \}$.
\end{definition}

Both $\mathcal{D}(p)$ and $\mathcal{A}(p)$ are embedded submanifolds
diffeomorphic to (maybe infinite dimensional) open disks. By Theorem
\ref{finite_number} and Corollary \ref{start_terminate}, we get the
following.

\begin{corollary}\label{flow_map}
Suppose $(M,f)$ satisfies Condition (C) and $-\infty < a < b <
+\infty$. Suppose $\{ p_{1}, \cdots, p_{n} \}$ consists of all
critical points in $M^{a,b}$. Denote $\mathcal{A}(p_{i}) \cap
f^{-1}(b)$ by $S^{+}_{i}$, and $\mathcal{D}(p_{i}) \cap f^{-1}(a)$
by $S^{-}_{i}$. Then the flow map can be defined and gives a
diffeomorphism:
\[
  \psi : f^{-1}(b) - \bigcup_{i=1}^{n} S^{+}_{i} \longrightarrow f^{-1}(a) - \bigcup_{i=1}^{n}
  S^{-}_{i}.
\]
In particular, if there is no critical point in $M^{a,b}$, we have
the following diffeomorphism:
\[
  \psi : f^{-1}(b) \longrightarrow f^{-1}(a).
\]
Here, if $x \in f^{-1}(b)$, $\phi_{t}(x) = y \in f^{-1}(a)$ for some
$t$, the flow map is defined by $\psi(x)=y$.
\end{corollary}

\begin{remark}
Although we use the notation $S_{i}^{\pm}$ in Corollary
\ref{flow_map}, $S_{i}^{\pm}$ are not necessarily homeomorphic to
spheres.
\end{remark}

\begin{definition}\label{CF_pair}
If $(M, f)$ satisfies Condition (C) and $\textrm{ind}(p) < + \infty$
for all critical points $p$, then we call $(M, f)$ a CF pair.
\end{definition}

\begin{definition}\label{transverality}
If the descending manifold $\mathcal{D}(p)$ and the ascending
manifold $\mathcal{A}(q)$ are transversal for all critical points
$p$ and $q$, then we say $- \nabla f$ satisfies transversality.
\end{definition}

\begin{remark}
Some papers in the literature call Definition \ref{transverality}
Morse-Smale Condition.
\end{remark}

If $- \nabla f$ satisfies transversality, then $\mathcal{D}(p) \cap
\mathcal{A}(q)$ is an embedded submanifold which consists of points
on flow lines from $p$ to $q$. Since a flow line has an $R$-action,
we may take the quotient of $\mathcal{D}(p) \cap \mathcal{A}(q)$ by
this $R$-action, i.e. consider its orbit space acted upon by the
flow. This leads to the following definition. (See also
\cite[observation 4]{burghelea_haller}, \cite[p.
3]{cohen_jones_segal1}, \cite[defn.\ 2.32]{schwarz} and \cite[p.
158]{banyaga_hurtubise}.)

\begin{definition}\label{moduli_space}
Suppose $- \nabla f$ satisfies transversality. Define
$\mathcal{W}(p,q) = \mathcal{D}(p) \cap \mathcal{A}(q)$. Define the
moduli space $\mathcal{M}(p,q)$ to be the orbit space
$\mathcal{W}(p,q)/R$.
\end{definition}

Clearly, both $\mathcal{W}(p,q)$ and $\mathcal{M}(p,q)$ are smooth
manifolds. Suppose $\gamma_{1}$ and $\gamma_{2}$ are two flow lines
such that $\gamma_{1}(-\infty) = \gamma_{2}(-\infty) = p$,
$\gamma_{1}(+\infty) = \gamma_{2}(+\infty) = q$ and $\gamma_{1}(0) =
\gamma_{2}(t_{0})$ for some $t_{0} \neq 0$. Then $\gamma_{1}$ and
$\gamma_{2}$ are two distinct flow lines which represent the same
point of $\mathcal{M}(p,q)$. For convenience and briefness, we
identify them as the same flow line. Then $\mathcal{M}(p,q)= \{
\gamma \mid \gamma$ is a flow line, $\gamma(-\infty) = p$ and
$\gamma(+\infty) = q. \}$. Suppose $a \in (f(q), f(p))$ is a regular
value. For all $\gamma \in \mathcal{M}(p,q)$, it intersects with
$f^{-1}(a)$ at a unique point. This gives $\mathcal{M}(p,q)$ a
natural identification with $\mathcal{W}(p,q) \cap f^{-1}(a)$ which
is a diffeomorphism.

We generalize the concept of flow lines. Suppose $\gamma$ is a flow
line. If it passes through a singularity, it is a constant flow
line. Otherwise, it is nonconstant. The following definition is
slightly different from the ``broken trajectories" in \cite[defn.\
4]{burghelea_haller}.

\begin{definition}\label{generalized_flow_line}
An ordered sequence of flow lines $\Gamma = (\gamma_{1},\cdots,
\gamma_{n})$, $n \geq 1$, is a generalized flow line if
$\gamma_{i}(+\infty) = \gamma_{i+1}(-\infty)$ and $\gamma_{i}$ are
constant or nonconstant alternatively according the order of their
places in the sequence. $\gamma_{i}$ is a component of $\Gamma$.
$\Gamma$ is a unbroken generalized flow line if $n=1$ and a broken
generalized flow line if $n>1$.
\end{definition}

\begin{example}
Suppose $p$ is a singularity. Assume $\gamma_{1}$, $\gamma_{2}$ and
$\gamma_{3}$ are flow lines in which $\gamma_{1}$ and $\gamma_{3}$
are nonconstant and $\gamma_{1}(+\infty) = \gamma_{3}(-\infty) = p$,
$\gamma_{2}(t) \equiv p$. Then $(\gamma_{1})$, $(\gamma_{1},
\gamma_{2})$, $(\gamma_{2}, \gamma_{3})$ and $(\gamma_{1},
\gamma_{2}, \gamma_{3})$ are generalized flow lines, $(\gamma_{1})$
is unbroken, and others are broken. Furthermore, $(\gamma_{1},
\gamma_{3})$ is not a generalized flow line.
\end{example}

For convenience, we may identify a flow line $\gamma$ with the
generalized flow line $(\gamma)$. Definition
\ref{generalized_flow_line} is a generalization of flow lines.

\begin{definition}\label{generalized_flow_line_points}
Suppose $x$ and $y$ are two points in $M$. A generalized flow line
$(\gamma_{1},\cdots, \gamma_{n})$ connects $x$ and $y$ if there
exist $t_{1}, t_{2} \in (-\infty, +\infty)$ such that
$\gamma_{1}(t_{1}) = x$ and $\gamma_{n}(t_{2}) = y$. A point $z$ is
a point on $(\gamma_{1},\cdots, \gamma_{n})$ if there exists
$\gamma_{i}$ and $t \in (-\infty, +\infty)$ such that $\gamma_{i}(t)
= z$.
\end{definition}

\begin{example}
Suppose $p$ and $q$ are two critical points. Let $\gamma_{1}$,
$\gamma_{2}$ and $\gamma_{3}$ be flow lines such that $\gamma_{1}(t)
\equiv p$, $\gamma_{3}(t) \equiv q$, $\gamma_{2}(-\infty) = p$ and
$\gamma_{2}(+\infty) = q$. Then $(\gamma_{1}, \gamma_{2},
\gamma_{3})$ is a generalized flow line connecting $p$ and $q$,
while $\gamma_{2}$ is not.
\end{example}

We need to consider the relations between two critical points.

\begin{definition}\label{points_partial_order}
Suppose $p$ and $q$ are two critical points. We define the relation
$p \succeq q$ if there is a flow line from $p$ to $q$. We define the
relation $p \succ q$ if $p \succeq q$ and $p \neq q$.
\end{definition}

\begin{definition}\label{critical_sequence}
An ordered set $I = \{ r_{0}, r_{1}, \cdots, r_{k+1} \}$ is a
critical sequence if $r_{i}$ ($i=0, \cdots, k+1$) are critical
points and $r_{0} \succ r_{1} \succ \cdots \succ r_{k+1}$. We call
$r_{0}$ the head of $I$, and $r_{k+1}$ the tail of $I$. The length
of $I$ is $|I|=k$.
\end{definition}

Suppose $I = \{ r_{0}, r_{1}, \cdots, r_{k+1} \}$ is a critical
sequence. We denote the following product manifolds by
$\mathcal{M}_{I}$ and $\mathcal{D}_{I}$.
\begin{equation}\label{M_I}
\mathcal{M}_{I} = \prod_{i=0}^{k} \mathcal{M}(r_{i}, r_{i+1}),
\qquad \mathcal{D}_{I} = \prod_{i=0}^{k} \mathcal{M}(r_{i}, r_{i+1})
\times \mathcal{D}(r_{k+1}).
\end{equation}

We shall consider the manifold structures of  compactifications of
the spaces $\mathcal{M}(p,q)$, $\mathcal{D}(p)$ and
$\mathcal{W}(p,q)$. They usually have corners. For the definition of
manifold with corners, we follow \cite[p. 2]{douady} and \cite[sec.\
1.1]{janich}.

\begin{definition}\label{manifold_with_corner}
A smooth manifold with corners is a space defined in the same way as
a smooth manifold except that its atlases are open subsets of $[0, +
\infty)^{n}$.
\end{definition}

If $L$ is a smooth manifold with corners, $x \in L$, a neighborhood
of $x$ is differomorphic to $[0, \epsilon)^{k} \times (0,
\epsilon)^{n-k}$, then define $c(x) = k$. Clearly, $c(x)$ does not
depend on the choice of atlas. We call a union of some components of
$\{ x \in L \mid c(x) = 1 \}$ a face.

\begin{definition}\label{manifold_with_face}
A smooth manifold with faces is a smooth manifold with corners such
that each $x$ belongs to the closures of $c(x)$ different connected
faces.
\end{definition}

Now we introduce another definition.

\begin{definition}\label{k_stratum}
Suppose $L$ is a smooth manifold. $\{ x \in L \mid c(x) = k \}$ is
the $k$-stratum of $L$. Denote it by $\partial^{k} L$.
\end{definition}

Clearly, faces and the $k$-strata are manifolds in the usual sense.
They are also submanifolds of $L$ of codimension $1$ and $k$
respectively.

Suppose $p$ is critical point. By the Morse Lemma, there exist
$\epsilon
> 0$ and a diffeomorphism
\begin{equation}\label{localization_map}
h: B(\epsilon) \longrightarrow U
\end{equation}
such that
\begin{equation}\label{localization_function}
  f \circ h (v_{1}, v_{2}) = f(p) - \frac{1}{2} \langle v_{1}, v_{1}
  \rangle + \frac{1}{2} \langle v_{2}, v_{2} \rangle.
\end{equation}
Here $B(\epsilon) = \{ (v_{1}, v_{2}) \in T_{p} M \mid v_{1} \in
V_{-}, v_{2} \in V_{+}, \| v_{1} \|^{2} < 2\epsilon$ and $\| v_{2}
\|^{2} < 2\epsilon \}$, $V_{-} \times \{0\}$ is the negative
spectrum space of $\nabla^{2} f$ and $\{0\} \times V_{+}$ is the
positive spectrum space of $\nabla^{2} f$, $U$ is a neighborhood of
$p$ and $h(0,0) = p$.

\begin{definition}\label{locally_trivial_metric}
If the map $h$ in (\ref{localization_map}) also preserves the
metric, then we say that the metric of $M$ is locally trivial at
$p$. If it is locally trivial at each critical point, then we say
that the metric on $M$ is locally trivial.
\end{definition}

If the metric is locally trivial at $p$, then we have
\begin{equation}\label{localization_dymamics}
- \nabla f|_{U} = dh \cdot (v_{1}, -v_{2}).
\end{equation}

When the metric is locally trivial, Figure \ref{model_figure} shows
the standard model of the neighborhood $U$, where $U$ is identified
with $B(\epsilon)$. Here, $a$ and $b$ are regular values such that
$b < f(p) < a$, and $f^{-1}(a)$ and $f^{-1}(b)$ are two level
surfaces. The arrows indicate the directions of the flow. The points
$(v_{1}, v_{2})$, $(v_{3}, v_{4})$ and $(v_{5}, v_{6})$ are on the
same flow line, whereas $(v_{7}, v_{8})$, $(v_{9}, v_{10})$,
$(v_{11}, v_{12})$ and $(0,0)$ are on the same broken generalized
flow line. Figure \ref{model_figure} will provide geometric
intuition for the arguments in this paper.

\begin{figure}[!htbp]
\centering
\includegraphics[scale=0.3]{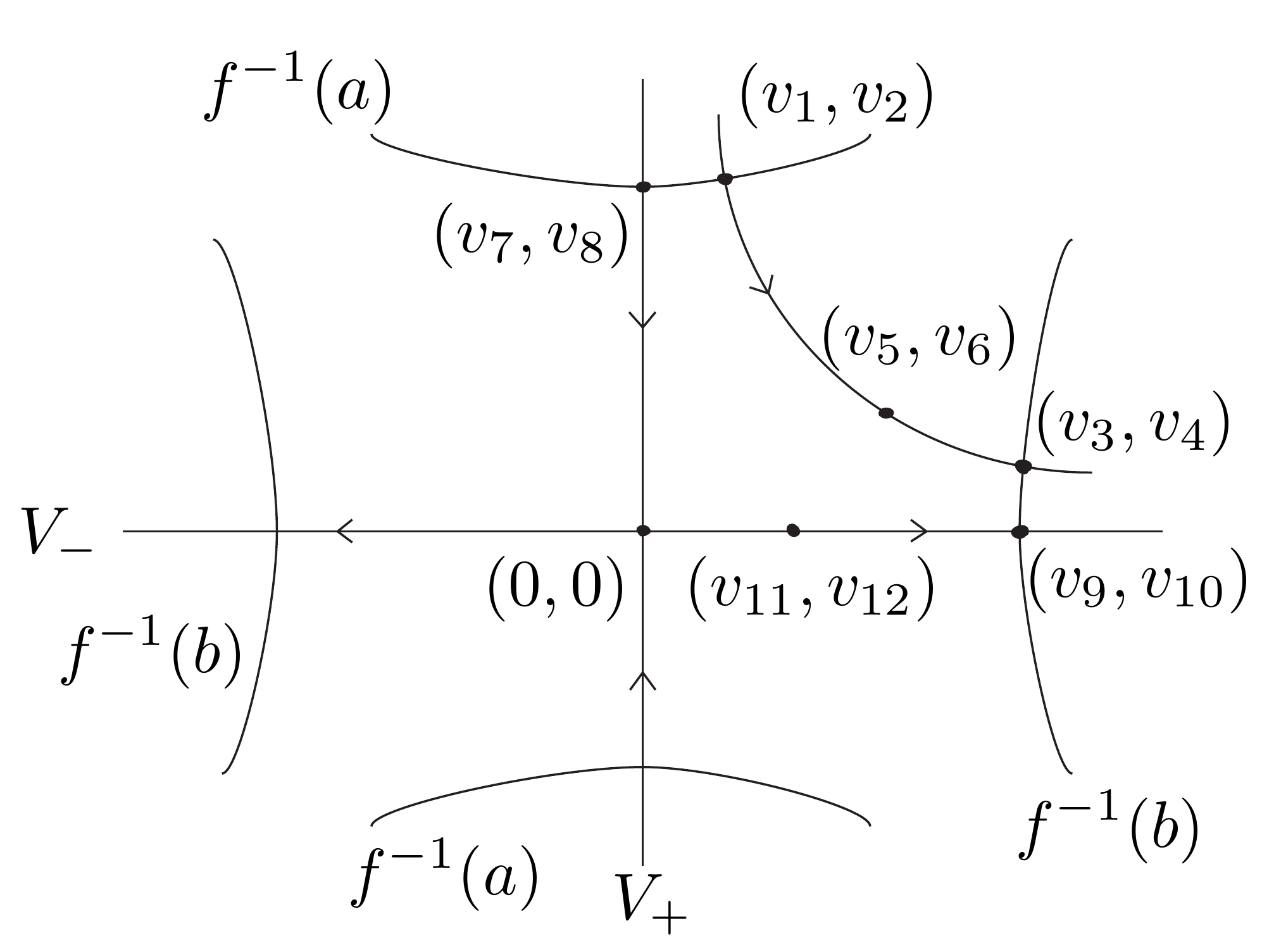} \caption{Standard Model}
\label{model_figure}
\end{figure}

%--------------------------------------------------------------------------------------------------------------------
%--------------------------------------------------------------------------------------------------------------------
\section{Main Results}\label{section_main_results}
All results are in CF case (see Definition \ref{CF_pair}) in this
paper. Theorem \ref{point_compactness} and Proposition
\ref{point_flow_limit} hold even in a more general setting. The
assumption of transversality (see Definition \ref{transverality}),
local triviality of the metric (see Definition
\ref{locally_trivial_metric}) or the lower bound of the Morse
function is needed for some results.

%--------------------------------------------------------------------------------------------------------------------
\subsection{Compactness}
Theorem \ref{flow_compactness}  shows that the closure of the space
of unbroken flow lines is compact and is contained in the (maybe
broken) generalized flow lines (see Definition
\ref{generalized_flow_line}). Theorem \ref{point_compactness}
considers the set consisting of points on the generalized flow lines
with a fixed head and tail. They are essential for proving the
compactness of $\overline{\mathcal{M}(p,q)}$,
$\overline{\mathcal{D}(p)}$ and $\overline{\mathcal{W}(p,q)}$ later.

\begin{theorem}[Compactness of Flows]\label{flow_compactness}
Suppose $(M,f)$ is a CF pair. Suppose $p$ and $q$ are two distinct
critical points and $\{ \gamma_{n} \}_{n=1}^{\infty}$ are flow lines
such that $\gamma_{n}(-\infty)=p$ and $\gamma_{n}(+\infty)=q$. Then
there exist finite many distinct critical points $r_{i}$ $(i=0,
\cdots, l+1)$ and flow lines $\hat{\gamma}_{i}$ $(i=0, \cdots, l)$
such that $\hat{\gamma}_{i}(-\infty) = r_{i}$,
$\hat{\gamma}_{i}(+\infty) = r_{i+1}$, $r_{0} = p$ and $r_{l+1} =
q$. There exist a subsequence $\{ \gamma_{n_{k}} \} \subseteq \{
\gamma_{n} \}$ and time $s_{n_{k}}^{0} < \cdots < s_{n_{k}}^{l}$
such that $\displaystyle \lim_{k \rightarrow \infty}
\gamma_{n_{k}}(s_{n_{k}}^{i}) = \hat{\gamma}_{i}(0)$.
\end{theorem}

\begin{remark}\label{compactness_remark_1}
As pointed out by the referees, the papers
\cite{abbondandolo_majer1} and \cite{abbondandolo_majer2} prove
results similar to Theorem \ref{flow_compactness}. The proof of
\cite{abbondandolo_majer1} relies on the study of differential
operators on vector fields, which is a very different approach from
that of this paper. However, the proof of \cite[prop. 2.4, 1.17 and
2.2]{abbondandolo_majer2} is essentially the same as that of Theorem
\ref{flow_compactness}. Thus, theoretically, it's unnecessary to
include the proof here. Nevertheless, for the sake of completeness,
we still keep it.
\end{remark}

\begin{theorem}[Compactness of Points]\label{point_compactness}
Suppose, for any real numbers $a < b$, $M^{a,b}$ only contains
finite many critical points. Suppose, for any two critical points
$p$ and $q$, the conclusion of Theorem \ref{flow_compactness} holds.
Let $\{ \Gamma_{n} \}_{n=1}^{\infty}$ be a sequence of generalized
flow lines connecting $p$ and $q$. Then we have the following
results.

(1). Suppose $x_{n}$ is on $\Gamma_{n}$. Then there exists a
subsequence $\{ x_{n_{k}} \}_{k=1}^{\infty}$ of $\{ x_{n}
\}_{n=1}^{\infty}$ such that $\displaystyle \lim_{k \rightarrow
\infty} x_{n_{k}}$ exists and is on a generalized flow line
connecting $p$ and $q$.

(2). Suppose $x_{n}^{i}$ are on $\Gamma_{n}$ and $\displaystyle
\lim_{n \rightarrow \infty} x_{n}^{i}$ exist ($i=1, \cdots, k$).
Then these limit points are on a same generalized flow line
connecting $p$ and $q$.

In particular, if $(M,f)$ is a CF pair, then the above (1) and (2)
hold.
\end{theorem}

\begin{remark}
Essentially, the compactness of points follows from the compactness
of flows. For a more precise description, see Proposition
\ref{point_flow_limit}.
\end{remark}

%--------------------------------------------------------------------------------------------------------------------
\subsection{Manifold Structures}
We consider the manifold structures of the compactified spaces of
$\mathcal{M}(p,q)$, $\mathcal{D}(p)$ and $\mathcal{W}(p,q)$ (see
Definitions \ref{descending_ascending_manifold} and
\ref{moduli_space}).

First, we introduce some notation. Suppose $I_{1} = (p, r_{1},
\cdots, r_{s})$ and $I_{2} = (r_{s+1}, \cdots, r_{k},$ $q)$ are
critical sequences (see Definition \ref{critical_sequence}) and
$r_{s} \succeq r_{s+1}$. Let $(I, s) = (p, r_{1}, \cdots, q)$. It is
not necessarily a critical sequence since $r_{s}$ may equal
$r_{s+1}$. Denote the following product manifold by $\mathcal{W}_{I,
s}$.
\begin{equation}\label{W_I_s}
\mathcal{W}_{I, s} = \mathcal{M}_{I_{1}} \times
\mathcal{W}(r_{s},r_{s+1}) \times \mathcal{M}_{I_{2}}.
\end{equation}

The compactifications are standard. Define (see (\ref{M_I}))
\begin{equation}\label{compactified_space}
\overline{\mathcal{M}(p,q)} = \bigsqcup_{I} \mathcal{M}_{I}, \quad
\overline{\mathcal{D}(p)} = \bigsqcup_{I} \mathcal{D}_{I}, \quad
\overline{\mathcal{W}(p,q)} = \bigsqcup_{(I,s)} \mathcal{W}_{I, s}.
\end{equation}
Here the first disjoint union is over all critical sequences with
head $p$ and tail $q$; the second one is over all critical sequences
with head $p$; and the third one is over all $(I,s) = (p, r_{1},
\cdots, r_{k}, q)$ such that $p \succ r_{1} \succ \cdots \succ r_{s}
\succeq r_{s+1} \succ \cdots \succ r_{k} \succ q$ for all $k$.
Clearly, $\mathcal{M}(p,q) \subseteq \overline{\mathcal{M}(p,q)}$,
$\mathcal{D}(p) \subseteq \overline{\mathcal{D}(p)}$ and
$\mathcal{W}(p,q) \subseteq \overline{\mathcal{W}(p,q)}$. Since
$\mathcal{M}(r_{i}, r_{i+1})$, $\mathcal{D}(r_{k})$ and
$\mathcal{W}(r_{s},r_{s+1})$ are smooth manifolds, so are
$\mathcal{M}_{I}$, $\mathcal{D}_{I}$ and $\mathcal{W}_{I, s}$. An
example of $\overline{\mathcal{D}(p)}$ is illustrated by Figure
\ref{dp_figure} in Subsection \ref{subsection_compactified_space}.

Suppose $\alpha \in \mathcal{M}_{I} \subseteq
\overline{\mathcal{M}(p,q)}$. Then $\alpha = (\gamma_{0}, \cdots,
\gamma_{k})$, where $\gamma_{i} \in \mathcal{M}(r_{i}, r_{i+1})$,
$r_{0}=p$ and $r_{k+1}=q$. By Condition (C), there are only finitely
many critical values in $[f(q), f(p)]$. Suppose the critical values
of $f$ divide $[f(q), f(p)]$ into $l+1$ intervals $[c_{i+1}, c_{i}]$
($i=0, \cdots, l$), where $c_{0} = f(p)$ and $c_{l+1} = f(q)$. For
all $a_{i} \in (c_{i+1}, c_{i})$, they are regular. The union of the
components of $\alpha$ intersects with $f^{-1}(a_{i})$ at exactly
one point $x_{i}(\alpha)$. There is an evaluation map $E:
\overline{\mathcal{M}(p,q)} \longrightarrow \prod_{i=0}^{l}
f^{-1}(a_{i})$ such that
\begin{equation}\label{evalution_m(p,q)}
E(\alpha) = (x_{0}(\alpha), \cdots, x_{l}(\alpha)).
\end{equation}

If $\alpha_{1} \in \prod_{i=0}^{j-1} \mathcal{M}(r_{i}, r_{i+1})
\subseteq \overline{\mathcal{M}(r_{0},r_{j})}$ and $\alpha_{2} \in
\prod_{i=j}^{k} \mathcal{M}(r_{i}, r_{i+1}) \subseteq
\overline{\mathcal{M}(r_{j},r_{k})}$, then $(\alpha_{1}, \alpha_{2})
\in \prod_{i=0}^{k} \mathcal{M}(r_{i}, r_{i+1}) \subseteq
\overline{\mathcal{M}(r_{0},r_{k})}$. This gives a map $i_{(p,r,q)}:
\overline{\mathcal{M}(p,r)} \times \overline{\mathcal{M}(r,q)}
\longrightarrow \overline{\mathcal{M}(p,q)}$. We shall prove the
following theorem (see Definition \ref{critical_sequence}).

\begin{theorem}[Smooth Structure of $\overline{\mathcal{M}(p,q)}$]\label{m(p,q)_manifold}
Let $(M,f)$ be a CF pair satisfying transversality and having a
locally trivial metric. Then, for each pair of critical points
$(p,q)$, there is a smooth structure on
$\overline{\mathcal{M}(p,q)}$ which satisfies the following
properties.

(1). It is a compact manifold with faces whose $k$-stratum is
exactly $\displaystyle \bigsqcup_{|I|=k} \mathcal{M}_{I}$, where the
disjoint union is over all critical sequences $I$ with head $p$ and
tail $q$.

(2). The smooth structure is compatible with that of
$\mathcal{M}_{I}$ in each stratum.

(3). The evaluation map $\displaystyle E:
\overline{\mathcal{M}(p,q)} \longrightarrow \prod_{i=0}^{l}
f^{-1}(a_{i})$ is a smooth embedding, where $E$ is defined by
(\ref{evalution_m(p,q)}).

(4). The smooth structures are compatible with critical pairs, i.e.,
$i_{(p,r,q)}: \overline{\mathcal{M}(p,r)} \times
\overline{\mathcal{M}(r,q)} \longrightarrow
\overline{\mathcal{M}(p,q)}$ is a smooth embedding.
\end{theorem}

We define the evaluation map $e: \overline{\mathcal{D}(p)}
\longrightarrow M$ as follows. The restriction of $e$ on
$\mathcal{D}_{I} = \mathcal{M}_{I} \times \mathcal{D}(r_{k})$ is
just the coordinate projection $\mathcal{M}_{I} \times
\mathcal{D}(r_{k}) \longrightarrow \mathcal{D}(r_{k})$. This defines
the map since $\mathcal{D}(r_{k}) \subseteq M$.

If $\alpha_{1} \in \prod_{i=0}^{j-1} \mathcal{M}(r_{i}, r_{i+1})
\subseteq \overline{\mathcal{M}(r_{0},r_{j})}$ and $(\alpha_{2},x)
\in \prod_{i=j}^{k} \mathcal{M}(r_{i}, r_{i+1}) \times
\mathcal{D}(r_{k}) \subseteq \overline{\mathcal{D}(r_{j})}$, then
$(\alpha_{1}, \alpha_{2}, x) \in \prod_{i=0}^{k} \mathcal{M}(r_{i},
r_{i+1}) \times \mathcal{D}(r_{k}) \subseteq
\overline{\mathcal{D}(r_{0})}$. This gives a map $i_{(p,r)}:
\overline{\mathcal{M}(p,r)} \times \overline{\mathcal{D}(r)}
\longrightarrow \overline{\mathcal{D}(p)}$. We shall prove the
following theorem.

\begin{theorem}[Smooth Structure of $\overline{\mathcal{D}(p)}$]\label{d(p)_manifold}
Under the assumptions of Theorem \ref{m(p,q)_manifold}, suppose $f$
has a lower bound. Then, for each critical point $p$, there is a
smooth structure on $\overline{\mathcal{D}(p)}$ satisfying the
following properties.

(1). It is a compact manifold with faces whose $k$-stratum is
exactly $\displaystyle \bigsqcup_{|I|=k-1} \mathcal{D}_{I}$ where
the disjoint union is over all critical sequences with head $p$.

(2). The smooth structure is compatible with that of
$\mathcal{D}_{I}$ in each stratum.

(3). The evaluation map $e: \overline{\mathcal{D}(p)}
\longrightarrow M$ is smooth, where the restriction of $e$ on
$\mathcal{D}_{I} = \mathcal{M}_{I} \times \mathcal{D}_{r_{k}}$ is
the coordinate projection onto $ \mathcal{D}_{r_{k}} \subseteq M$.

(4). The smooth structures are compatible with critical pairs, i.e.,
$i_{(p,r)}: \overline{\mathcal{M}(p,r)} \times
\overline{\mathcal{D}(r)} \longrightarrow \overline{\mathcal{D}(p)}$
is a smooth embedding, where the smooth structure of
$\overline{\mathcal{M}(p,r)}$ is defined in Theorem
\ref{m(p,q)_manifold}.
\end{theorem}

\begin{remark}
It's easy to see that Theorem \ref{d(p)_manifold} will not be true
if we don't assume that $f$ is lower bounded.
\end{remark}

Similarly, we also have the following theorem about
$\mathcal{W}(p,q)$. The maps $e$, $i_{(p,r,q)}^{1}$ and
$i_{(p,r,q)}^{2}$ are defined in Subsection 5.1.

\begin{theorem}[Smooth Structure of $\overline{\mathcal{W}(p,q)}$]\label{w(p,q)_manifold}
Under the assumptions of Theorem \ref{m(p,q)_manifold}, for each
pair of critical points $(p,q)$, there is a smooth structure on
$\overline{\mathcal{W}(p,q)}$ satisfying the following properties.

(1). It is a compact manifold with faces whose $k$-stratum is
exactly $\displaystyle \bigsqcup_{(I, s)} \mathcal{W}_{I, s}$. Here
$(I,s) = (p, r_{1}, \cdots, r_{k}, q)$ such that $p \succ r_{1}
\succ \cdots \succ r_{s} \succeq r_{s+1} \succ \cdots \succ r_{k}
\succ q$. The disjoint union is over all $(I,s)$ which contain $k+2$
components.

(2). The smooth structure is compatible with that of
$\mathcal{W}_{I, s}$ in each stratum.

(3). The evaluation map $e: \overline{\mathcal{W}(p,q)}
\longrightarrow M$ is smooth.

(4). The smooth structures are compatible with critical pairs, i.e.,
$i_{(p,r,q)}^{1}: \overline{\mathcal{W}(p,r)} \times
\overline{\mathcal{M}(r,q)} \longrightarrow
\overline{\mathcal{W}(p,q)}$ and $i_{(p,r,q)}^{2}:
\overline{\mathcal{M}(p,r)} \times \overline{\mathcal{W}(r,q)}
\longrightarrow \overline{\mathcal{W}(p,q)}$ are smooth embeddings.

Here the smooth structure of $\overline{\mathcal{M}(*,*)}$ is
defined in Theorem \ref{m(p,q)_manifold}.
\end{theorem}

The above theorems are under the assumption that the metric is
locally trivial. If this assumption is dropped, can
$\overline{\mathcal{M}(p,q)}$, $\overline{\mathcal{D}(p)}$ and
$\overline{\mathcal{W}(p,q)}$ still be equipped with smooth
structures with the desired stratifications? To the best of my
knowledge, this question is still open even when $M$ is finite
dimensional. However, even if the answer is positive, there is still
some difference between the case of a locally trivial metric and
that of a general metric.

Consider $E: \overline{\mathcal{M}(p,q)} \longrightarrow
\prod_{i=0}^{l} f^{-1}(a_{i})$ in (\ref{evalution_m(p,q)}). By (3)
of Theorem \ref{m(p,q)_manifold}, the image of $E$, $\textrm{Im}(E)
\subseteq \prod_{i=0}^{l} f^{-1}(a_{i})$ is a smooth ($C^{\infty}$)
embedded submanifold of $\prod_{i=0}^{l} f^{-1}(a_{i})$ when the
metric is locally trivial. For a general metric, we have the
following counterexample.

\begin{example}[Not $C^{1}$]\label{not_c1}
Let $CP^{2}$ the complex projective plane. Then there exist a metric
and a Morse function $f$ on $CP^{2}$, where $f$ has three critical
points $p$, $q$ and $r$ such that $\textrm{ind}(p) = 4$,
$\textrm{ind}(q) = 0$ and $\textrm{ind}(r) = 2$.
$\overline{\mathcal{M}(p,q)} = \mathcal{M}(p,q) \sqcup
(\mathcal{M}(p,r) \times \mathcal{M}(r,q))$. And
$E(\overline{\mathcal{M}(p,q)})$ is NOT a $C^{1}$ embedded
submanifold with boundary $E(\mathcal{M}(p,r) \times
\mathcal{M}(r,q))$ of $\prod_{i=0}^{1} f^{-1}(a_{i})$ (see
(\ref{evalution_m(p,q)})). In other words, it's impossible to give
$\textrm{Im}(E)$ a $C^{1}$ structure compatible with
$\prod_{i=0}^{1} f^{-1}(a_{i})$.
\end{example}

%--------------------------------------------------------------------------------------------------------------------
\subsection{Orientation}
We consider the orientation of $\overline{\mathcal{M}(p,q)}$,
$\overline{\mathcal{D}(p)}$ and $\overline{\mathcal{W}(p,q)}$. Since
$\textrm{ind}(p) < + \infty$, we can assign $\mathcal{D}(p)$ an
orientation arbitrarily. The orientations of $\mathcal{D}(p)$ for
all $p$ determine the orientations of $\mathcal{M}(p,q)$ and
$\mathcal{W}(p,q)$ and then those of their compactified manifolds
for all pairs $(p,q)$. Now we consider the $1$-stratum (see
Definition \ref{k_stratum}) of $\overline{\mathcal{M}(p,q)}$, i.e.,
$\partial^{1} \overline{\mathcal{M}(p,q)} = \bigsqcup_{r}
\mathcal{M}(p,r) \times \mathcal{M}(r,q)$. The orientation of
$\overline{\mathcal{M}(p,q)}$ gives $\partial^{1}
\overline{\mathcal{M}(p,q)}$ a boundary orientation. On the other
hand, the orientations of $\mathcal{M}(p,r)$ and $\mathcal{M}(r,q)$
give $\mathcal{M}(p,r) \times \mathcal{M}(r,q)$ a product
orientation. We shall consider the relation between the two
orientations of $\partial^{1} \overline{\mathcal{M}(p,q)}$.
Similarly, $\overline{\mathcal{D}(p)}$ and
$\overline{\mathcal{W}(p,q)}$ also have such orientation issues. The
definition of the above orientations is given in Subsection
\ref{subsection_definition_orientation}. We have the following
orientation formulas.

\begin{theorem}[Orientation Formulas]\label{orientation}
Under the assumption of Theorem \ref{m(p,q)_manifold}, as oriented
manifolds, we have

(1). $\displaystyle \partial^{1} \overline{\mathcal{M}(p,q)} =
\bigsqcup_{p \succ r \succ q} (-1)^{\textrm{ind}(p) -
\textrm{ind}(r)} \mathcal{M}(p,r) \times \mathcal{M}(r,q)$;

(2). $\displaystyle \partial^{1} \overline{\mathcal{D}(p)} =
\bigsqcup_{p \succ r} \mathcal{M}(p,r) \times \mathcal{D}(r)$, where
$f$ is lower bounded;

(3). $\displaystyle \partial^{1} \overline{\mathcal{W}(p,q)} =
\bigsqcup_{p \succeq r \succ q} (-1)^{\textrm{ind}(p) -
\textrm{ind}(r) + 1} \mathcal{W}(p,r) \times \mathcal{M}(r,q) \sqcup
\bigsqcup_{p \succ r \succeq q} \mathcal{M}(p,r) \times
\mathcal{W}(r,q)$.

In the above, $\partial^{1} \Box$ are equipped with boundary
orientations, $\Box \times \Box$ are equipped with product
orientations.
\end{theorem}

\begin{remark}\label{orientation_remark}
The papers \cite[lem.\ 3.4]{austin_braam} and \cite[sec.\ 2.14 and
2.15]{latour} announce formulas similar to (1) and (2) of Theorem
\ref{orientation} in finite dimensional case (\cite{austin_braam}
even does the Morse-Bott case). Our method to define orientations is
different from theirs. Thus our formulas are different from theirs.
By our definition of orientations, there is no sign in (2) of
Theorem \ref{orientation}.
\end{remark}

%--------------------------------------------------------------------------------------------------------------------
\subsection{CW Structure}\label{subsection_cw_structure}
Finally, we point out that the compatification of $\mathcal{D}(p)$
results in a bona fide smooth CW decomposition of $M$.

Clearly, $\mathcal{D}(p)$ is diffeomorphic to an open disk of
dimension $\textrm{ind}(p)$, and $\mathcal{D}(p) \cap \mathcal{D}(q)
= \emptyset$ when $p \neq q$. Recall the evaluation map $e:
\overline{\mathcal{D}(p)} \longrightarrow M$ and that
$\overline{\mathcal{D}(p)} = \bigsqcup_{I} \mathcal{M}_{I} \times
\mathcal{D}(r_{k})$ (see Theorem \ref{d(p)_manifold}). The
restriction of $e$ to $\mathcal{M}_{I} \times \mathcal{D}(r_{k})$ is
just the coordinate projection onto $\mathcal{D}(r_{k})$. Thus
$e|_{\mathcal{D}(p)}$ is the identity map, and $e(\partial
\overline{\mathcal{D}(p)})$ consists of finite number of
$\mathcal{D}(q)$ such that $\textrm{ind}(q) < \textrm{ind}(p)$. Thus
if $\overline{\mathcal{D}(p)}$ is homeomorphic to a closed disk for
all $p$, then, $\forall a \in R$, $K^{a} = \bigsqcup_{f(p) \leq a}
\mathcal{D}(p)$ is a finite CW complex with characteristic maps $e$.
We shall prove the following theorems.

\begin{theorem}[Topology of $ \overline{\mathcal{D}(p)}$]\label{disk}
Under the assumption of Theorem \ref{d(p)_manifold}, there is a
homeomorphism $\Psi: (D^{\textrm{ind}(p)}, S^{\textrm{ind}(p)-1})
\longrightarrow (\overline{\mathcal{D}(p)},
\partial \overline{\mathcal{D}(p)})$, where $D^{\textrm{ind}(p)}$ is the
$\textrm{ind}(p)$ dimensional closed disk and $S^{\textrm{ind}(p)-1}
=
\partial D^{\textrm{ind}(p)}$.
\end{theorem}

For the definition of simple homotopy equivalence and elementary
expansion, see \cite[p. 14-15]{m.cohen}

\begin{theorem}[CW Structure]\label{CW}
Under the assumption of Theorem \ref{d(p)_manifold}, let $a$ be a
regular value of $f$. Then $K^{a} = \bigsqcup_{f(p) \leq a}
\mathcal{D}(p)$ is a finite CW complex with characteristic maps $e:
\overline{\mathcal{D}(p)} \longrightarrow K^{a}$, where $e$ is
defined in (3) of Theorem \ref{d(p)_manifold}. In particular, if $f$
is proper, then the inclusion $K^{a} \hookrightarrow M^{a}$ is a
simple homotopy equivalence. In fact, in this special case, there is
a CW decomposition of $M^{a}$ such that $K^{a}$ expands to $ M^{a}$
by elementary expansions.
\end{theorem}

The following theorem explicitly computes the boundary operator of
the CW chain complex $C_{*}(K^{a})$ associated with the CW
structure.

\begin{theorem}[Boundary Operator]\label{boundary_operator}
Let $K^{a}$ be the CW complex in Theorem \ref{CW} (we do NOT assume
$f$ is proper). Let $C_{*}(K^{a})$ be the associated CW chain
complex and $[\overline{\mathcal{D}(p)}]$ be the base element
represented by $\overline{\mathcal{D}(p)}$ in $C_{*}(K^{a})$. Then
\[
  \partial [\overline{\mathcal{D}(p)}] = \sum_{\textrm{ind}(q) = \textrm{ind}(p) -1}
  \# \mathcal{M}(p,q) [\overline{\mathcal{D}(q)}],
\]
where $\# \mathcal{M}(p,q)$ is the sum of the orientations $\pm 1$
of all points in $\mathcal{M}(p,q)$ defined in Theorem
\ref{orientation}.
\end{theorem}

\begin{remark}\label{morse_homology}
Theorem \ref{boundary_operator} shows that the boundary operator of
$C_{*}(K^{a})$ coincides with that of the Thom-Smale complex in
Morse homology when $M$ is compact. This shows Morse homology arises
from  a cellular chain complex. However, unlike the assumption of
Theorem \ref{boundary_operator}, Morse homology does not require the
local triviality of metrics. For Morse homology, see \cite[cor.\
7.3]{milnor2}, \cite{banyaga_hurtubise} and \cite{schwarz}. For some
of its generalizations to Hilbert manifolds, see \cite{peng},
\cite{abbondandolo_majer1} and \cite{abbondandolo_majer2}.
\end{remark}

%--------------------------------------------------------------------------------------------------------------------
%--------------------------------------------------------------------------------------------------------------------
\section{Compactness}

%--------------------------------------------------------------------------------------------------------------------
\subsection{Proof of Theorem \ref{flow_compactness}}
In order to prove Theorem \ref{flow_compactness}, we need the
classical Grobman-Hartman Theorem in Banach spaces.

Suppose $U_{i}$ ($i=1,2$) are two open subsets in two Banach spaces
$E_{i}$. Let $X_{i}$ be a smooth vector field on $U_{i}$. Let
$\phi^{i}_{t}(x_{i})$ be the associated flow on $U_{i}$ with initial
value $x_{i}$. We say $\phi^{i}_{t}$ ($i=1,2$) are topologically
conjugate if there exists a homeomorphism $h: U_{1} \longrightarrow
U_{2}$ such that $h(\phi^{1}_{t}(x_{1})) = \phi^{2}_{t}(h(x_{1}))$
(see \cite[p. 26]{palis_de}). The Grobman-Hartman Theorem states
that, if $p$ is a hyperbolic singularity of $X$ on an open subset
$U$ of a Banach space $E$, then the flow generated by $X$ is locally
topologically conjugate to that generated by the linear vector field
$\nabla X(p) v$ near $0$ on $T_{p} U$ (see \cite[sec.\ 4]{pugh},
\cite[sec.\ 5]{palis} and \cite[thm.\ 4.10, p. 66]{palis_de}.
Although the statements in \cite{palis} and \cite{palis_de} are only
up to topological equivalence, they actually construct the
conjugate.)

In our case, $\nabla^{2}f(p)$ splits $T_{p} M$ into two subspace
$T_{p} M = V_{-} \times V_{+}$, where $\{0\} \times V_{+}$ ($V_{-}
\times \{0\}$) is the positive (negative) spectrum space of
$\nabla^{2}f(p)$. Thus the flow of $-\nabla f$ is topologically
conjugate to the flow of $(- \nabla^{2}f(p) v_{1}, - \nabla^{2}f(p)
v_{2})$ on $T_{p} M$. Furthermore, $- \nabla^{2}f(p)$ is symmetric
and negative (positive) definite in $\{0\} \times V_{+}$ ($V_{-}
\times \{0\}$), thus $- \nabla^{2}f(p) v_{i}$ is transversal to the
unit sphere in $V_{\pm}$. By the method of the proof of \cite[prop.\
2.15, p. 52]{palis_de}, we have the flow of $(- \nabla^{2}f(p)
v_{1}, - \nabla^{2}f(p) v_{2})$ is topologically conjugate to the
flow of $(v_{1}, - v_{2})$. Thus we get the flowing lemma (compare
(\ref{localization_dymamics})).

\begin{lemma}\label{conjugate}
The flow generated by $-\nabla f$ near a critical point $p$ is
locally topologically conjugate to the flow generated by $(v_{1}, -
v_{2})$ near $0$ on $T_{p} M = V_{-} \times V_{+}$.
\end{lemma}

\begin{lemma}\label{1_singularity_compactness}
Suppose $\{ \gamma_{n}(t) \}_{n=1}^{\infty}$ and
$\hat{\gamma}_{1}(t)$ are flow lines such that $\displaystyle
\lim_{n \rightarrow \infty} \gamma_{n}(0) = \hat{\gamma}_{1}(0)$,
$\hat{\gamma}_{1}(+\infty) = p$ with $\textrm{ind}(p) < +\infty$,
and, for all $n$, $\gamma_{n}(+\infty) \neq p$. Then there exist a
subsequence $\{ \gamma_{n_{k}} \} \subseteq \{ \gamma_{n} \}$, time
$s_{n_{k}} > 0$ and a nonconstant flow line $\hat{\gamma}_{2}(t)$
such that $\hat{\gamma}_{2}(-\infty) = p$ and $\displaystyle \lim_{k
\rightarrow \infty} \gamma_{n_{k}}(s_{n_{k}}) =
\hat{\gamma}_{2}(0)$.
\end{lemma}
\begin{proof}
By Lemma \ref{conjugate}, there exist a neighborhood $U_{2}$ of $p$
in $M$, a neighborhood $U_{1}$ of $0$ in $T_{p}M$ and a
homeomorphism $h: U_{1} \longrightarrow U_{2}$ such that $h(0)=p$
and $h$ conjugates between the flow generated by $(v_{1}, -v_{2})$
in $U_{1}$ and the flow generated by $-\nabla f$ in $U_{2}$ (see
Figure \ref{topological_conjugate_figure}).

\begin{figure}[!htbp]
\centering
\includegraphics[scale=0.25]{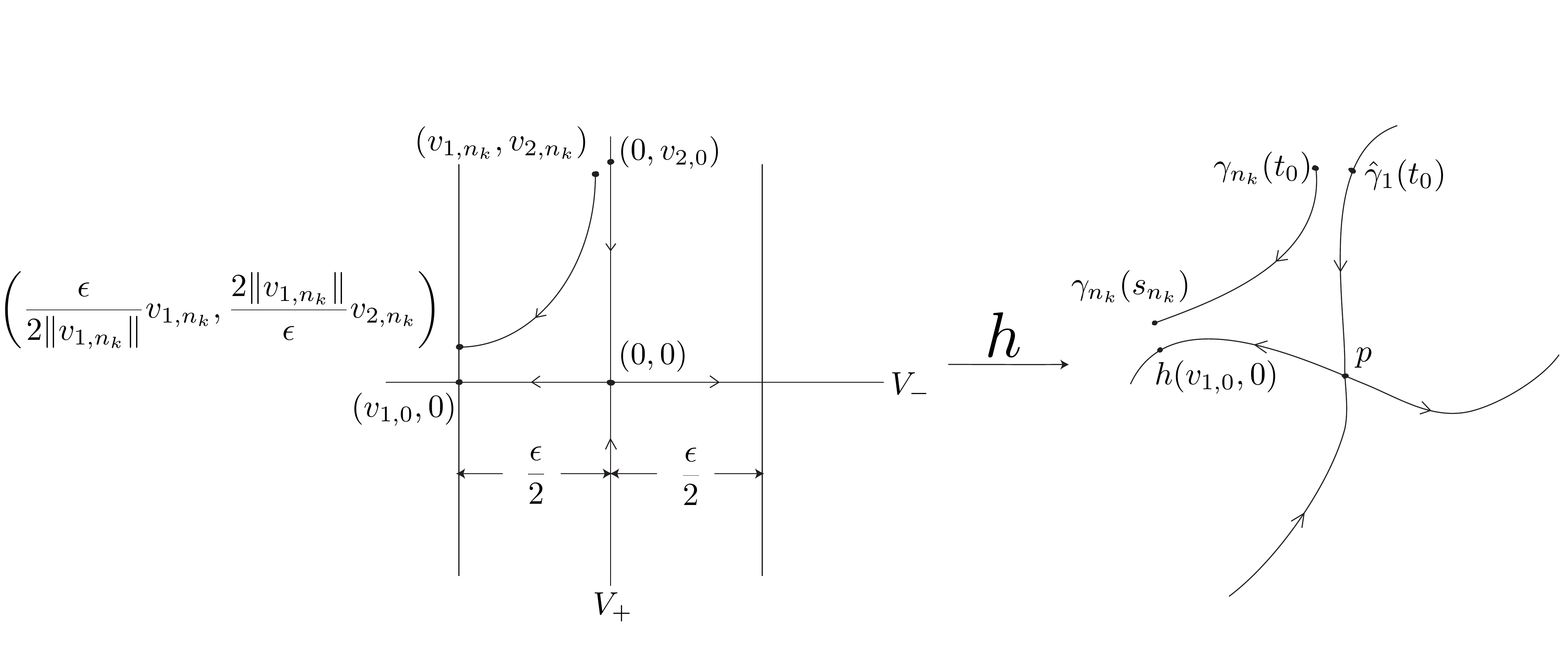} \caption{Topological Conjugate}
\label{topological_conjugate_figure}
\end{figure}

Choosing an open subset if necessary, we may assume $U_{1} =
D_{1}(\epsilon) \times D_{2}(\epsilon)$ for some $\epsilon$, where
$D_{1}(\epsilon) = \{ v_{1} \in V_{-} \mid \| v_{1} \| < \epsilon
\}$ and $D_{2}(\epsilon) = \{ v_{2} \in V_{+} \mid \| v_{2} \| <
\epsilon \}$. In $U_{1}$, $(V_{-} \times \{0\}) \cap U_{1}$ is the
unstable submanifold, $(\{0\} \times V_{+}) \cap U_{1}$ is the
stable submanifold. Thus $h(V_{-} \times \{0\})$ and $h(\{0\} \times
V_{+})$ are locally unstable and locally stable submanifolds
respectively in $U_{2}$.

Since $\hat{\gamma}_{1}(+\infty) = p$, $\exists t_{0}$ such that
$\forall t \geq t_{0}$, $\hat{\gamma}_{1}(t) \in h(\{0\} \times
V_{+})$. Suppose $h^{-1}(\hat{\gamma}_{1}(t_{0})) = (0, v_{2,0})$.
Since $\gamma_{n}(0) \rightarrow \hat{\gamma}_{1}(t_{0})$, we have
$\gamma_{n}(t_{0}) \in U_{2}$, $h^{-1}(\gamma_{n}(t_{0}))=(v_{1,n},
v_{2,n})$ and $\| v_{1,n} \| < \frac{\epsilon}{2}$ when $n$ is large
enough. Since $\gamma_{n}(+\infty) \neq p$, we have $v_{1,n} \neq
0$. As a result, in $U_{1}$, the flow line passing through
$(v_{1,n}, v_{2,n})$ intersects $S_{1} (\frac{\epsilon}{2}) \times
D_{2}(\epsilon)$ at $\displaystyle \left( \frac{\epsilon}{2\|
v_{1,n} \|} v_{1,n}, \frac{2\| v_{1,n} \|}{\epsilon} v_{2,n}
\right)$, where $S_{1} (\frac{\epsilon}{2}) = \{ v_{1} \in V_{-}
\mid \| v_{1} \| = \frac{\epsilon}{2} \}$. When $n \rightarrow
\infty$, we have $(v_{1,n}, v_{2,n}) \rightarrow (0, v_{2,0})$. Thus
$\displaystyle \frac{2\| v_{1,n} \|}{\epsilon} v_{2,n} \rightarrow
0$. Since it is a $\textrm{ind}(p)-1$ dimensional sphere and
$\textrm{ind}(p) < +\infty$, we have $S_{1}\left( \frac{\epsilon}{2}
\right)$ is compact. So there exists a subsequence $\displaystyle
\left\{ \frac{\epsilon v_{1,n_{k}}}{2\| v_{1,n_{k}} \|} \right\}$ of
$\displaystyle \left \{ \frac{\epsilon v_{1,n}}{2\| v_{1,n} \|}
\right \}$ such that $\displaystyle \lim_{k \rightarrow \infty}
\frac{\epsilon}{2\| v_{1,n_{k}} \|} v_{1,n_{k}} = v_{1,0}$. Clearly,
there exists $s_{n_{k}} > 0$ such that
\[
\gamma_{n_{k}}(s_{n_{k}}) = h \left( \frac{\epsilon}{2\| v_{1,n_{k}}
\|} v_{1,n_{k}}, \frac{2\| v_{1,n_{k}} \|}{\epsilon} v_{2,n_{k}}
\right).
\]
Thus
\[
  \lim_{k \rightarrow \infty} \gamma_{n_{k}} (s_{n_{k}}) = h (v_{1,0}, 0).
\]

Denote the flow line with initial value $h (v_{1,0}, 0)$ by
$\hat{\gamma}_{2}(t)$. Then $\displaystyle \lim_{k \rightarrow
\infty} \gamma_{n_{k}}(s_{n_{k}})$ $= \hat{\gamma}_{2}(0)$. Since
$h^{-1}(\hat{\gamma}_{2}(0)) = (v_{1,0}, 0) \in V_{-} \times \{0\}$,
we know that $h^{-1}(\hat{\gamma}_{2}(-\infty))= (0,0)$ or
$\hat{\gamma}_{2}(-\infty) = p$. Since $\hat{\gamma}_{2}(0) \neq p$,
$\hat{\gamma}_{2}$ is nonconstant.
\end{proof}

\begin{proof}[Proof of Theorem \ref{flow_compactness}]
Let $a$ be a regular value such that $a < f(p)$ and there is no
critical value in $(a, f(p))$. Let $S_{p}^{-} = \mathcal{D}(p) \cap
f^{-1}(q)$. Then $S_{p}^{-}$ is a sphere with dimension
$\textrm{ind}(p) < +\infty$, and it is compact. Suppose
$\gamma_{n}(s_{n}^{0}) \in S_{p}^{-}$. Then there exists a
subsequence of $\{ \gamma_{n}(s_{n}^{0}) \}$, we may still denote it
by $\{ \gamma_{n}(s_{n}^{0}) \}$, which converges. Suppose
$\displaystyle \lim_{n \rightarrow \infty} \gamma_{n}(s_{n}^{0}) =
x_{0}$. Then $x_{0} \in S_{p}^{-}$. Denote the flow line with
initial value $x_{0}$ by $\hat{\gamma}_{0}$. Then $\hat{\gamma}_{0}
(-\infty) = p$ because $\hat{\gamma}_{0} (0) = x_{0} \in S_{p}^{-}
\subseteq \mathcal{D}(p)$. Since $\gamma_{n}(+\infty)=q$, we have,
for all $t$, $f(\gamma_{n}(s_{n}^{0}+t)) \geq f(q)$. Thus, for all
$t$,
\[
  f(\hat{\gamma}_{0}(t)) = \lim_{n \rightarrow \infty} f(
  \gamma_{n}(s_{n}^{0}+t) ) \geq f(q),
\]
i.e., $f(\hat{\gamma}_{0}(t))$ has a lower bound $f(q)$. By Theorem
\ref{bounded}, $\displaystyle \lim_{t \rightarrow +\infty}
\hat{\gamma}_{0}(t)$ exists and $\hat{\gamma}_{0}(+\infty) = r_{1}$
is a critical point in $M^{f(q),f(p)}$. Clearly, $\hat{\gamma}_{0}$
is nonconstant. Thus $r_{1} \neq p$. There are exactly the following
two cases.

Case (1): $r_{1} = q$. In this case, the proof is finished.

Case (2): $r_{1} \neq q$. Since $\gamma_{n}(+\infty) = q \neq r_{1}$
and $\textrm{ind}(r_{1}) < +\infty$, by Lemma
\ref{1_singularity_compactness}, there exists a nonconstant flow
line $\hat{\gamma}_{1}$ such that $\hat{\gamma}_{1}(-\infty) =
r_{1}$. Furthermore, there exists a subsequence of $\{ \gamma_{n}
\}$, which we still denote by $\{ \gamma_{n} \}$, and time
$s_{n}^{1} > s_{n}^{0}$ such that $\displaystyle \lim_{n \rightarrow
\infty} \gamma_{n}(s_{n}^{1}) = \hat{\gamma}_{1}(0)$. Similar to the
case of $\hat{\gamma}_{0}$, we have $\displaystyle \lim_{t
\rightarrow +\infty} \hat{\gamma}_{1}(t)$ exists and
$\hat{\gamma}_{1}(+\infty) = r_{2}$ is also a critical point in
$M^{f(q),f(p)}$. Since $\hat{\gamma}_{1}$ is nonconstant, $p$,
$r_{1}$ and $r_{2}$ are distinct. If $r_{2} = q$, the proof is
finished. Otherwise, repeat the argument of Case (2).

By Theorem \ref{finite_number}, there are only finitely many
critical points in $M^{f(q),f(p)}$, the process of the above
argument terminates in finitely many steps.
\end{proof}
%--------------------------------------------------------------------------------------------------------------------
\subsection{Proof of Theorem \ref{point_compactness}}
We first give two results needed for the proof of Theorem
\ref{point_compactness}.

\begin{lemma}\label{function_limit}
Suppose $\{ \gamma_{n} \}_{n=1}^{\infty}$ and $\hat{\gamma}$ are
flow lines such that $\hat{\gamma}(-\infty) = p$,
$\hat{\gamma}(+\infty) = q$ and $\displaystyle \lim_{n \rightarrow
\infty} \gamma_{n}(s_{n}) = \hat{\gamma}(0)$. If $\displaystyle
\lim_{n \rightarrow \infty} (t_{n} - s_{n}) = +\infty$
($\displaystyle \lim_{n \rightarrow \infty} (t_{n} - s_{n}) =
-\infty$), then $\displaystyle \limsup_{n \rightarrow \infty}
f(\gamma_{n}(t_{n})) \leq f(q)$ ($\displaystyle \liminf_{n
\rightarrow \infty} f(\gamma_{n}(t_{n})) \geq f(p)$).
\end{lemma}
\begin{proof}
It suffices to prove the case $\displaystyle \lim_{n \rightarrow
\infty} (t_{n} - s_{n}) = +\infty$.

Since $\hat{\gamma}(+\infty) = q$, then $\forall \epsilon > 0$,
$\exists T$, such that $\forall t \geq T$, we have
$f(\hat{\gamma}(t)) < f(q) + \epsilon$. By that $\displaystyle
\lim_{n \rightarrow \infty} (t_{n} - s_{n}) = +\infty$, we have
$t_{n} > s_{n} + T$ and $f(\gamma_{n}(t_{n})) <
f(\gamma_{n}(s_{n}+T))$ when $n$ is large enough. Since
$\displaystyle \lim_{n \rightarrow \infty} \gamma_{n}(s_{n}) =
\hat{\gamma}(0)$, we infer
\[
  \lim_{n \rightarrow \infty} f(\gamma_{n}(s_{n}+T)) =
  f(\hat{\gamma}(T)) < f(q) + \epsilon.
\]
Thus
\[
  \limsup_{n \rightarrow \infty} f(\gamma_{n}(t_{n})) \leq \lim_{n \rightarrow \infty}
  f(\gamma_{n}(s_{n}+T)) < f(q) + \epsilon.
\]
Now let $\epsilon \rightarrow 0$. Then we get $\displaystyle
\limsup_{n \rightarrow \infty} f(\gamma_{n}(t_{n})) \leq f(q)$.
\end{proof}

The following proposition requires neither Condition (C) nor finite
indices.

\begin{proposition}\label{point_flow_limit}
Suppose $p$ and $q$ are two critical points, $\{ \gamma_{n}
\}_{n=1}^{\infty}$ are flow lines such that $\gamma_{n}(-\infty)=p$
and $\gamma_{n}(+\infty)=q$, and there exist $s_{n}^{0} < \cdots <
s_{n}^{l}$ such that $\displaystyle \lim_{n \rightarrow \infty}
\gamma_{n}(s_{n}^{i}) = \hat{\gamma}_{i}(0)$. Here
$\hat{\gamma}_{i}$ are flow lines such that
$\hat{\gamma}_{i}(-\infty) = r_{i}$, $\hat{\gamma}_{i}(+\infty) =
r_{i+1}$, and $r_{0}=p$, $r_{l+1}=q$. Then we have the following
convergence result.

(1). If $\displaystyle \lim_{n \rightarrow \infty} (t_{n} -
s_{n}^{i}) = \tau$, $|\tau| < +\infty$, then $\displaystyle \lim_{n
\rightarrow \infty} \gamma_{n}(t_{n}) = \hat{\gamma}_{i}(\tau)$;

(2). If $s_{n}^{i} < t_{n} < s_{n}^{i+1}$, and $\displaystyle
\lim_{n \rightarrow \infty} (t_{n} - s_{n}^{i}) = \lim_{n
\rightarrow \infty} (s_{n}^{i+1} - t_{n}) = + \infty$, then
$\displaystyle \lim_{n \rightarrow \infty} \gamma_{n}(t_{n}) =
r_{i+1}$, where $s_{n}^{-1} = -\infty$ and $s_{n}^{l+1} = +\infty$.
\end{proposition}
\begin{proof}
Case (1) is obvious. We only need to prove Case (2).

We may assume $s_{n}^{i} < t_{n} < s_{n}^{i+1}$ and $i \geq 0$
because the subcase of $i=-1$ will be converted to the subcase of
$i=l$ if $f$ is replaced by $-f$.

We shall prove $\displaystyle \lim_{n \rightarrow \infty}
\gamma_{n}(t_{n}) = r_{i+1}$ by contradiction.

Suppose it doesn't hold, then there exist a subsequence of
$\{\gamma_{n}(t_{n})\}$, which we still denote by
$\{\gamma_{n}(t_{n})\}$, and a neighborhood $U$ of $r_{i+1}$ such
that $\gamma_{n}(t_{n}) \notin U$. Choose an open geodesic disk
$D(r_{i+1}, \epsilon)$ with center $r_{i+1}$ and radius $\epsilon$
such that $\overline{D(r_{i+1}, \epsilon)} \subseteq U$. Since
$r_{i+1}$ is a nondegenerate critical point, by the Taylor
expansion, we may choose $\epsilon$ small enough such that, there
exist constants $C_{1}$ and $C_{2}$, and $0 < C_{1} \leq \| \nabla f
\| \leq C_{2}$ in $\overline{D(r_{i+1}, \epsilon)} - D(r_{i+1},
\frac{\epsilon}{2})$ for a fixed $\epsilon$.

Suppose $\gamma(t)$ is a flow line, $\tau_{1} < \tau_{2}$, such that
$\gamma(\tau_{1}) \in D(r_{i+1}, \frac{\epsilon}{2})$ and
$\gamma(\tau_{2}) \notin \overline{D(r_{i+1}, \epsilon)}$. Thus
there exist $\tau'_{1}, \tau'_{2}$ such that $\tau_{1} < \tau'_{1} <
\tau'_{2} < \tau_{2}$, $\gamma([\tau'_{1}, \tau'_{2}]) \subseteq
\overline{D(r_{i+1}, \epsilon)} - D(r_{i+1}, \frac{\epsilon}{2})$,
$\gamma(\tau'_{1}) \in
\partial D(r_{i+1}, \frac{\epsilon}{2})$ and $\gamma(\tau'_{2}) \in
\partial D(r_{i+1}, \epsilon)$.

Consider the distance $d(\gamma(\tau'_{1}), \gamma(\tau'_{2}))$
between $\gamma(\tau'_{1})$ and $\gamma(\tau'_{2})$. Clearly, $
d(\gamma(\tau'_{1}),$ $ \gamma(\tau'_{2}))$ $\geq
\frac{\epsilon}{2}$. Thus
\begin{eqnarray*}
   &   & \frac{\epsilon}{2} \leq d(\gamma(\tau'_{1}), \gamma(\tau'_{2}))
   \leq \int_{\tau'_{1}}^{\tau'_{2}} \left \| \frac{d}{dt} \gamma (t) \right \| dt \\
   & = & \int_{\tau'_{1}}^{\tau'_{2}} \| \nabla f (\gamma(t)) \| dt
   \leq  \int_{\tau'_{1}}^{\tau'_{2}} C_{2} dt = C_{2} (\tau'_{2} - \tau'_{1}).
\end{eqnarray*}
We have $\tau'_{2} - \tau'_{1} \geq \frac{\epsilon}{2 C_{2}}$. Then
\[
  \int_{\tau'_{1}}^{\tau'_{2}} \| \nabla f \|^{2}
  \geq \int_{\tau'_{1}}^{\tau'_{2}} C_{1}^{2} \geq \frac{C_{1}^{2} \epsilon}{2
  C_{2}}.
\]
Thus we get
\[
f(\gamma (\tau_{1})) - f(\gamma (\tau_{2})) =
\int_{\tau_{1}}^{\tau_{2}} \| \nabla f \|^{2} \geq
\int_{\tau'_{1}}^{\tau'_{2}} \| \nabla f \|^{2} \geq \frac{C_{1}^{2}
\epsilon}{2 C_{2}} > 0.
\]
Denoting $\frac{C_{1}^{2} \epsilon}{2 C_{2}}$ by $K$, we get
\begin{equation}\label{point_flow_limit_1}
f(\gamma (\tau_{1})) - f(\gamma (\tau_{2})) \geq K > 0.
\end{equation}

Since $\hat{\gamma}_{i}(+\infty) = r_{i+1}$, then there exists
$t_{\infty}$ such that $\hat{\gamma}_{i}(t_{\infty}) \in B(r_{i+1},
\frac{\epsilon}{2})$ and $f(\hat{\gamma}_{i}(t_{\infty})) <
f(r_{i+1}) + \frac{K}{2}$. Since $\gamma_{n}(s_{n}^{i}) \rightarrow
\hat{\gamma}_{i}(0)$, we have $\gamma_{n}(s_{n}^{i} + t_{\infty})
\in B(r_{i+1}, \frac{\epsilon}{2})$ and $f(\gamma_{n}(s_{n}^{i} +
t_{\infty})) < f(r_{i+1}) + \frac{K}{2}$ when $n$ is large enough.
Also since $(t_{n} - s_{n}^{i}) \rightarrow + \infty$, we get $t_{n}
> s_{n}^{i} + t_{\infty}$ when $n$ is large enough. Now we can
replace $\gamma(\tau_{1})$ and $\gamma(\tau_{2})$ in
(\ref{point_flow_limit_1}) by $\gamma_{n}(s_{n}^{i} + t_{\infty})$
and $\gamma_{n}(t_{n})$, then $f(\gamma_{n}(s_{n}^{i} + t_{\infty}))
- f(\gamma_{n}(t_{n})) \geq K$. Furthermore,
\[
  f(\gamma_{n}(t_{n})) \leq f(\gamma_{n}(s_{n}^{i} + t_{\infty})) - K < f(r_{i+1}) -
  \frac{K}{2}.
\]
Thus
\begin{equation}
\limsup_{n \rightarrow \infty} f(\gamma_{n}(t_{n})) \leq f(r_{i+1})
- \frac{K}{2} < f(r_{i+1}).
\end{equation}
However, since $\hat{\gamma}_{i+1}(-\infty) = r_{i+1}$, and $(t_{n}
- s_{n}^{i+1}) \rightarrow - \infty$, by Lemma \ref{function_limit},
we have
\begin{equation}
\liminf_{n \rightarrow \infty} f(\gamma_{n}(t_{n})) \geq f(r_{i+1}),
\end{equation}
which is a contradiction.
\end{proof}

\begin{proof}[Proof of Theorem \ref{point_compactness}]
(1). By assumption, there are only finite many critical points in
$M^{f(q),f(p)}$. We can find two critical points $p'$ and $q'$, a
subsequence $\{ x_{n_{k}} \}_{k=1}^{\infty}$ of $\{ x_{n}
\}_{n=1}^{\infty}$ such that $x_{n_{k}}$ is on $\gamma_{n_{k}}$,
$\gamma_{n_{k}}(-\infty) = p'$, $\gamma_{n_{k}}(+\infty) = q'$ and
$\gamma_{n_{k}}$ is a component of $\Gamma_{n_{k}}$. Clearly,  a
generalized flow line connecting $p'$ and $q'$ can be extended to
one connecting $p$ and $q$. If there is a cluster point of $\{
x_{n_{k}} \}_{k=1}^{\infty}$ on a generalized flow line connecting
$p'$ and $q'$, this cluster point is also on one connecting $p$ and
$q$. So we may assume that $x_{n}$ is on $\gamma_{n}$,
$\gamma_{n}(-\infty) = p$ and $\gamma_{n}(+\infty) = q$.

If $p=q$, this is obviously true. Now we assume $p \neq q$. Suppose
$\gamma_{n}(t_{n}) = x_{n}$. Since the conclusion of Theorem
\ref{flow_compactness} holds, choosing a subsequence if necessary,
we can find $s_{n}^{0} < \cdots < s_{n}^{l}$ such that
$\displaystyle \lim_{n \rightarrow \infty} \gamma_{n}(s_{n}^{i}) =
\hat{\gamma}_{i}(0)$, where $\hat{\gamma}_{i}(-\infty) = r_{i}$,
$\hat{\gamma}_{i}(+\infty) = r_{i+1}$, and $r_{0}=p$, $r_{l+1}=q$.

Choosing a subsequence again if necessary, we can find a fixed $i$
such that, for all $n$, we have $t_{n} \in [s_{n}^{i},
s_{n}^{i+1}]$, where $s_{n}^{-1} = -\infty$ and $s_{n}^{l+1} =
+\infty$. In addition, we may assume there are exactly the following
three cases when $n \rightarrow \infty$. By Proposition
\ref{point_flow_limit}, we have:

Case (a): $\displaystyle \lim_{n \rightarrow \infty}(t_{n} -
s_{n}^{i}) = \tau < +\infty$. Then $x_{n}$ converges to a point on
$\hat{\gamma}_{i}$;

Case (b): $\displaystyle \lim_{n \rightarrow \infty}(s_{n}^{i+1} -
t_{n}) = \tau < +\infty$. Then $x_{n}$ converges to a point on
$\hat{\gamma}_{i+1}$;

Case (c): $\displaystyle \lim_{n \rightarrow \infty}(t_{n} -
s_{n}^{i}) = \lim_{n \rightarrow \infty}(s_{n}^{i+1} - t_{n}) =
+\infty$. Then $x_{n}$ converges to $r_{i+1} =
\hat{\gamma}_{i}(+\infty) = \hat{\gamma}_{i+1}(-\infty)$.

This completes the proof of the first result.

(2). Since the limit of $\{ x_{n}^{i} \}$ exists, its subsequences
share the same limit with it. So we only need to check the limit of
a subsequence of $\{ x_{n}^{i} \}$. Since there are only finitely
many critical points in $M^{f(q),f(p)}$, we may argue as in (1):
choosing a subsequence if necessary, we may assume $\Gamma_{n} =
(\gamma_{n,1}, \cdots, \gamma_{n,m})$, $\gamma_{n,j}(-\infty) =
r_{j}$ and $\gamma_{n,j}(+\infty) = r_{j+1}$ are fixed and
independent of $n$. In addition, $\forall i$, there is a fixed $j$
such that for all $n$, $x_{n}^{i}$ is on $\gamma_{n,j}$. If $r_{j} =
r_{j+1}$, then $\gamma_{n,j}$ converges to the constant flow
connecting $r_{j}$ and $r_{j}$. Otherwise, choosing a subsequence
again if necessary, $\{ \gamma_{n,j} \}_{n=1}^{\infty}$ converges to
a generalized flow line connecting $r_{j}$ and $r_{j+1}$. The
combination of the limits of $\{ \gamma_{n,j} \}_{n=1}^{\infty}$ for
$j=1, \cdots, m$ yields a generalized flow line, $\Gamma$,
connecting $p$ and $q$. By an argument similar to that of (1), the
limits of all $\{ x_{n}^{i} \}$ are on $\Gamma$.
\end{proof}
%--------------------------------------------------------------------------------------------------------------------
%--------------------------------------------------------------------------------------------------------------------
\section{Manifold Structure}
%--------------------------------------------------------------------------------------------------------------------
\subsection{Different Viewpoints on Compactified Spaces}\label{subsection_compactified_space}
If $\alpha \in \mathcal{M}_{I} \subseteq
\overline{\mathcal{M}(p,q)}$, then $\alpha = (\gamma_{0}, \cdots,
\gamma_{k})$, where $\gamma_{i} \in \mathcal{M}(r_{i}, r_{i+1})$,
$r_{0}=p$ and $r_{k+1}=q$. Denote the constant flow line passing
through $r_{i}$ by $\beta(r_{i})$. We can identify $\alpha$ with the
generalized flow line $(\beta(r_{0}), \gamma_{0}, \beta(r_{1}),
\cdots, \gamma_{k}, \beta(r_{k+1}))$ connecting $p$ and $q$. Thus we
get
\begin{equation}\label{comapctify_m}
\overline{\mathcal{M}(p,q)}  = \{ \Gamma \mid \Gamma \text{ is a
generalized flow line connecting $p$ and $q$} \}.
\end{equation}
Suppose $(\alpha, x) \in \mathcal{M}_{I} \times \mathcal{D}(r_{k})
\subseteq \overline{\mathcal{D}(p)}$. We can identify $\alpha$ with
a generalized flow line connecting $p$ and $r_{k}$. Adding the flow
line passing through $x$ to the above generalized flow line, we get
a generalized flow line connecting $p$ and $x$. The latter
generalized flow line is uniquely determined by $(\alpha, x)$. Thus
we get
\begin{equation}\label{comapctify_d}
\overline{\mathcal{D}(p)}  =  \{ (\Gamma, x) \mid \Gamma \text{ is a
generalized flow line connecting $p$ and $x$} \}.
\end{equation}
Similarly, we also get
\begin{equation}\label{comapctify_w}
\overline{\mathcal{W}(p,q)}  =  \{ (\Gamma, x) \mid \Gamma \in
\overline{\mathcal{M}(p,q)},\ \text{$x$ is on $\Gamma$} \}.
\end{equation}
From the above viewpoint, $\Gamma \in \mathcal{M}(p,q)$, $(\Gamma,
x) \in \mathcal{D}(p)$ (or $\mathcal{W}(p,q)$) if and only if
$\Gamma$ has no intermediate critical points.

By (\ref{comapctify_d}), we can see that the evaluation map $e$ in
(3) of Theorem \ref{d(p)_manifold} is just defined by $e(\Gamma, x)
= x$. If $\Gamma_{1} \in \overline{\mathcal{M}(p,r)}$ and
$\Gamma_{2} \in \overline{\mathcal{M}(r,q)}$, then the combination
of $\Gamma_{1}$ and $\Gamma_{2}$ gives an element in
$\overline{\mathcal{M}(p,q)}$. If $\Gamma_{1} \in
\overline{\mathcal{M}(p,r)}$ and $(\Gamma_{2}, x) \in
\overline{\mathcal{D}(r)}$, then the combination of $\Gamma_{1}$ and
$\Gamma_{2}$ is a generalized flow line connecting $p$ and $x$.
These precisely define the maps $i_{(p,r,q)}$ in Theorem
\ref{m(p,q)_manifold} and $i_{(p,r)}$ in Theorem
\ref{w(p,q)_manifold}.

Similarly, if $(\Gamma_{1}, x) \in \overline{\mathcal{W}(p,r)}$ and
$\Gamma_{2} \in \overline{\mathcal{M}(r,q)}$, then the combination
of $\Gamma_{1}$ and $\Gamma_{2}$ gives an element in
$\overline{\mathcal{M}(p,q)}$ and $x$ is on it. This defines the map
$i_{(p,r,q)}^{1}$ in (4) of Theorem \ref{w(p,q)_manifold}.
$i_{(p,r,q)}^{2}$ is defined in a similar way. The map $e$ in (3) of
Theorem \ref{w(p,q)_manifold} is defined as $e(\Gamma, x) = x$. The
restriction of $e$ on $\mathcal{W}_{I,s} = \mathcal{M}_{I_{1}}
\times \mathcal{W}(r_{s}, r_{s+1}) \times \mathcal{M}_{I_{2}}$ is
just the coordinate projection onto $\mathcal{W}(r_{s}, r_{s+1})$.

Figure \ref{dp_figure} shows a standard example on a torus $T^{2} =
S^{1} \times S^{1}$. Consider $S^{1}$ as the unit circle on the
complex plane. Define a Morse function on $T^{2}$ by $f(z_{1},
z_{2}) = \textrm{Re}(z_{1}) + \textrm{Re}(z_{2})$. $f$ has $4$
critical points $p$, $r$, $s$ and $q$. Their indices are $2$, $1$,
$1$ and $0$ respectively. Equip $T^{2}$ with the standard metric.
The left part of Figure \ref{dp_figure} shows the flow on $T^{2}$,
where the opposite sides of the square are identified with each
other. The right part is $\overline{\mathcal{D}(p)}$.
$\overline{\mathcal{D}(p)}$ is an octagon. $\mathcal{M}(p,r) \times
\mathcal{D}(r)$ (or $\mathcal{M}(p,s) \times \mathcal{D}(s)$)
consists of open edges containing $r_{i}$ (or $s_{i}$), where
$i=1,2$. $\mathcal{M}(p,q) \times \mathcal{D}(q)$ consists of the
other $4$ open edges. $(\mathcal{M}(p,r) \times \mathcal{M}(r,q)
\times \mathcal{D}(q)) \cup (\mathcal{M}(p,s) \times
\mathcal{M}(s,q) \times \mathcal{D}(q))$ consists of the $8$
vertices. $e$ maps $r_{i}$ (or $s_{i}$) to $r$ (or $s$).

\begin{figure}[!htbp]
\centering
\includegraphics[scale=0.3]{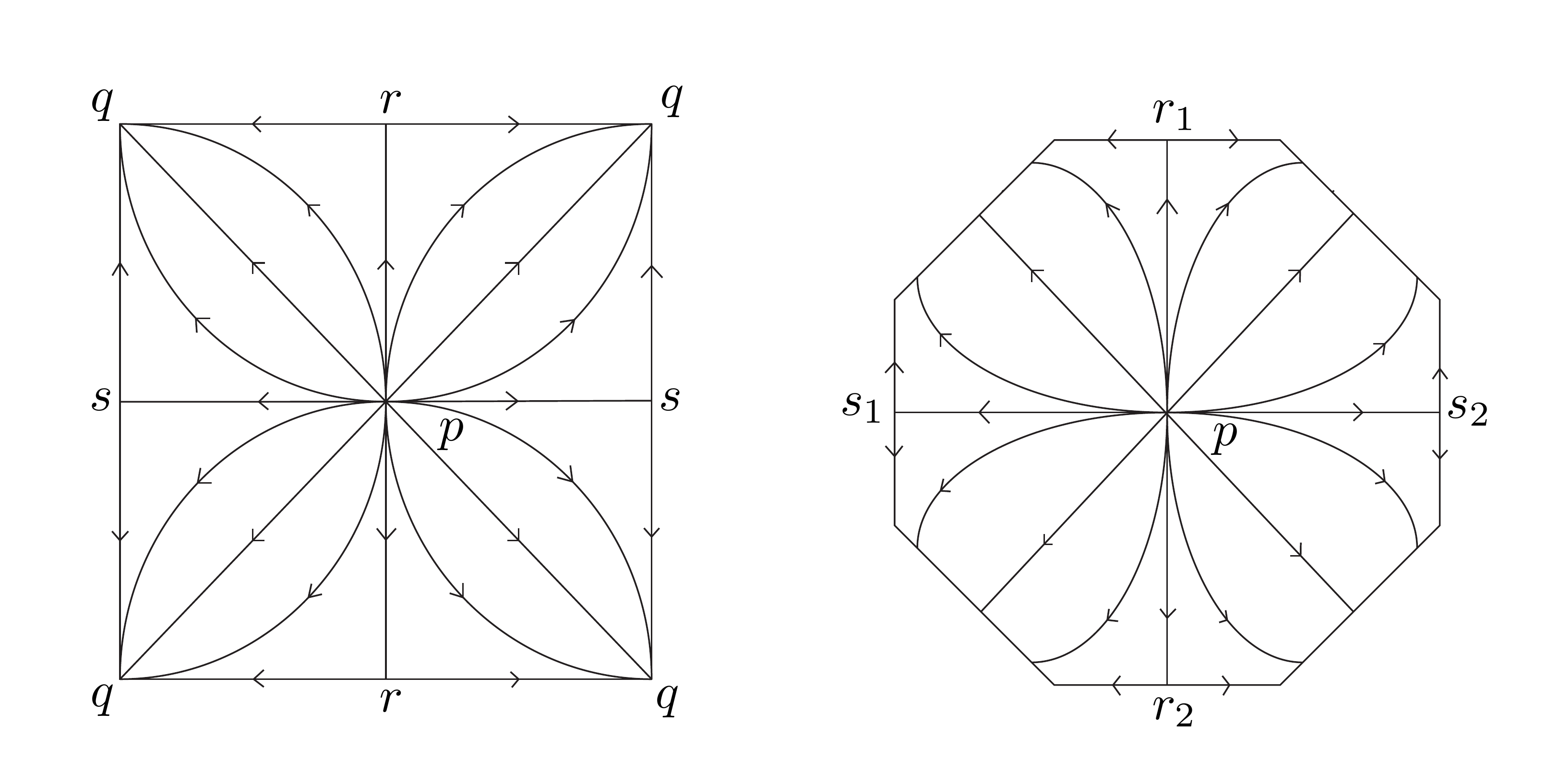} \caption{Compactification of the Descending Manifolds}
\label{dp_figure}
\end{figure}

%--------------------------------------------------------------------------------------------------------------------
\subsection{A Remark on the Literature}\label{remark_manifold_structure}
To the best of my knowledge, in the case of a general metric, there
is no well developed theory of smooth structures on these
compactified spaces. In addition, few papers in the literature study
$\overline{\mathcal{W}(p,q})$.

When $M$ is finite dimensional and the metric is locally trivial,
the papers \cite[prop.\ 2.11]{latour} and \cite[thm.\
1]{burghelea_haller} study $\overline{\mathcal{M}(p,q})$ and
$\overline{\mathcal{D}(p})$. We extend the proof in
\cite{burghelea_haller} to the infinite dimensional CF case.
Actually \cite{burghelea_haller} proves (1) and (2) of Theorem
\ref{m(p,q)_manifold} and (1), (2), (3) of Theorem
\ref{d(p)_manifold} except for the face structures. The book
\cite{burghelea_friedlander_kappeler} contains a proof of the face
structures which is different from the proof in this paper. Although
(3) and (4) of Theorem \ref{m(p,q)_manifold} and (4) of Theorem
\ref{d(p)_manifold} are not pointed out in \cite{burghelea_haller},
they are straightforward results from the geometric construction in
that proof. Despite the infinite dimensions, our main geometric
constructions to prove the smooth structures follow those in
\cite{burghelea_haller} except that Corollary \ref{flow_map} is
elementary in the finite dimensional case. The really big difference
between our proof and \cite{burghelea_haller} is to prove the
compactness of these manifolds. When $M$ is finite dimensional, both
\cite[prop.\ 3]{burghelea_haller} and \cite[prop.\ 2.35]{schwarz}
consider generalized flow lines as maps from an interval to $M$ and
prove compactness by the Arzela-Ascoli Theorem. However, the
Arzela-Ascoli Theorem does not hold when $M$ is infinite
dimensional. Our proof is based on Theorems \ref{flow_compactness}
and \ref{point_compactness}.

Modifying the geometric construction in \cite{burghelea_haller}, we
get a proof of Theorem \ref{w(p,q)_manifold}.

The paper \cite{burghelea_haller} explains its geometric
constructions clearly. However, these geometric constructions are
important for our proofs of other results (see Lemmas
\ref{d(p)_embedding}, \ref{d(p)_normal}, \ref{orientation_3},
\ref{orientation_1}, \ref{orientation_2}, \ref{pull_back_gradient},
\ref{critical_boundary} and \ref{action_d(p)} and Example
\ref{not_c1}). For the completeness of this paper, we explain the
main constructions in \cite{burghelea_haller} again.

%--------------------------------------------------------------------------------------------------------------------
\subsection{Preparation Lemmas}
The following two lemmas, Lemmas \ref{p_manifold} and
\ref{q_manifold} are crucial for our proof. Example \ref{not_c1}
shows that they necessarily depend on the local triviality of the
metric. These two lemmas are announced in \cite[observations 8 and
9]{burghelea_haller}. A proof for them in the finite dimensional
case is given in \cite{burghelea_friedlander_kappeler}. For the
importance of them, we present a proof which follows that in
\cite{burghelea_friedlander_kappeler}.

Figure \ref{model_figure} gives an illustration for the following
argument. Suppose $c$ is a critical value of $f$. The critical
points with function value $c$ are exactly $p_{1}, \cdots, p_{n}$.
Just as (\ref{localization_map}), we have diffeomorphisms $h_{i}:
B_{i}(\epsilon) \longrightarrow U_{i}$ such that
(\ref{localization_function}) and (\ref{localization_dymamics})
hold, where $B_{i}(\epsilon)$ is the open subset of $T_{p_{i}} M$
and $U_{i}$ is the neighborhood of $p_{i}$. Choose $\epsilon$ small
enough such that there is no critical value in $[c-\epsilon,
c+\epsilon]$ other than $c$. Let $M_{c}^{+} = \{ x \in M \mid f(x) =
c + \frac{1}{2} \epsilon \}$ and $M_{c}^{-} = \{ x \in M \mid f(x) =
c - \frac{1}{2} \epsilon \}$. Let
\[
P_{c} = \{ (x^{+}, x^{-}) \in M_{c}^{+} \times M_{c}^{-} \mid
\text{$x^{+}$ and $x^{-}$ are connected by a generalized flow line}
\}.
\]
Clearly, $x^{+}$ and $x^{-}$ are connected by broken generalized
flow lines if and only if $(x^{+}, x^{-}) \in \bigsqcup_{p_{i}}
S_{p_{i}}^{+} \times S_{p_{i}}^{-}$, where $S_{p_{i}}^{+}$ and
$S_{p_{i}}^{-}$ are $\mathcal{A}(p_{i}) \cap M_{c}^{+}$ and
$\mathcal{D}(p_{i}) \cap M_{c}^{-}$ respectively. Suppose the
smallest (largest) critical value greater (smaller) than $c$ is
$c_{+}$ ($c_{-}$). Here $c_{\pm}$ may be $\pm \infty$. Define $M(c)
= f^{-1}((c_{-}, c_{+}))$. Let
\[
Q_{c}^{+} = \{ (x^{+}, z) \in M_{c}^{+} \times M(c) \mid x^{+} \
\text{and} \  z \  \text{are connected by a generalized flow line}
\}.
\]
\[
Q_{c}^{-} = \{ (z, x^{-}) \in M(c) \times M_{c}^{-} \mid x^{-} \
\text{and} \  z \  \text{are connected by a generalized flow line}
\}.
\]
Then $x^{\pm}$ and $z$ are connected by broken generalized flow
lines if and only if $(x^{+}, z) \in \bigsqcup_{p_{i}} S_{p_{i}}^{+}
\times D_{p_{i}}$ and $(z, x^{-}) \in \bigsqcup_{p_{i}} A_{p_{i}}
\times S_{p_{i}}^{-}$ respectively, where $D_{p_{i}} =
\mathcal{D}(p_{i}) \cap M(c)$ and $A_{p_{i}} = \mathcal{A}(p_{i})
\cap M(c)$.

\begin{lemma}\label{p_manifold}
Suppose the metric is locally trivial. Then $P_{c}$ is a smoothly
embedded submanifold with boundary $\bigsqcup_{p_{i}} S_{p_{i}}^{+}
\times S_{p_{i}}^{-}$ of $M_{c}^{+} \times M_{c}^{-}$.
\end{lemma}
\begin{proof}
There is no essential difference in the proof between the case of
one critical point and that of several critical points. For
convenience, we may assume there is only one critical point $p$ in
$M^{c-\epsilon, c+\epsilon}$. We shall prove that $P_{c}$ is a
smooth embedding submanifold with boundary $S_{p}^{+} \times
S_{p}^{-}$ of $M_{c}^{+} \times M_{c}^{-}$.

Firstly, we shall prove $P_{c} - S_{p}^{+} \times S_{p}^{-}$ is a
smoothly embedded submanifold of $M_{c}^{+} \times M_{c}^{-}$.

Since it is an open subset of $M_{c}^{+}$, $M_{c}^{+} - S_{p}^{+}$
is a smooth submanifold of $M_{c}^{+}$. By Corollary \ref{flow_map},
we can define the flow map $\psi: M_{c}^{+} - S_{p}^{+}
\longrightarrow M_{c}^{-} - S_{p}^{-}$. Define $\varphi: M_{c}^{+} -
S_{p}^{+} \longrightarrow M_{c}^{+} \times M_{c}^{-}$ by $\varphi
(x_{+}) = (x_{+}, \psi(x_{+}))$. Clearly, $\varphi$ is smooth and
$\textrm{Im}(\varphi) = P_{c} - S_{p}^{+} \times S_{p}^{-}$. Define
$\pi_{+}: M_{c}^{+} \times M_{c}^{-} \longrightarrow M_{c}^{+}$ to
be the natural projection. We have $\pi_{+}$ is smooth and $\pi_{+}
\varphi = \textrm{Id}$, so $\varphi$ is a homeomorphism to its
image. Since $d \pi_{+} d \varphi =\textrm{Id}$, $d \varphi$ is an
isomorphism to its image. Thus $P_{c} - S_{p}^{+} \times S_{p}^{-} =
\textrm{Im}(\varphi)$ is a smooth manifold of $M_{c}^{+} \times
M_{c}^{-}$.

Secondly, we shall prove that there is an open neighborhood $W$ of
$S_{p}^{+} \times S_{p}^{-}$ in $M_{c}^{+} \times M_{c}^{-}$ such
that $W \cap P_{c}$ is a smooth submanifold with boundary $S_{p}^{+}
\times S_{p}^{-}$ in $W$.

By local triviality of the metric, there is a diffeomorphism $h: B
\longrightarrow U$ given by (\ref{localization_map}) which satisfies
(\ref{localization_function}) and (\ref{localization_dymamics}). For
convenience, we identify $B$ with $U$. Then $S_{p}^{+} = \{ (0,
v_{2}) \mid \| v_{2} \|^{2} = \epsilon \}$, $S_{p}^{-} = \{ (v_{1},
0) \mid \| v_{1} \|^{2} = \epsilon \}$, $M_{c}^{+} \cap U = \{
(v_{1}, v_{2}) \mid \| v_{1} \|^{2} <2 \epsilon,$ and $\| v_{2}
\|^{2} <2 \epsilon, - \| v_{1} \|^{2} + \| v_{2} \|^{2} = \epsilon
\}$ and $M_{c}^{-} \cap U = \{ (v_{1}, v_{2}) \mid \| v_{1} \|^{2}
<2 \epsilon,$ and $\| v_{2} \|^{2} <2 \epsilon, - \| v_{1} \|^{2} +
\| v_{2} \|^{2} = -\epsilon \}$. Let $U_{+} = \{ (v_{1}, v_{2}) \mid
\| v_{1} \|^{2} < \epsilon$ and $\frac{\epsilon}{2} < \| v_{2}
\|^{2} <2 \epsilon. \}$ and $U_{-} = \{ (v_{1}, v_{2}) \mid \| v_{2}
\|^{2} < \epsilon$ and $\frac{\epsilon}{2} < \| v_{1} \|^{2} <2
\epsilon. \}$. Then $U_{+} \times U_{-}$ is an open neighborhood of
$S_{p}^{+} \times S_{p}^{-}$ in $M^{c-\epsilon, c+\epsilon} \times
M^{c-\epsilon, c+\epsilon}$. For convenience, we identify
$S_{p}^{-}$ with $\{ v_{1} \mid (v_{1}, 0) \in S_{p}^{-} \}$ and
$S_{p}^{+}$ with $\{ v_{2} \mid (0, v_{2}) \in S_{p}^{+} \}$.

Consider the map $\varphi: S_{p}^{+} \times S_{p}^{-} \times [0,1)
\longrightarrow U_{+} \times U_{-}$ satisfying
\[
\varphi (v_{2}, v_{1}, s) = \left( (sv_{1}, (1+s^{2})^{\frac{1}{2}}
v_{2}), ((1+s^{2})^{\frac{1}{2}} v_{1}, sv_{2}) \right).
\]

Clearly, $\varphi$ is smooth, $\textrm{Im}(\varphi) = P_{c} \cap
(U_{+} \times U_{-})$ and $\varphi|_{ S_{p}^{+} \times S_{p}^{-} } =
\textrm{Id}$. On the other hand, consider the map $\alpha: U_{+}
\times U_{-} \longrightarrow S_{p}^{+} \times S_{p}^{-} \times
[0,1)$ satisfying
\[
  \alpha ( (z_{1}, z_{2}),(z_{3}, z_{4}) ) = \left( \epsilon^{\frac{1}{2}} \frac{z_{2}}{\|z_{2}\|},
  \epsilon^{\frac{1}{2}} \frac{z_{3}}{\|z_{3}\|}, \epsilon^{-\frac{1}{2}} \|z_{1}\|
  \right).
\]
Then $\alpha$ is continuous and $\alpha \varphi = \textrm{Id}$. In
addition, $\alpha$ is smooth when $z_{1} \neq 0$. Then $\varphi$ is
a homeomorphism to its image, and $d \varphi$ is an isomorphism onto
its image when $s \neq 0$.

Now we consider the case of $s=0$. We shall prove that $d
\varphi|_{s=0}$ is an isomorphism onto its image. It suffices to
prove that there exists $\lambda > 0$, for all $v \in T(S_{p}^{+}
\times S_{p}^{-} \times [0,1))$, such that
\begin{equation}\label{p_manifold_1}
\|d \varphi \cdot v \| \geq \lambda \|v\|.
\end{equation}
Let $\frac{\partial}{\partial s}$ be the positive unit tangent
vector of $[0,1)$, $e_{2}$ and $e_{1}$ are tangent vectors of
$S_{p}^{+}$ and $S_{p}^{-}$. Then
\[
  d \varphi|_{s=0} \left( \frac{\partial}{\partial s} \right) = (v_{1}, 0, 0,
  v_{2}), \quad d \varphi|_{s=0} (e_{1}) = (0, 0, e_{1}, 0), \quad d \varphi|_{s=0} (e_{2}) = (0, e_{2}, 0,
  0).
\]
It's easy to see (\ref{p_manifold_1}) holds.

Thus $\varphi$ is a smooth embedding into $U_{+} \times U_{-}$. Let
$W= (U_{+} \times U_{-}) \cap (M_{c}^{+} \times M_{c}^{-})$. Then
$W$ is an open neighborhood of $S_{p}^{+} \times S_{p}^{-}$ in
$M_{c}^{+} \times M_{c}^{-}$, $P_{c} \cap W = \textrm{Im}(\varphi)$
and $P_{c} \cap W$ is a smoothly embedded submanifold with boundary
$S_{p}^{+} \times S_{p}^{-}$.
\end{proof}

\begin{lemma}\label{q_manifold}
Suppose the metric is locally trivial. Then $Q_{c}^{+}$
($Q_{c}^{-}$) is a smoothly embedded submanifold with boundary
$\bigsqcup_{p_{i}} S_{p_{i}}^{+} \times D_{p_{i}}$
($\bigsqcup_{p_{i}} A_{p_{i}} \times S_{p_{i}}^{-}$) of $M_{c}^{+}
\times M(c)$ ($M(c) \times M_{c}^{-}$).
\end{lemma}
\begin{proof}
We only need to prove the case of $Q_{c}^{+}$.

Let $\tilde{Q}_{c}^{+} = \{ (x^{+},z) \in Q_{c}^{+} \mid f(z) \in
(c_{i} - \frac{\epsilon}{2}, c_{i} + \frac{\epsilon}{2}) \}$. If we
shrink $M(c)$ by an isotopy along flow lines, we get a
diffeomorphism from $M(c)$ to $f^{-1}((c_{i} - \frac{\epsilon}{2},
c_{i} + \frac{\epsilon}{2}))$. This diffeomorphism preserves flow
lines. Thus it induces a diffeomorphism from $Q_{c}^{+}$ to
$\tilde{Q}_{c}^{+}$. Then we only need to prove that
$\tilde{Q}_{c}^{+}$ is a submanifold of $M_{c}^{+} \times
f^{-1}((c_{i} - \frac{\epsilon}{2}, c_{i} + \frac{\epsilon}{2}))$.
We can therefore assume $M(c) = f^{-1}((c_{i} - \frac{\epsilon}{2},
c_{i} + \frac{\epsilon}{2}))$.

The proof is very similar to that of Lemma \ref{p_manifold}. We
assume there is only one critical point in $M(c)$.

Firstly, we prove $Q_{c}^{+} - S_{p}^{+} \times D_{p}$ is a smooth
embedding submanifold of $M_{c}^{+} \times M(c)$. There is a smooth
map $\varphi: M(c)-D_{p} \longrightarrow M_{c}^{+} \times M(c)$ such
that $\varphi (z) = (\psi(z), z)$, where $\psi$ is the flow map from
$M(c)-D_{p}$ to $M_{c}^{+}$. Similarly to Lemma \ref{p_manifold},
$\varphi$ is also a smooth embedding. This gives the proof.

Secondly, we shall find an open neighborhood $W$ of $S_{p}^{+}
\times D_{p}$ such that $Q_{c}^{+} \cap W$ is a smoothly embedded
submanifold with boundary $S_{p}^{+} \times D_{p}$ of $M_{c}^{+}
\times M(c)$.

Just as the proof of Lemma \ref{p_manifold}, we use the same
notation of $h$, $B$, $U$ and $U_{+}$, identify $U$ with $B$, and we
define $\widetilde{U}_{-} = \{(v_{1}, v_{2}) \mid \| v_{2} \|^{2} <
\epsilon$ and $\| v_{1} \|^{2} <2 \epsilon. \}$. Define $\varphi:
S_{p}^{+} \times D_{p} \times [0,1) \longrightarrow U_{+} \times
\widetilde{U}_{-}$ by
\[
  \varphi (v_{2}, v_{1}, s) = \left( (s v_{1},(s^{2} \| v_{1} \|^{2} + \epsilon)^{\frac{1}{2}}
  \epsilon^{-\frac{1}{2}} v_{2}), (v_{1}, s (s^{2} \| v_{1} \|^{2} + \epsilon)^{\frac{1}{2}}
  \epsilon^{-\frac{1}{2}} v_{2}) \right).
\]
Define $\alpha: U_{+} \times \widetilde{U}_{-} \longrightarrow
S_{p}^{+} \times D_{p} \times [0,1)$ by
\[
  \alpha ((z_{1}, z_{2}), (z_{3}, z_{4})) = \left( \epsilon^{\frac{1}{2}} \frac{z_{2}}{\|z_{2}\|},
  z_{3}, \frac{\| z_{4} \|}{\| z_{2} \|} \right).
\]
Then $\alpha \varphi = \textrm{Id}$. Similar to the proof of Lemma
\ref{p_manifold}, $\varphi$ is a homeomorphism. And $d \varphi$ is
an isomorphism to its image when $s \neq 0$. When $s=0$,
\begin{equation}\label{q_manifold_1}
d \varphi|_{s=0} \left( \frac{\partial}{\partial s} \right) =
(v_{1}, 0, 0, v_{2}),
\end{equation}
\[
  d \varphi|_{s=0} (e_{1}) = (0, 0, e_{1}, 0), \qquad d \varphi|_{s=0} (e_{2}) = (0, e_{2}, 0,
  0),
\]
and $d \varphi$ is also an isomorphism to its image. Thus $\varphi$
is a smooth embedding. Let $W=(U_{+} \times \widetilde{U}_{-}) \cap
(M_{c}^{+} \times M(c))$. This finishes the proof.
\end{proof}

We shall cut out a submanifold with corners from a manifold with
corners. This requires a result about transversality. (See \cite[II.
E]{palais2} for more details about transversality on Hilbert
manifolds.) First we recall a classical result about manifold with
boundary. Suppose $L$ is a Hilbert manifold with boundary, and
$N_{1}$ and $N_{2}$ are Hilbert manifolds. Assume $N_{2}$ is an
embedded submanifold of $N_{1}$. Suppose $g: L \longrightarrow
N_{1}$ is a smooth manifold transversal to $N_{2}$ both in
$L^{\circ} = L - \partial L$ and in $\partial L$. Then
$g^{-1}(N_{2})$ is an embedded submanifold with boundary inside $L$,
and $\partial g^{-1}(N_{2}) = g^{-1}(N_{2}) \cap \partial L$. Now we
extend this result to the product of manifolds with boundary.
Suppose $L_{i}$ ($i=1, \cdots, n$) are Hilbert manifolds with
boundary. Then $\prod_{i=1}^{n} L_{i}$ is a Hilbert manifold with
corners. Its $k$-stratum is just $\partial^{k} \prod_{i=1}^{n} L_{i}
= \bigsqcup_{|\Lambda|=k}(\prod_{i \in \Lambda}
\partial L_{i} \times \prod_{i \notin \Lambda} L_{i}^{\circ})$,
where $\Lambda$ is a subset of $\{1, \cdots, n \}$. The above
extends Definitions \ref{manifold_with_corner} and \ref{k_stratum}.
We have the following result, whose proof is a straightforward
extension of that in the case of a manifold with boundary.

\begin{lemma}\label{cut_submanifold}
If $g: \prod_{i=1}^{n} L_{i} \longrightarrow N_{1}$ is transversal
to $N_{2}$ in each stratum of $ \prod_{i=1}^{n} L_{i}$, then
$g^{-1}(N_{2})$ is a smoothly embedded submanifold with corners of
$\prod_{i=1}^{n} L_{i}$ such that $\partial^{k} g^{-1}(N_{2}) =
g^{-1}(N_{2}) \cap \partial^{k} \prod_{i=1}^{n} L_{i}$.
\end{lemma}

%--------------------------------------------------------------------------------------------------------------------
\subsection{Proof of Theorem \ref{m(p,q)_manifold}}
\begin{proof}
(1) \& (2). We only prove the corner structure now. The face
structure will follow from (4). Suppose the critical values in
$[f(q), f(p)]$ are exactly $c_{l+1} < \cdots < c_{1} < c_{0}$, where
$c_{0} = f(p)$ and $c_{l+1} = f(q)$. Define
\[
P = \prod_{i=1}^{l} P_{i}, \quad R = S_{p}^{-} \times
\prod_{i=1}^{l-1} M_{i}^{-} \times S_{q}^{+}, \quad O =
\prod_{i=1}^{l} (M_{i}^{+} \times M_{i}^{-}).
\]
Here $P_{i} = P_{c_{i}}$, $M_{i}^{+} = M_{c_{i}}^{+}$, $M_{i}^{-} =
M_{c_{i}}^{-}$, $S_{p}^{-} = \mathcal{D}(p) \cap M_{1}^{+}$ and
$S_{q}^{+} = \mathcal{A}(q) \cap M_{l}^{-}$ are defined as before
Lemma \ref{p_manifold}. By Lemma \ref{p_manifold}, $P$ is a manifold
with corners whose $k$-stratum is exactly the disjoint union of
$\prod_{i=1}^{k} (S_{r_{i}}^{+} \times S_{r_{i}}^{-}) \times
\prod_{j \notin \Lambda_{I}} P_{j}^{\circ}$, where $I= \{p, r_{1},
\cdots, r_{k}, q\}$ is a critical sequence and $\Lambda_{I}= \{ j
\mid c_{j} = f(r_{i}), i=1, \cdots, k \}$. Clearly, $P$ is a
submanifold of $O$, so there is an inclusion $\iota: P
\longrightarrow O$. On the other hand, define a smooth embedding
$\Delta: R \longrightarrow O$ as follows. Since there is no critical
point in $M^{c_{i+1}+\frac{\epsilon}{2}, c_{i}-\frac{\epsilon}{2}}$,
by Corollary \ref{flow_map}, we have a flow map $\psi_{i}: M_{i}^{-}
\longrightarrow M_{i+1}^{+}$. Define
\[
  \Delta (y_{0}^{-}, y_{1}^{-}, \cdots, y_{l-1}^{-}, y_{l+1}^{+}) =
  (\psi_{0} y_{0}^{-}, y_{1}^{-}, \psi_{1} y_{1}^{-}, \cdots, y_{l-1}^{-}, \psi_{l-1} y_{l-1}^{-}, \psi_{l+1}^{-1}
  y_{l+1}^{+}).
\]

Now we point out that $\iota$ is transversal to $\Delta$ in each
stratum of $P$. When $M$ is compact, transversality is proved by
\cite[thm.\ 1]{burghelea_haller}. (The paper \cite{burghelea_haller}
uses different notations from ours. Its $\mathcal{P}$, $\mathcal{S}$
and $\mathcal{O}$ are our $P$, $R$ and $O$ respectively. Its maps
$p$ and $s$ are our $\iota$ and $\Delta$ respectively.) The proof
needs Corollary \ref{flow_map} which is trivial in the compact case.
Our proof of the transverality duplicates that in
\cite{burghelea_haller}, so we omit it.

Denote $K = \iota^{-1}(\textrm{Im}(\Delta))$. By Lemma
\ref{cut_submanifold}, $K$ is a smoothly embedded submanifold of $P$
whose $k$-stratum is exactly the intersection of $K$ with the
$k$-stratum of $P$.

Now we identify the strata of $K$ with the disjoint unions of
$\mathcal{M}(p,r_{1}) \times \mathcal{M}(r_{1},r_{2}) \times \cdots
\times \mathcal{M}(r_{k},q)$ as smooth manifolds.

It's easy to see that
\begin{eqnarray}\label{m(p,q)_K}
K & = &\{ (x_{1}^{+}, x_{1}^{-}, \cdots,
x_{l}^{+}, x_{l}^{-}) \in O \mid x_{i}^{\pm} \  (1 \leq i \leq l) \\
& & \text{are on a same generalized flow line connecting $p$ and
$q$} \}. \nonumber
\end{eqnarray}
Let $I = \{p, r_{1}, \cdots, r_{k}, q\}$ be a critical sequence. For
all $\Gamma \in \mathcal{M}_{I}$ (see (\ref{comapctify_m})),
$\Gamma$ intersects $M_{i}^{\pm}$ at exactly one point
$x_{i}^{\pm}(\Gamma)$. Thus there is also an evaluation map
$\tilde{E}_{I}: \mathcal{M}_{I} \longrightarrow O$ such that
$\tilde{E}_{I}(\Gamma) = (x_{1}^{+}(\Gamma), \cdots,
x_{l}^{-}(\Gamma))$. Clearly, $\tilde{E}_{I}$ is a smooth embedding,
and $\textrm{Im}(\tilde{E}_{I})$ is exactly $K \cap (
\prod_{i=1}^{k}(S_{r_{i}}^{+} \times S_{r_{i}}^{-}) \times \prod_{j
\notin \Lambda_{I}} P_{j}^{\circ})$ which is an open subset of the
$k$-stratum of $K$. This gives an identification preserving smooth
structures.

As a result, identifying $\overline{\mathcal{M}(p,q)}$ with $K$, we
give $\overline{\mathcal{M}(p,q)}$ a smooth structure which is
compatible with the smooth structure of $\prod_{i=0}^{k}
\mathcal{M}(r_{i}, r_{i+1})$ for all critical sequences and its
$k$-stratum is exactly $\bigsqcup_{|I|=k} \mathcal{M}_{I}$.

Now we prove the compactness of $\overline{\mathcal{M}(p,q)}$.

By (\ref{m(p,q)_K}), for all $\{x_{n}\}_{n=1}^{\infty} \subseteq K$,
$x_{n} = (x_{n,1}^{+}, x_{n,1}^{-}, \cdots, x_{n,l}^{+},
x_{n,l}^{-})$, $x_{n,i}^{\pm} \in M_{i}^{\pm}$ and are on a same
generalized flow line connecting $p$ and $q$. By Theorem
\ref{point_compactness}, $\{x_{n}\}$ has a cluster point $x_{0} =
(x_{1}^{+}, x_{1}^{-}, \cdots, x_{l}^{+}, x_{l}^{-})$, and
$x_{i}^{\pm}$ are on a same generalized flow line connecting $p$ and
$q$. Since $M_{i}^{\pm}$ is closed, $x_{i}^{\pm} \in M_{i}^{\pm}$ or
$x_{0} \in K$. So $K$ and then $\overline{\mathcal{M}(p,q)}$ are
compact.

(3). Since $a_{i}$, $c_{i} - \frac{\epsilon}{2}$ and $ c_{i+1} +
\frac{\epsilon}{2}$ are in $(c_{i+1}, c_{i})$ and there is no
critical value in $(c_{i+1}, c_{i})$, by Corollary \ref{flow_map},
the flow map gives a smooth map from $f^{-1}(a_{i})$ to $M_{i}^{-}
\times M_{i+1}^{+}$. This induces a map $\varphi: \prod_{i=0}^{l}
f^{-1}(a_{i}) \longrightarrow O$. Clearly, $\varphi \circ E:
\overline{\mathcal{M}(p,q)} \longrightarrow O$ is exactly the
inclusion if we identify $\overline{\mathcal{M}(p,q)}$ with $K$. So
$\varphi \circ E$ and then $E$ are smooth embeddings.

(4). Suppose $f(r) = c_{k}$. By (3), we have the following
commutative diagram. Here $E_{p,q}$, $E_{p,r}:
\overline{\mathcal{M}(p,r)} \longrightarrow \prod_{i=0}^{k-1}
f^{-1}(a_{i})$, and $ E_{r,q}: \overline{\mathcal{M}(r,q)}
\longrightarrow \prod_{i=k}^{l} f^{-1}(a_{i})$ are evaluation maps.
\[
\xymatrix{
  \overline{\mathcal{M}(p,r)} \times \overline{\mathcal{M}(r,q)}\ar[d]_{i_{(p,r,q)}} \ar[rr]^-{E_{p,r} \times E_{r,q}} && \prod_{i=0}^{l} f^{-1}(a_{i}) \\
  \overline{\mathcal{M}(p,q)} \ar[urr]_{E_{p,q}}                     }
\]
Also by (3), the above three evaluation maps are smooth embeddings.
Then so is $i_{(p,r,q)}$. This completes the proof of (4).

Finally, we establish the face structure of
$\overline{\mathcal{M}(p,q)}$. Suppose $x$ is in the $k$-stratum.
Then $x \in \mathcal{M}_{I}$ for some $I = \{ p, r_{1}, \cdots,
r_{k}, q \}$. Thus $x \in \overline{\mathcal{M}(p,r_{i})} \times
\overline{\mathcal{M}(r_{i},q)}$ for $i = 1, \cdots, k$. Clearly,
$\mathcal{M}(p,r_{i}) \times \mathcal{M}(r_{i},q)$ are $k$ pairwise
disjoint faces. We only need to prove that their closures are
$\overline{\mathcal{M}(p,r_{i})} \times
\overline{\mathcal{M}(r_{i},q)}$ respectively. On the one hand,
since it is compact, $\overline{\mathcal{M}(p,r_{i})} \times
\overline{\mathcal{M}(r_{i},q)}$ contains the closure of
$\mathcal{M}(p,r_{i}) \times \mathcal{M}(r_{i},q)$ in
$\overline{\mathcal{M}(p,q)}$. On the other hand, as
$\mathcal{M}(p,r_{i}) \times \mathcal{M}(r_{i},q)$ is the
$0$-stratum (the interior) of $\overline{\mathcal{M}(p,r_{i})}
\times \overline{\mathcal{M}(r_{i},q)}$, we infer that the closure
of $\mathcal{M}(p,r_{i}) \times \mathcal{M}(r_{i},q)$ contains
$\overline{\mathcal{M}(p,r_{i})} \times
\overline{\mathcal{M}(r_{i},q)}$. Thus the closure is exactly
$\overline{\mathcal{M}(p,r_{i})} \times
\overline{\mathcal{M}(r_{i},q)}$.
\end{proof}

%--------------------------------------------------------------------------------------------------------------------
\subsection{Proof of Theorem \ref{d(p)_manifold}}
Since $f$ is lower bounded, by Theorem \ref{finite_number}, there
are only finite number of critical values in $(-\infty, f(p)]$.
Suppose they are $c_{l} < \cdots < c_{0} = f(p)$. Denote $M(c_{i})$
by $M(i)$, $P_{c_{i}}$ by $P_{i}$ and $Q_{c_{i}}^{+}$ by
$Q_{i}^{+}$, where $M(c_{i})$, $P_{c_{i}}$ and $Q_{c_{i}}^{+}$ are
as defined before Lemma \ref{p_manifold}. Define $U(i) \subseteq
\overline{\mathcal{D}(p)}$ as $U(i) = e^{-1}(M(i))$.

\begin{proof}
(1), (2) \& (3). We shall give each $U(i)$ a smooth structure, and
show that $U(i) \cap U(j)$ is open in both $U(i)$ and $U(j)$ and
smooth structures are compatible in $U(i) \cap U(j)$.

Firstly, when $i=0$, $U(0)$ is identified with $\mathcal{D}(p) \cap
M(0)$. $\mathcal{D}(p) \cap M(0)$ is a smooth embedded submanifold
of $M$. Thus $U(0)$ has a smooth structure by this identification.

Secondly, when $i>0$, let $ Q(i) = \prod_{j=1}^{i-1} P_{j} \times
Q_{i}^{+}$, $O(i) = \prod_{j=1}^{i-1} (M_{j}^{+} \times M_{j}^{-})
\times M_{i}^{+}$ and $R(i) = S_{p}^{-} \times \prod_{j=1}^{i-1}
M_{j}^{-}$. We know that, $\forall x \in Q(i)$, $x = (x_{1}^{+},
x_{1}^{-}, \cdots, x_{i-1}^{-}, x_{i}^{+}, z_{i})$, where
$x_{j}^{\pm} \in M_{j}^{\pm}$ and $z_{i} \in M(i)$. Define a smooth
map $\iota_{i}: Q(i) \longrightarrow O(i)$ by
\[
  \iota_{i} (x_{1}^{+}, x_{1}^{-}, \cdots, x_{i-1}^{-}, x_{i}^{+},
  z_{i}) = (x_{1}^{+}, x_{1}^{-}, \cdots, x_{i-1}^{-}, x_{i}^{+}).
\]
Define a smooth embedding $\Delta_{i}: R(i) \longrightarrow O(i)$ by
\[
  \Delta_{i} (y_{0}^{-}, y_{1}^{-}, \cdots, y_{i-1}^{-}) =
  (\psi_{0} y_{0}^{-}, y_{1}^{-}, \psi_{1} y_{1}^{-}, \cdots, y_{i-1}^{-}, \psi_{i-1}
  y_{i-1}^{-}),
\]
where $\psi_{j}$ is the flow map from $M_{j}^{-}$ to $M_{j+1}^{+}$.

As in the proof of Theorem \ref{m(p,q)_manifold}, we point out that
$\iota_{i}$ is transversal to $\Delta_{i}$ in each stratum of
$Q(i)$. The proof is similar to that of Theorem
\ref{m(p,q)_manifold}.

Thus $\tilde{U}(i) = \iota_{i}^{-1}(\textrm{Im}(\Delta_{i}))$ is a
smooth embedding submanifold of $Q(i)$ whose $k$-stratum is exactly
the intersection of $\tilde{U}(i)$ with the $k$-stratum of $Q(i)$.

Now we identify $U(i)$ with $\tilde{U}(i)$. It's easy to see that
\begin{eqnarray}\label{d(p)_u}
\tilde{U}(i) & = &\{ (x_{1}^{+}, x_{1}^{-}, \cdots,
x_{i-1}^{-}, x_{i}^{+}, z_{i}) \in O(i) \times M(i) \mid x_{j}^{\pm} \\
& & \text{are on a same generalized flow line connecting $p$ and
$z_{i}$.} \}. \nonumber
\end{eqnarray}
Let $I = (p, r_{1}, \cdots, r_{k})$ be a critical sequence. For any
element $(\Gamma, x) \in \mathcal{D}_{I} \cap U(i)$ (see
(\ref{comapctify_d})), $\Gamma$ intersects $M_{j}^{\pm}$ at exactly
one point $x_{j}^{\pm}(\Gamma)$. Thus there is an evaluation map
$\tilde{E}_{I}: \mathcal{D}_{I} \cap U(i) \longrightarrow O(i)
\times M(i)$ such that $\tilde{E}_{I}(\Gamma, x) =
(x_{1}^{+}(\Gamma), x_{1}^{-}(\Gamma), \cdots, x_{i}^{+}(\Gamma),
x)$. Similar to the identification of $K$ with
$\overline{\mathcal{M}(p,q)}$ in the proof of Theorem
\ref{m(p,q)_manifold}, this also identifies $U(i)$ with
$\tilde{U}(i)$ and preserves the smooth structure of the strata. So
we get a desired smooth structure on $U_{i}$.

In each $\tilde{U}(i)$, define $\tilde{e}_{i}: \tilde{U}(i)
\longrightarrow M$ by $\tilde{e}_{i}(x_{1}^{+}, \cdots, x_{i}^{+},
z_{i}) = z_{i}$, then $\tilde{e}_{i}$ is smooth. When we identify
$U(i)$ with $\tilde{U}(i)$, we have $e|_{U(i)} = \tilde{e}_{i}$.
Thus $e|_{U(i)}$ is smooth.

Now we check the compatibility of smooth structures for all $U(i)$
($0 \leq i \leq l$). Clearly, if $|i-j|>1$, then $U(i) \cap U(j) =
\emptyset$. We only need to check the compatibility of $U(i)$ and
$U(i+1)$.

Denote $M(i)^{-} = f^{-1}((c_{i+1}, c_{i}))$. For clarity, when we
consider $U(i) \cap U(i+1)$ as a topological subspace of $U(i)$ (or
$U(i+1)$), we denote it by $U(i,i+1)$ (or $U(i+1,i)$). Since
$U(i,i+1) = e|_{U(i)}^{-1}(M(i)^{-})$, it is an open subset of
$U(i)$. Furthermore, $U(i+1,i)$ is an open subset of $U(i+1)$. When
$i \geq 1$, $U(i,i+1) \subseteq \prod_{j=1}^{i-1} (M_{j}^{+} \times
M_{j}^{-}) \times M_{i}^{+} \times M(i)^{-}$ and $U(i+1,i) \subseteq
\prod_{j=1}^{i} (M_{j}^{+} \times M_{j}^{-}) \times M_{i+1}^{+}
\times M(i)^{-}$. Define $\pi: \prod_{j=1}^{i} (M_{j}^{+} \times
M_{j}^{-}) \times M_{i+1}^{+} \times M(i)^{-} \longrightarrow
\prod_{j=1}^{i-1} (M_{j}^{+} \times M_{j}^{-}) \times M_{i}^{+}
\times M(i)^{-}$ be the natural projection. Define $\varphi:
\prod_{j=1}^{i-1} (M_{j}^{+} \times M_{j}^{-}) \times M_{i}^{+}
\times M(i)^{-} \longrightarrow \prod_{j=1}^{i} (M_{j}^{+} \times
M_{j}^{-}) \times M_{i+1}^{+} \times M(i)^{-} $ such that
\[
  \varphi (x_{1}^{+}, x_{1}^{-}, \cdots, x_{i-1}^{-}, x_{i}^{+},
  z_{i}) = (x_{1}^{+}, x_{1}^{-}, \cdots, x_{i-1}^{-}, x_{i}^{+},
  \psi_{-}(z_{i}), \psi_{+}(z_{i}), z_{i}),
\]
where $\psi_{-}$ and $\psi_{+}$ are flow maps from $M(i)^{-}$ to
$M_{i}^{-}$ and $M_{i+1}^{+}$ respectively. Then $\pi (U(i+1,i)) =
U(i,i+1)$, $\varphi (U(i,i+1)) = U(i+1,i)$, $\pi \varphi|_{U(i,i+1)}
= \textrm{Id}$, and $\varphi \pi|_{U(i+1,i)} = \textrm{Id}$. Thus
$\pi$ and $\varphi$ are diffeomorphisms between $U(i,i+1)$ and
$U(i+1,i)$, and they are the identity on the set $U(i) \cap U(i+1)$.
Thus $U(i)$ and $U(i+1)$ have compatible smooth structures when $i
\geq 1$.

Similarly, $U(0,1) \subseteq M(0)^{-}$ and $U(1,0) \subseteq
M_{1}^{+} \times M(0)^{-}$, and $U(0,1)$ and $U(1,0)$ also have
compatible smooth structures.

As a result, we can patch the smooth structures on all $U(i)$
together to give a smooth structure on $\overline{\mathcal{D}(p)}$
satisfying all properties of (1) and (2) but the face structure and
compactness. Similar to Theorem \ref{m(p,q)_manifold}, the face
structure will follow from (4). Also $e$ is smooth since $e|_{U(i)}$
is smooth. This proves (3).

Finally, we prove compactness.

Let $K(i) = e^{-1} (L(i))$, where $L(i) = f^{-1} ( [
\frac{c_{i+1}+c_{i}}{2}, \frac{c_{i}+c_{i-1}}{2} ] )$. Then $L(i)$
is closed. Similar to proving the compactness of $K$ in the proof of
Theorem \ref{comapctify_d}, we get $K(i)$ is compact. Thus
$\overline{\mathcal{D}(p)}$ is compact because
$\overline{\mathcal{D}(p)} = \bigcup_{i=0}^{l} K(i)$.

This completes the proof of (1), (2) and (3).

(4). First, similar to the proof of (3) of Theorem
\ref{m(p,q)_manifold}, we can prove the following lemma.

\begin{lemma}\label{d(p)_embedding}
Suppose the critical values in $(-\infty, f(p)]$ are $c_{l} < \cdots
< c_{0}$. Let $U(i) = e^{-1} \circ f^{-1} ((c_{i+1}, c_{i-1}))$.
Define $E(i): U(i) \longrightarrow \prod_{j=0}^{i-1} f^{-1}(a_{j})$
$\times M$ by $E(i)(\Gamma, x) = (x_{0}(\Gamma), \cdots,
x_{i-1}(\Gamma), x)$ for any $(\Gamma, x) \in U(i)$ (see
(\ref{comapctify_d})), where $x_{j}(\Gamma)$ is the unique
intersection of $\Gamma$ with $f^{-1}(a_{j})$. Then $E(i)$ is a
smooth embedding for $i=0, \cdots, l$. Here $c_{-1} = -\infty$ and
$c_{l+1} = +\infty$.
\end{lemma}

Clearly, $i_{(p,r)}$ is one to one. Suppose $f(r) = c_{m}$. For
clarity, denote the evaluation map from $\overline{\mathcal{D}(p)}$
and $\overline{\mathcal{D}(r)}$ to $M$ by $e_{p}$ and $e_{r}$
respectively. Let $U_{p}(k) = e_{p}^{-1}(M(k))$ and $U_{r}(k) =
e_{r}^{-1}(M(k))$. Since $\overline{\mathcal{M}(p,r)} \times
\overline{\mathcal{D}(r)}$ is compact, we only need to prove that
$i_{(p,r)}(k): \overline{\mathcal{M}(p,r)} \times U_{r}(k)
\longrightarrow \overline{\mathcal{D}(p)}$ is a smooth embedding for
$k \geq m$.

By Lemma \ref{d(p)_embedding}, $E_{r}(k): U_{r}(k) \longrightarrow
\prod_{j=m}^{k-1} f^{-1}(a_{j}) \times M$ and $E_{p}(k): U_{p}(k)
\longrightarrow \prod_{j=0}^{k-1} f^{-1}(a_{j}) \times M$ are smooth
embeddings. By (3) of Theorem \ref{m(p,q)_manifold}, we know that
$E_{p,r}: \overline{\mathcal{M}(p,r)} \longrightarrow
\prod_{j=0}^{m-1} f^{-1}(a_{j})$ is also a smooth embedding. Thus
$E_{p,r} \times E_{r}(k): \overline{\mathcal{M}(p,r)} \times
U_{r}(k) \longrightarrow \prod_{i=0}^{k-1} f^{-1}(a_{i}) \times M$
is a smooth embedding. In addition, $E_{p,r} \times E_{r}(k) =
E_{p}(k) \circ i_{(p,r)}(k)$. Thus $i_{(p,r)}(k)$ is a smooth
embedding.

Finally, $\overline{\mathcal{D}(p)}$ has the disjoint faces
$\mathcal{M}(p,r) \times \mathcal{D}(r)$. Their closures are
$\overline{\mathcal{M}(p,r)} \times \overline{\mathcal{D}(r)}$. This
gives the face structure of (1).
\end{proof}

%--------------------------------------------------------------------------------------------------------------------
\subsection{Proof of Theorem \ref{w(p,q)_manifold}}
The proof of Theorem \ref{w(p,q)_manifold} is a mixture of the
proofs of Theorems \ref{m(p,q)_manifold} and \ref{d(p)_manifold}.
Thus we only need to give the key constructions in the proof. Just
as the proofs of the previous two theorems, we still use the
notation $M(i)$, $P_{i}$ and $Q_{i}^{\pm}$. Suppose the critical
values in $[f(q), f(p)]$ are exactly $f(q) = c_{l+1} < \cdots <
c_{0} = f(p)$. Define $U(i) \subseteq \overline{\mathcal{W}(p,q)}$
as $U(i) = e^{-1}(M(i))$. Use the notation of $S_{p}^{\pm}$, $D_{p}$
and $A_{p}$ as those appearing before Lemma \ref{p_manifold}.

\begin{proof}
Similarly to the proof of Theorem \ref{d(p)_manifold}, we shall give
each $U(i)$ a smooth structure and then patch them together.

Define $Q(0) = Q_{0}^{-} \times \prod_{j=1}^{l} P_{j}$, $R(0) =
D_{p} \times \prod_{j=0}^{l-1} M_{j}^{-} \times S_{q}^{+}$ and $O(0)
= M(0) \times M_{0}^{-} \times \prod_{j=1}^{l} (M_{j}^{+} \times
M_{j}^{-})$. Define $\Delta_{0}: R(0) \longrightarrow O(0)$ by
\[
\Delta_{0} (z_{0}, y_{0}^{-}, \cdots, y_{l-1}^{-}, y_{l+1}^{+}) =
(z_{0}, y_{0}^{-}, \psi_{0} y_{0}^{-}, \cdots, y_{l-1}^{-},
\psi_{l-1} y_{l-1}^{-}, \psi_{l}^{-1} y_{l+1}^{+}).
\]

Define $Q(l+1) = \prod_{j=1}^{l} P_{j} \times Q_{l+1}^{+}$, $R(l+1)
= S_{p}^{-} \times \prod_{j=1}^{l} M_{j}^{-} \times A_{q}$, and
$O(l+1) = \prod_{j=1}^{l} (M_{j}^{+} \times M_{j}^{-}) \times
M_{l+1}^{+} \times M(l+1)$. Define $\Delta_{l+1}: R(l+1)
\longrightarrow O(l+1)$ by
\[
\Delta_{l+1} (y_{0}^{-}, \cdots, y_{l}^{-}, z_{l+1}) = ( \psi_{0}
y_{0}^{-}, y_{1}^{-}, \psi_{1} y_{1}^{-}, \cdots, y_{l}^{-},
\psi_{l} y_{l}^{-}, z_{l+1}).
\]

When $1 \leq i \leq l$, define $Q(i) = \prod_{j=1}^{i-1} P_{j}
\times Q_{i}^{+} \times Q_{i}^{-} \times \prod_{j=i+1}^{l} P_{j}$,
$R(i) = S_{p}^{-} \times \prod_{j=1}^{i-1} M_{j}^{-} \times M(i)
\times \prod_{j=i}^{l-1} M_{j}^{-} \times S_{q}^{+}$ and $O(i) =
\prod_{j=1}^{i-1} (M_{j}^{+} \times M_{j}^{-}) \times M_{i}^{+}
\times M(i) \times M(i) \times M_{i}^{-} \times \prod_{j=i+1}^{l}
(M_{j}^{+} \times M_{j}^{-})$. Define $\Delta_{i}: R(i)
\longrightarrow O(i)$ by
\begin{eqnarray*}
  & & \Delta_{i} (y_{0}^{-}, y_{1}^{-}, \cdots, y_{i-1}^{-}, z_{i}, y_{i}^{-}, \cdots, y_{l-1}^{-}, y_{l+1}^{+}) \\
  & = & (\psi_{0} y_{0}^{-}, y_{1}^{-}, \psi_{1} y_{1}^{-}, \cdots,
  y_{i-1}^{-}, \psi_{i-1} y_{i-1}^{-}, z_{i}, z_{i}, y_{i}^{-}, \psi_{i} y_{i}^{-}, \cdots, \psi_{l-1} y_{l-1}^{-}, \psi_{l}^{-1} y_{l+1}^{+}).
\end{eqnarray*}

In the above, $\psi_{i}$ are flow maps from $M_{i}^{-}$ to
$M_{i+1}^{+}$.

Define $\iota_{i}: Q(i) \longrightarrow O(i)$ to be the inclusion
for all $i=0, \cdots, l+1$.

Similar to the proof of Theorems \ref{m(p,q)_manifold} and
\ref{d(p)_manifold}, $\iota_{i}$ is transversal to $\Delta_{i}$ in
each stratum of $Q(i)$. Thus $\tilde{U}(i) =
\iota_{i}^{-1}(\textrm{Im}(\Delta_{i}))$ is a smooth manifold with
corners. $U(i)$ can be identified with $\tilde{U}(i)$ and the smooth
structures are preserved. This gives a smooth structure to each
$U(i)$. $e|_{U(i)}$ is smooth, and $U(i)$ and $U(j)$ have compatible
smooth structures. Thus $e$ is smooth. The face structures will
follow from (4). Let $L(i) = f^{-1}([\frac{c_{i+1}+c_{i}}{2},
\frac{c_{i}+c_{i-1}}{2}])$, then $e^{-1}(L(i))$ is compact. Thus
$\mathcal{W}(p,q)$ is compact. This finishes the proof of (1), (2)
and (3).

Finally, (4) is proved by an argument similar to that in (4) of
Theorem \ref{d(p)_manifold}.

This completes the proof.
\end{proof}

%--------------------------------------------------------------------------------------------------------------------
\subsection{Proof of Example \ref{not_c1}}
\begin{proof}
Clearly, there is a Morse function on $CP^{2}$ with such three
critical points and $f(r)=0$. By the Morse Lemma, in a neighborhood
$U$ of $r$, there is a local coordinate chart $(v_{1}, v_{2}, v_{3},
v_{4})$ such that $r$ has the coordinate $(0,0,0,0)$,
$\sum_{i=1}^{4} v_{i}^{2} < 4 \epsilon^{2}$ and, in the local chart,
we have $f(v) = \frac{1}{2} (- v_{1}^{2} - v_{2}^{2} + v_{3}^{2} +
v_{4}^{2})$. We can choose $f$ such that $\epsilon = 1$. We equip
$CP^{2}$ with a metric such that, in $U$, it has the form
\begin{equation}
  (d x_{1})^{2} + \frac{1}{2} (d x_{2})^{2} +
  \frac{1}{4} (d x_{3})^{2} + \frac{1}{4} (d x_{4})^{2}.
\end{equation}
Then the flow with initial value $(v_{1}, v_{2}, v_{3}, v_{4})$ is
$(e^{t} v_{1}, e^{2t} v_{2}, e^{-4t} v_{3}, e^{-4t} v_{4})$.

Consider the map $E: \overline{\mathcal{M}(p,q)} \longrightarrow
M_{0} \times M_{1}$, where $M_{0} = f^{-1}(\frac{1}{2})$ and $M_{1}
= f^{-1}(-\frac{1}{2})$. We shall prove $\textrm{Im} (E)$ is not a
$C^{1}$ embedded submanifold with boundary $E(\mathcal{M}(p,r)
\times \mathcal{M}(r,q))$ of $M_{0} \times M_{1}$.

Clearly, $E(\mathcal{M}(p,r) \times \mathcal{M}(r,q)) = S^{+} \times
S^{-}$, where $S^{+} = \{ (0, 0, v_{3}, v_{4}) \mid v_{3}^{2} +
v_{4}^{2} = 1 \}$ and $S^{-} = \{ (v_{1}, v_{2}, 0, 0) \mid
v_{1}^{2} + v_{2}^{2} = 1 \}$. Let $\tilde{S}^{+} = \{ (v_{3},
v_{4}) \mid v_{3}^{2} + v_{4}^{2} = 1 \}$ and $\tilde{S}^{-} = \{
(v_{1}, v_{2}) \mid v_{1}^{2} + v_{2}^{2} = 1 \}$.

The flow map gives a diffeomorphism from an open neighborhood
$W_{0}$ of $S^{+}$ in $M_{0}$ onto $U \cap (R^{2} \times
\tilde{S}^{+})$ and a diffeomorphism from an open neighborhood
$W_{1}$ of $S^{-}$ in $M_{1}$ onto $U \cap (\tilde{S}^{-} \times
R^{2})$. Thus there is a diffeomorphism $\psi: W_{0} \times W_{1}
\longrightarrow (U \cap (R^{2} \times \tilde{S}^{+})) \times (U \cap
(\tilde{S}^{-} \times R^{2}))$. Denote $\psi(\textrm{Im}(E) \cap
(W_{0} \times W_{1}))$ by $P$. Then
\begin{eqnarray*}
P & = &\{ ((v_{1}, v_{2}, v_{3}, v_{4}), (v_{5}, v_{6}, v_{7},
v_{8})) \mid \text{$(v_{1}, v_{2}, v_{3}, v_{4})
 \in U \cap (R^{2} \times \tilde{S}^{+})$ and} \\
& & \text{$(v_{5}, v_{6}, v_{7}, v_{8}) \in U \cap (\tilde{S}^{-}
\times R^{2})$ are connected by a generalized flow line.} \}.
\end{eqnarray*}

In order to prove $\textrm{Im}(E)$ is not a $C^{1}$ embedded
submanifold of $M_{1} \times M_{2}$, we only need to check $P$ is
not a $C^{1}$ embedded submanifold of $(U \cap (R^{2} \times
\tilde{S}^{+})) \times (U \cap (\tilde{S}^{-} \times R^{2}))$.

Suppose $(v_{1}, v_{2}, v_{3}, v_{4}) \in U \cap (R^{2} \times
\tilde{S}^{+})$ and $(v_{1}, v_{2}) \neq (0,0)$, by a direct
calculation, $(v_{1}, v_{2}, v_{3}, v_{4})$ is connected to
\begin{equation}\label{not_c1_1}
( d^{- \frac{1}{2}} v_{1}, d^{-1} v_{2}, d^{2} v_{3}, d^{2} v_{4} )
\end{equation}
by an unbroken flow line, where
\begin{equation}\label{not_c1_2}
d = \frac{1}{2} v_{1}^{2} + \frac{1}{2} (v_{1}^{4} + 4
v_{2}^{2})^{\frac{1}{2}}.
\end{equation}

We prove our result by contradiction. If $P$ were a $C^{1}$ embedded
submanifold with boundary $\partial P = S^{+} \times S^{-}$, then
there is a $C^{1}$ collar embedding $\varphi: \tilde{S}^{+} \times
\tilde{S}^{-} \times [0, \epsilon) \longrightarrow (R^{2} \times
\tilde{S}^{+}) \times (\tilde{S}^{-} \times R^{2})$ such that
\[
\varphi(\cos \theta^{+}, \sin \theta^{+}, \cos \theta^{-}, \sin
\theta^{-}, s) = ((v_{1}, v_{2}, v_{3}, v_{4}), (v_{5}, v_{6},
v_{7}, v_{8})),
\]
and
\[
\varphi(\cos \theta^{+}, \sin \theta^{+}, \cos \theta^{-}, \sin
\theta^{-}, 0) = ((0, 0, \cos \theta^{+}, \sin \theta^{+}), (\cos
\theta^{-}, \sin \theta^{-}, 0, 0)).
\]
When $s \neq 0$, $\textrm{Im}(\varphi) \cap \partial P = \emptyset$,
thus $(v_{1}, v_{2}) \neq (0,0)$ and $(v_{5}, v_{6}, v_{7}, v_{8})$
equals (\ref{not_c1_1}).

In the following four steps, we will use some estimates. The same
notation $C$ or $C_{i}$ may stand for different constants in
different steps.

Firstly, we prove that $\frac{\partial}{\partial s}|_{s=0} v_{7} =
\frac{\partial}{\partial s}|_{s=0} v_{8} = 0$.

Fix $\theta^{+}$ and $\theta^{-}$, then $v_{1}$ and $v_{2}$ are
$C^{1}$ functions of $s$, and $v_{1} = v_{2} = 0$ when $s=0$. So
there exist $C_{1} > 0$ and $\delta > 0$ such that, for all $s \in
[0, \delta)$, we have $|v_{1}| \leq C_{1} s$ and $|v_{2}| \leq C_{1}
s$. Since $(v_{3}, v_{4}) \in \tilde{S}^{+}$, $(v_{3}, v_{4})$ is
bounded, by (\ref{not_c1_1}) and (\ref{not_c1_2}), there exists
$C_{2} > 0$ such that $|v_{7}| \leq C_{2} s^{2}$ and $|v_{8}| \leq
C_{2} s^{2}$. This proves our first claim.

Secondly, we claim that $\frac{\partial}{\partial s}|_{s=0} (v_{1},
v_{2}) \neq (0,0)$.

If not, then
\[
(d \varphi)|_{s=0} \frac{\partial}{\partial s} = \left( 0, 0,
\frac{\partial}{\partial s}|_{s=0} v_{3}, \frac{\partial}{\partial
s}|_{s=0} v_{4}, \frac{\partial}{\partial s}|_{s=0} v_{5},
\frac{\partial}{\partial s}|_{s=0} v_{6}, 0, 0 \right) \in T(S^{+}
\times S^{-}),
\]
So $(d \varphi)|_{s=0} \frac{\partial}{\partial s}$ is not a normal
vector of $S^{+} \times S^{-}$. This gives a contradiction.

Thirdly, we prove that $\frac{\partial}{\partial s}|_{s=0} v_{2} =
0$.

By the continuity of $\frac{\partial}{\partial s}|_{s=0} v_{2}$, we
only need to prove this is true when $\cos \theta^{-} \neq 0$. Fix
$\theta^{+}$ and $\theta^{-}$, then $\displaystyle \lim_{s
\rightarrow 0} v_{5} = \cos \theta^{-} \neq 0$ and $\displaystyle
\lim_{s \rightarrow 0} v_{6} = \sin \theta^{-}$. Thus there exist
$\delta > 0$, $C_{2} > 0$ and $C_{3} > 0$, such that, for all $s \in
(0, \delta)$, we have $0 < C_{2} \leq |v_{5}|$ and $|v_{6}| \leq
C_{3}$. By (\ref{not_c1_1}), $C_{2} \leq |d^{- \frac{1}{2}} v_{1}|$
and $|d^{- 1} v_{2}| \leq C_{3}$. Then $C_{2}^{2} d \leq
|v_{1}|^{2}$ and $|v_{2}| \leq C_{3} d$. So $|v_{2}| \leq C_{3}
C_{2}^{-2} |v_{1}|^{2}$. In the first step, we showed that there
exists $C_{1} > 0$, shrinking $\delta$ if necessary, we get $|v_{1}|
\leq C_{1} s$. Thus $|v_{2}| \leq C_{3} C_{2}^{-2} C_{1}^{2} s^{2}$.
This gives our third claim.

Finally, we derive the contradiction.

Let $\cos \theta^{-} = 0$ and $\sin \theta^{-} = 1$. Fix
$\theta^{+}$. By the second and the third claims,
$\frac{\partial}{\partial s}|_{s=0} v_{1} \neq 0$. Since $v_{1} = 0$
when $s=0$, then there exist $\delta_{1} > 0$ and $C_{1}
> 0$ such that, for all $s \in [0, \delta_{1})$, we have
\begin{equation}\label{not_c1_3}
|v_{1}| \geq C_{1} s.
\end{equation}
Since $v_{5} = \cos \theta^{-} = 0$ when $s=0$, and $v_{5}$ is a
$C^{1}$ function, then there exist $\delta_{2} > 0$ and $C_{2}
> 0$ such that, for all $s \in [0, \delta_{2})$, we have
\begin{equation}\label{not_c1_4}
|v_{5}| \leq C_{2} s.
\end{equation}
Let $\delta = \min \{ \delta_{1}, \delta_{2} \}$. Combining
(\ref{not_c1_1}),  (\ref{not_c1_3}) and (\ref{not_c1_4}), we can
find $C > 0$ such that, for all $s \in (0, \delta)$, we have
$|d^{-\frac{1}{2}} v_{1}| = |v_{5}| \leq C |v_{1}|$ and $v_{1} \neq
0$. Thus by (\ref{not_c1_2}),
\begin{equation}
\frac{2}{ v_{1}^{2} + (v_{1}^{4} + 4 v_{2}^{2})^{\frac{1}{2}} } =
d^{-1} \leq C^{2}.
\end{equation}
However, when $s \rightarrow 0$, we have $v_{1} \rightarrow 0$,
$v_{2} \rightarrow 0$ and $v_{1}^{2} + (v_{1}^{4} + 4
v_{2}^{2})^{\frac{1}{2}} \rightarrow 0$, then $d^{-1} \rightarrow +
\infty$. This gives a contradiction.
\end{proof}

%--------------------------------------------------------------------------------------------------------------------
\subsection{Additional Results}
We prove two results which are needed later.

First, we have the following result which follows straightforwardly
from the face structure of $\overline{\mathcal{D}(p)}$ (see
Definition \ref{manifold_with_face}).

\begin{lemma}\label{embed_d(p)_normal}
Suppose $I = \{p, r_{1}, \cdots, r_{k}\}$ is a critical sequence and
$x \in \mathcal{D}_{I} \subseteq \overline{\mathcal{D}(p)}$. Then
there exist an open neighborhood $W$ of $x$ in $\mathcal{D}_{I}$ and
a smooth map $\varphi: W \times [0, \epsilon)^{k} \longrightarrow
\overline{\mathcal{D}(p)}$, where $\varphi$ is a diffeomorphism onto
an open neighborhood of $x$ in $\overline{\mathcal{D}(p)}$
satisfying the following stratum condition. For all $y \in W$,
$\rho_{I} = (\rho_{1}, \cdots, \rho_{k}) \in [0, \epsilon)^{k}$ and
$J= \{ p, r_{i_{1}}, \cdots, r_{i_{s}} \}$, we have $\varphi(x,
\rho_{I}) \in \mathcal{D}_{J}$ if and only if $\rho_{j} > 0$ when
$r_{j} \notin J$ and $\rho_{j} = 0$ when $r_{j} \in J$.
\end{lemma}
\begin{proof}
Since $\mathcal{D}_{I}$ is an open subset of the $k$-stratum of
$\overline{\mathcal{D}(p)}$, there is a smooth map $\varphi: W
\times [0, \epsilon)^{k} \longrightarrow \overline{\mathcal{D}(p)}$
which is a diffeomorphism onto an open neighborhood of $x$ in
$\overline{\mathcal{D}(p)}$. As mentioned at the end of the proof of
Theorem \ref{d(p)_manifold}, $\mathcal{D}_{I}$ is contained in
$\overline{F_{i}} = \overline{\mathcal{M}(p,r_{i})} \times
\overline{\mathcal{D}(r_{i})}$, the closure of $k$ disjoint faces
$F_{i} = \mathcal{M}(p,r_{i}) \times \mathcal{D}(r_{i})$ ($i=1,
\cdots, k$). Furthermore, $W \times [0, \epsilon)^{k}$ also has $k$
disjoint faces $G_{i} = W \times (0,\epsilon)^{i-1} \times \{0\}
\times (0, \epsilon)^{k-i}$. The closure of $G_{i}$ is
$\overline{G_{i}} = W \times [0,\epsilon)^{i-1} \times \{0\} \times
[0, \epsilon)^{k-i}$. Since it is a diffeomorphism, $\varphi$ maps a
face into a face. Choose $W$ to be connected, then permutating the
coordinates of $[0,\epsilon)^{k}$ if necessary, we have
$\varphi(G_{i}) \subseteq F_{i}$. Thus $\varphi(\overline{G_{i}})
\subseteq \overline{F_{i}}$. Moreover, using the fact that $\varphi$
is a diffeomorphism again, $x$ is in the $i$-stratum if and only if
$\varphi(x)$ is in the $i$-stratum.
\end{proof}

\begin{lemma}\label{d(p)_normal}
Let $e: \overline{\mathcal{D}(p)} \longrightarrow M$ be the map in
(3) of Theorem \ref{d(p)_manifold}, and let $I=\{ p, r_{1}, \cdots,
r_{k} \}$ and $J=\{ p, r_{1}, \cdots, r_{k-1} \}$ be critical
sequences. Suppose $(\alpha, r_{k}) \in \mathcal{M}_{I} \times
\mathcal{D}(r_{k}) = \mathcal{D}_{I}$. Let $\mathcal{N} \in
T_{(\alpha, r_{k})} (\mathcal{M}_{J} \times
\overline{\mathcal{D}(r_{k-1})})$ represent an inward normal vector
in $N_{(\alpha, r_{k})}(\mathcal{D}_{I},\mathcal{M}_{J} \times
\overline{\mathcal{D}(r_{k-1})})$, and $d e (\mathcal{N}) =
(\mathcal{N}_{1}, \mathcal{N}_{2}) \in V_{-} \times V_{+} =
T_{r_{k}}M$. Then $\mathcal{N}_{2} \neq 0$. (Here $\displaystyle
N_{(\alpha, r_{k})} (\mathcal{D}_{I}, \mathcal{M}_{J} \times
\overline{\mathcal{D}(r_{k-1})}) = \frac{T_{(\alpha, r_{k})}
\mathcal{M}_{J} \times \overline{\mathcal{D}(r_{k-1})}}{T_{(\alpha,
r_{k})} \mathcal{D}_{I}}$ is the normal space  of $\mathcal{D}_{I}$
in $\mathcal{M}_{J} \times \overline{\mathcal{D}(r_{k-1})}$, and
$de$ is the derivative of $e$.)
\end{lemma}
\begin{proof}
Suppose the critical values in $(-\infty, f(p)]$ are exactly $c_{0}
> c_{1} > \cdots > c_{l}$. Let $c_{-1} = + \infty$ and $c_{l+1} = - \infty$.
Suppose $f(r_{i}) = c_{t_{i}}$ ($i=1, \cdots, k$). Recall the
evaluation map $e$ in Theorem \ref{d(p)_manifold}. Let $U(t_{k}) =
e^{-1} \circ f^{-1}((c_{t_{k}+1}, c_{t_{k}-1}))$. Then $(\alpha,
r_{k}) \in \mathcal{D}_{I} \cap U(t_{k})$.

Returning to the proof of Theorem \ref{d(p)_manifold}, we have
$U(t_{k})$ is an embedded submanifold of $\prod_{i=1}^{t_{k}-1}
P_{i} \times Q_{t_{k}}^{+}$. We may assume $r_{i}$ is the unique
critical point with function value $c_{t_{i}}$. Otherwise, replace
$P_{t_{i}}$ by its open subset $\{ (x,y) \in P_{t_{i}} \mid \forall
r \neq r_{t_{i}}, \  x \notin \mathcal{A}(r) \cap M_{t_{i}}^{+} \}$
and replace $Q_{t_{k}}^{+}$ by its open subset $\{ (x,y) \in
Q_{t_{k}}^{+} \mid \forall r \neq r_{t_{k}}, \  x \notin
\mathcal{A}(r) \cap M_{t_{i}}^{+} \}$ in this proof.

Denote $\mathcal{D}_{I} \cap U(t_{k})$ by $D_{I}$, and
$(\mathcal{M}_{J} \times \overline{\mathcal{D}(r_{k-1})}) \cap
U(t_{k})$ by $D_{J}$. Then $T_{(\alpha, r_{k})} \mathcal{D}_{I} =
T_{(\alpha, r_{k})} D_{I}$ and $T_{(\alpha, r_{k})} (\mathcal{M}_{J}
\times \overline{\mathcal{D}(r_{k-1})}) = T_{(\alpha, r_{k})}
D_{J}$. Denote $\prod_{j \neq t_{s}} P_{j}^{\circ} \times \prod_{j <
k} \partial P_{t_{j}} \times Q_{t_{k}}^{+}$ by $H$. Then $\partial H
= \prod_{j \neq t_{s}} P_{j}^{\circ} \times \prod_{j<k} \partial
P_{t_{j}} \times \partial Q_{t_{k}}^{+}$. Here $P_{j}^{\circ} =
P_{j} - \partial P_{j}$.

Clearly, $D_{I} = \partial H \cap \iota_{t_{k}}^{-1}(\textrm{Im}
\Delta_{t_{k}})$ and $D_{J} = H \cap \iota_{t_{k}}^{-1}(\textrm{Im}
\Delta_{t_{k}})$. We have the following inclusion of pairs
\[
(T_{(\alpha, r_{k})} D_{I}, T_{(\alpha, r_{k})} D_{J})
\longrightarrow (T_{(\alpha, r_{k})} \partial H, T_{(\alpha, r_{k})}
H).
\]
Since $\iota_{t_{k}}$ is transversal to $\Delta_{t_{k}}$ in
$\partial H$, the above inclusion induce an isomorphism
\[
N_{(\alpha, r_{k})}(D_{I}, D_{J}) \longrightarrow N_{(\alpha,
r_{k})} (\partial H, H)
\]
Thus $\mathcal{N}$ also represents an inward normal vector in
$N_{(\alpha, r_{k})} (\partial H, H)$.

By the proof of Lemma \ref{q_manifold} (see (\ref{q_manifold_1})),
another such representative element is
\[
\widetilde{\mathcal{N}} = (0, \cdots, 0, (v_{1}, 0), (0, v_{2})) \in
T_{(\alpha, r_{k})} \left( \prod_{j \neq t_{i}} P_{j}^{\circ} \times
\prod_{j=1}^{k-1} \partial P_{t_{j}} \times Q_{t_{k}}^{+} \right) =
T_{(\alpha, r_{k})} H,
\]
where $((v_{1}, 0), (0, v_{2})) \in T Q_{t_{k}}^{+} \subseteq T
M_{t_{k}}^{+} \times T M(t_{k})$, and $0 \neq (0, v_{2}) \in V_{-}
\times V_{+} = T_{r_{k}}M$.

Since both $\mathcal{N}$ and $\widetilde{\mathcal{N}}$ are inward
normal vectors, we have $\mathcal{N} = a \widetilde{\mathcal{N}} +
w$ for some $a > 0$ and $w \in T_{(\alpha, r_{k})} (\prod_{j \neq
t_{i}} P_{j}^{\circ} \times \prod_{j=1}^{k-1} \partial P_{t_{j}}
\times \partial Q_{t_{k}}^{+}) = T_{(\alpha, r_{k})} \partial H$.
Clearly,
\[
w = (w_{1}, \cdots, w_{t_{k}-1}, (0, \tilde{v}_{2}), (\tilde{v}_{1},
0)),
\]
where $(0, \tilde{v}_{2}) \in T S_{t_{k}}^{+}$ and $(\tilde{v}_{1},
0) \in V_{-} \times \{0\} = T_{r_{k}} \mathcal{D}(r_{k})$.

Since the evaluation map $e$ on $U(t_{k})$ is just the projection
$\prod_{j=1}^{t_{k}-1} P_{j} \times Q_{t_{k}}^{+} \longrightarrow
Q_{t_{k}}^{+} \subseteq M(t_{k})$, we have $de (\mathcal{N}) =
(\tilde{v}_{1}, a v_{2})$. Thus $\mathcal{N}_{2} = a v_{2} \neq 0$.
\end{proof}

%--------------------------------------------------------------------------------------------------------------------
%--------------------------------------------------------------------------------------------------------------------
\section{Orientation}

%--------------------------------------------------------------------------------------------------------------------
\subsection{Definition of Orientations}\label{subsection_definition_orientation}
Before defining the orientations of $\overline{\mathcal{M}(p,q)}$,
$\overline{\mathcal{D}(p)}$ and $\overline{\mathcal{W}(p,q)}$, we
give a general way to get an orientation by transversality.

Suppose $M_{1}$, $M_{2}$ and $M_{3}$ are three Hilbert manifolds
such that $M_{2}$ is embedded in $M_{3}$. The normal bundle of
$M_{2}$ with respect to $M_{3}$ is defined as $\displaystyle
N(M_{2}, M_{3}) = \frac{T_{M_{2}}M_{3}}{TM_{2}}$. Here
$T_{M_{2}}M_{3}$ is the restriction of $TM_{3}$ on $M_{2}$. If
$\varphi: M_{1} \longrightarrow M_{3}$ is transversal to $M_{2}$,
then $M_{0} = \varphi^{-1}(M_{2})$ is an embedded submanifold of
$M_{1}$, and $d \varphi$ induces a bundle map $d \varphi: N(M_{0},
M_{1}) \longrightarrow N(M_{2}, M_{3})$, i.e., $d \varphi$ is an
isomorphism in each fiber. If $M_{1}$ is finite dimensional and
oriented and $N(M_{2}, M_{3})$ is a finite dimensional (i.e., the
fiber is finite dimensional) and oriented bundle, then we can give
an orientation of $M_{0}$ as follows. The orientation of $N(M_{2},
M_{3})$ gives an orientation to $N(M_{0}, M_{1})$ via $d \varphi$.
Let $\pi: T_{M_{0}}M_{1} \longrightarrow N(M_{0}, M_{1})$ be the
natural projection. For all $x \in M_{0}$, choose $\{ e_{k+1},
\cdots, e_{n} \} \subseteq T_{x} M_{1}$ such that $\{ \pi (e_{k+1}),
\cdots, \pi (e_{n}) \}$ is a positive base of $N_{x}(M_{0}, M_{1})$.
Choose $\{ e_{1}, \cdots, e_{k} \} \subseteq T_{x} M_{0}$ such that
$\{ e_{1}, \cdots, e_{k}, e_{k+1}, \cdots, e_{n} \}$ is a positive
base of $T_{x} M_{1}$, then $\{ e_{1}, \cdots,$ $e_{k} \}$ gives
$M_{0}$ an orientation. Clearly, this is well defined and only
depends on the orientations of $M_{1}$ and $N(M_{2}, M_{3})$.

By the above method, we can derive the orientations of
$\overline{\mathcal{M}(p,q)}$, $\overline{\mathcal{D}(p)}$ and
$\overline{\mathcal{W}(p,q)}$ provided that the orientations of
$\mathcal{D}(p)$ and $\mathcal{D}(q)$ have been assigned
arbitrarily.

Firstly, we can give $\mathcal{W}(p,q)$ an orientation.

Since $\mathcal{A}(q)$ is transversal to $\mathcal{D}(q)$ at $q$,
then the orientation of $T_{q} \mathcal{D}(q)$ induces an
orientation of $N(\mathcal{A}(q), M)$. Let $i: \mathcal{D}(p)
\longrightarrow M$ be the inclusion. $i$ is transversal to
$\mathcal{A}(q)$ and $i^{-1}(\mathcal{A}(q)) = \mathcal{W}(p,q)$.
The orientations of $\mathcal{D}(p)$ and $N(\mathcal{A}(q), M)$
determine an orientation of $\mathcal{W}(p,q)$.

Secondly, we can give $\mathcal{M}(p,q)$ an orientation.

Choose a regular value $a \in (-\infty, f(p))$. We give $S_{p}^{-} =
\mathcal{D}(p) \cap f^{-1}(a)$ the \textbf{induced orientation} from
$\mathcal{D}(p)$ as follows. For all $x \in S_{p}^{-}$, $\{ e_{1},
\cdots, e_{n} \}$ is a positive base of $T_{x} S_{p}^{-}$ if and
only if $\{ -\nabla f, e_{1}, \cdots, e_{n} \}$ is a positive base
of $T_{x} \mathcal{D}(p)$. Suppose $a \in (f(q), f(p))$. Denote
$\mathcal{A}(q) \cap f^{-1}(a)$ by $S_{q}^{+}$. Then both
$S_{p}^{-}$ and $S_{q}^{+}$ are embedded submanifolds of $f^{-1}(a)$
which are transversal to each other. $S_{p}^{-}$ has its induced
orientation from $\mathcal{D}(p)$ as above. There is a natural
bundle map from $N(S_{q}^{+}, f^{-1}(a))$ to $N(\mathcal{A}(q), M)$.
Thus $N(S_{q}^{+}, f^{-1}(a))$ is an oriented bundle. The
orientations of $S_{p}^{-}$ and $N(S_{q}^{+}, f^{-1}(a))$ give
$S_{p}^{-} \cap S_{q}^{+} = \mathcal{W}(p,q) \cap f^{-1}(a)$ an
orientation. The natural identification between $\mathcal{M}(p,q)$
and $\mathcal{W}(p,q) \cap f^{-1}(a)$ (see the comment below
Definition \ref{moduli_space}) moves the orientation of
$\mathcal{W}(p,q) \cap f^{-1}(a)$ to an orientation of
$\mathcal{M}(p,q)$. Clearly, this orientation only depends on those
of $\mathcal{D}(p)$ and $\mathcal{D}(q)$.

Thirdly, since $\mathcal{M}(p,q)$, $\mathcal{D}(p)$ and
$\mathcal{W}(p,q)$ are the interiors of
$\overline{\mathcal{M}(p,q)}$, $\overline{\mathcal{D}(p)}$ and
$\overline{\mathcal{W}(p,q)}$ respectively, the orientation of each
interior determines a unique orientation of each compactified space.

Assign orientations to descending manifolds of all critical points
arbitrarily. We can consider the orientations of the $1$-strata
$\partial^{1} \overline{\mathcal{M}(p,q)}$, $\partial^{1}
\overline{\mathcal{D}(p)}$ and $\partial^{1}
\overline{\mathcal{W}(p,q)}$ of $\overline{\mathcal{M}(p,q)}$,
$\overline{\mathcal{D}(p)}$ and $\overline{\mathcal{W}(p,q)}$. As
unoriented manifolds, $\partial^{1} \overline{\mathcal{M}(p,q)} =
\bigsqcup_{p \succ r \succ q} \mathcal{M}(p,r)$ $\times
\mathcal{M}(r,q)$. There are two orientations of it. First, since
$\overline{\mathcal{M}(p,q)}$ has an orientation, $\mathcal{M}(p,q)
\sqcup \partial^{1} \overline{\mathcal{M}(p,q)}$ is an oriented
manifold with boundary $\partial^{1} \overline{\mathcal{M}(p,q)}$.
For all $x \in
\partial^{1} \overline{\mathcal{M}(p,q)}$, let $\mathcal{N}$ be an
outward normal vector at $x$. We define an oriented base $\{ e_{1},
\cdots, e_{k} \}$ of $T_{x} \partial^{1}
\overline{\mathcal{M}(p,q)}$ to be positive if and only if $\{
\mathcal{N}, e_{1}, \cdots, e_{k} \}$ is a positive base of $T_{x}
(\mathcal{M}(p,q) \sqcup \partial^{1} \overline{\mathcal{M}(p,q)})$.
We call this the \textbf{boundary orientation} of $\partial^{1}
\overline{\mathcal{M}(p,q)}$. Second, since both $\mathcal{M}(p,r)$
and $\mathcal{M}(r,q)$ have orientations, $\mathcal{M}(p,r) \times
\mathcal{M}(r,q)$ has the product orientation of these two
orientations. This gives $\partial^{1} \overline{\mathcal{M}(p,q)}$
the \textbf{product orientation}. Similarly, we can also define the
boundary orientations and the product orientations for $\partial^{1}
\overline{\mathcal{D}(p)}$ and $\partial^{1}
\overline{\mathcal{W}(p,q)}$.

Theorem \ref{orientation} answers the relations between the boundary
orientations and the product orientations of the above $1$-strata.

%--------------------------------------------------------------------------------------------------------------------
\subsection{Proof of (1) of Theorem \ref{orientation}}
\begin{proof}
We only need to prove that, for all $r$,
\[
\partial (\mathcal{M}(p,q) \sqcup \mathcal{M}(p,r) \times
\mathcal{M}(r,q)) = (-1)^{\textrm{ind}(p) - \textrm{ind}(r)}
\mathcal{M}(p,r) \times \mathcal{M}(r,q).
\]
Denote $\mathcal{M}(p,q) \sqcup \mathcal{M}(p,r) \times
\mathcal{M}(r,q)$ by $\widehat{\mathcal{M}(p,q)}$. By local
triviality of the metric, we have the diffeomorphism $h$ in
(\ref{localization_map}) such that (\ref{localization_function}) and
(\ref{localization_dymamics}) hold. In addition, choose $\epsilon$
small enough such that $f(r)$ is the only critical value in
$[f(r)-\epsilon, f(r)+\epsilon]$. For now on, we identify $U$ with
$B$ without any difference. Let $M^{+} = f^{-1}(f(r) + \frac{1}{2}
\epsilon)$ and $M^{-} = f^{-1}(f(r) - \frac{1}{2} \epsilon)$. Let
$S_{p}^{-} = \mathcal{D}(p) \cap M^{+}$, $S_{q}^{+} = \mathcal{A}(q)
\cap M^{-}$, $S_{r}^{+} = \mathcal{A}(r) \cap M^{+}$ and $S_{r}^{-}
= \mathcal{D}(r) \cap M^{-}$. Then $S_{r}^{+} = \{ (0, v_{2}) \in
V_{-} \times V_{+} \mid \| v_{2} \|^{2} = \epsilon \}$ and
$S_{r}^{-} = \{ (v_{1}, 0) \in V_{-} \times V_{+} \mid \| v_{1}
\|^{2} = \epsilon \}$.

Define
\[
L = \{ (x,y) \in S_{p}^{-} \times M^{-} \mid \text{$x$ and $y$ are
connected by a generalized flow line.} \}.
\]
We may assume there is only one critical point $r$ in
$f^{-1}([f(r)-\epsilon, f(r)+\epsilon])$. Otherwise, define $L$ to
be
\[
\{ (x,y) \in (S_{p}^{-} - \bigcup_{r_{i} \neq r} S_{r_{i}}^{+} )
\times M^{-} \mid \text{$x$ and $y$ are connected by a generalized
flow line.} \}
\]
in this argument. Consider the projection $\pi_{+}: M^{+} \times
M^{-} \longrightarrow M^{+}$, then $L = \pi_{+}^{-1}(S_{p}^{-}) \cap
P_{c}$, where $P_{c}$ is defined in Lemma \ref{p_manifold} and
$c=f(r)$. By transversality, $L$ is an smoothly embedded submanifold
with boundary of $M^{+} \times M^{-}$. The interior of $L$ is
\[
L^{\circ} = \{ (x,y) \in L \mid \text{$x$ and $y$ are connected by a
unbroken flow line.} \},
\]
and $\partial L = (S_{p}^{-} \cap S_{r}^{+}) \times S_{r}^{-}$.
Clearly, $S_{p}^{-} \cap S_{r}^{+}$ can be identified with
$\mathcal{M}(p,r)$. We consider it as $\mathcal{M}(p,r)$. Then
$\partial L = \mathcal{M}(p,r) \times S_{r}^{-}$.

Consider the projection $\pi_{\pm}: M^{+} \times M^{-}
\longrightarrow M^{\pm}$. We have $\pi_{+}(L^{\circ}) = S_{p}^{-} -
S_{r}^{+}$, $\pi_{-}(L^{\circ}) = \mathcal{D}(p) \cap M^{-}$, and
$\pi_{\pm}$ give diffeomorphisms from $L^{\circ}$ to its images.
Give $S_{p}^{-} - S_{r}^{+}$ and $\mathcal{D}(p) \cap M^{-}$ the
induced orientations from $\mathcal{D}(p)$ (see Subsection
\ref{subsection_definition_orientation}). Then $\pi_{+}$ and
$\pi_{-}$ move the above two orientations to $L^{\circ}$. These
orientations on $L^{\circ}$ are the same. Thus $L^{\circ}$ has a
preferred orientation.

Clearly, $\pi_{-}: L \longrightarrow M^{-}$ is transversal to
$S_{q}^{+}$ in $L^{\circ}$ and $\partial L$. Just as in (3) of
Theorem \ref{m(p,q)_manifold}, $\pi_{-}^{-1}(S_{q}^{+})$ can be
identified with $\widehat{\mathcal{M}(p,q)}$ because $(x, y) \in
\pi_{-}^{-1}(S_{q}^{+})$ is a pair of points on a generalized flow
line $\Gamma \in \widehat{\mathcal{M}(p,q)}$. Likewise
$(\pi_{-}|_{\partial L})^{-1}(S_{q}^{+})$ can be identified with
$\partial \widehat{\mathcal{M}(p,q)}$. The boundary of
$\pi_{-}^{-1}(S_{q}^{+})$ is exactly $(\pi_{-}|_{\partial
L})^{-1}(S_{q}^{+})$. We consider the orientation of $L$ first in
order to study the one of $\widehat{\mathcal{M}(p,q)}$.

Similarly to $\partial \widehat{\mathcal{M}(p,q)}$, there are two
orientations of $\partial L$. First, the orientation of $L$ gives it
a boundary orientation. Second, the orientations of
$\mathcal{M}(p,r)$ and $S_{r}^{-}$ give $\partial L =
\mathcal{M}(p,r) \times S_{r}^{-}$ a product orientation, where the
orientation of $S_{r}^{-}$ is induced from that of $\mathcal{D}(r)$
(see Subsection \ref{subsection_definition_orientation}). The
following key lemma shows the difference between these two
orientations of $\partial L$.

\begin{lemma}\label{orientation_3}
$\partial L = (-1)^{\textrm{ind}(p) - \textrm{ind}(r)}
\mathcal{M}(p,r) \times S_{r}^{-}$. Here, $\partial L$ is given the
boundary orientation and $\mathcal{M}(p,r) \times S_{r}^{-}$ is
given the product orientation.
\end{lemma}

The proof of Lemma \ref{orientation_3} is based on a good local
collar embedding of $\partial L$ into $ L$ and a subtle computation
of orientations. The collar embedding is provided by the following
two lemmas.

Fix a point$(0, x_{2}) \in \mathcal{M}(p,r)$. We know
$\mathcal{M}(p,r) = S_{p}^{-} \cap S_{r}^{+} \subseteq \{0\} \times
V_{+}$. Define $\widetilde{\mathcal{M}}(p,r) = \{ v_{2} \in V_{+}
\mid (0, v_{2}) \in \mathcal{M}(p,r) \}$.

\begin{lemma}\label{orientation_1}
There exist an open neighborhood $\Omega$ of $x_{2}$ in $V_{+}$, a
$\delta > 0$, and a map $\tilde{\theta}: B_{1}(\delta) \times
(\Omega \cap \widetilde{\mathcal{M}}(p,r)) \longrightarrow V_{-}
\times V_{+}$ such that $\tilde{\theta}(v_{1}, v_{2}) = (v_{1},
\theta(v_{1}, v_{2}))$, $\theta(0, v_{2}) = v_{2}$ and
$\tilde{\theta}$ is a diffeomorphism from $B_{1}(\delta) \times
(\Omega \cap \widetilde{\mathcal{M}}(p,r))$ to $S_{p}^{-} \cap
(B_{1}(\delta) \times \Omega)$. Here $B_{1}(\delta) = \{ v_{1} \in
V_{-} \mid \| v_{1} \|^{2} < \delta \}$.
\end{lemma}

Let $\tilde{S}_{r}^{+} = \{ v_{2} \in V_{+} \mid (0, v_{2}) \in
S_{r}^{+} \}$ and $\tilde{S}_{r}^{-} = \{ v_{1} \in V_{-} \mid
(v_{1}, 0) \in S_{r}^{-} \}$. We can identify $\mathcal{M}(p,r)$
with $\widetilde{\mathcal{M}}(p,r)$ and $\tilde{S}_{r}^{\pm}$ with
$S_{r}^{\pm}$ naturally. Fix a point $(x_{1},0) \in S_{r}^{-}$.

\begin{lemma}\label{orientation_2}
There exist $\delta > 0$, a neighborhood $\Omega_{2}$ of $x_{2}$ in
$V_{+}$ and a neighborhood $\Omega_{1}$ of $x_{1}$ in $V_{+}$ such
that $\varphi: [0, \delta) \times (\Omega_{2} \cap
\widetilde{\mathcal{M}}(p,r)) \times (\Omega_{1} \cap
\tilde{S}_{r}^{-}) \longrightarrow V_{-} \times V_{+} \times V_{-}
\times V_{+}$ is a local collar neighborhood embedding of $\partial
L$ into $L$ near $((0, x_{2}), (x_{1},0))$. Here
\[
\varphi(s, v_{2}, v_{1}) = (s v_{1}, \theta(s v_{1}, v_{2}),
\epsilon^{-\frac{1}{2}} \| \theta(s v_{1}, v_{2}) \| v_{1}, s
\epsilon^{\frac{1}{2}} \| \theta(s v_{1}, v_{2}) \|^{-1} \theta(s
v_{1}, v_{2})),
\]
and $\theta$ is defined in Lemma \ref{orientation_1}.
\end{lemma}

The proof of these three lemmas will be given later.

Since $L$ and $N(S_{q}^{+}, M^{-})$ have orientations, $\pi_{-}^{-1}
(S_{q}^{+})$ has an orientation. By the definitions of the
orientations of $L$ and $\widehat{\mathcal{M}(p,q)}$, the
orientations of $\pi_{-}^{-1} (S_{q}^{+})$ and
$\widehat{\mathcal{M}(p,q)}$ are the same under this identification.
The boundary orientation of $\partial L$ and the orientation of
$N(S_{q}^{+}, M^{-})$ also give $(\pi_{-}|_{\partial L})^{-1}
(S_{q}^{+})$ an orientation. This orientation of
$(\pi_{-}|_{\partial L})^{-1} (S_{q}^{+})$ coincides with the
boundary orientation induced from $\pi_{-}^{-1} (S_{q}^{+})$. The
reason is as follows. At $((0,x_{2}),(x_{1},0)) \in
(\pi_{-}|_{\partial L})^{-1} (S_{q}^{+})$, let $\{ e_{1}, \cdots,
e_{k} \}$ be a base of $T(\pi_{-}|_{\partial L})^{-1} (S_{q}^{+})$
and $\{ e_{k+1}, \cdots, e_{n} \} \subseteq T(\partial L)$ represent
a base of $N((\pi_{-}|_{\partial L})^{-1} (S_{q}^{+}),
\partial L)$. Let $\mathcal{N}$ be an outward normal vector of
$(\pi_{-}|_{\partial L})^{-1} (S_{q}^{+})$ with respect to
$\pi_{-}^{-1} (S_{q}^{+})$. Then $\{ \mathcal{N}, e_{1}, \cdots,
e_{k} \}$ gives an orientation of $\pi_{-}^{-1} (S_{q}^{+})$, $\{
e_{1}, \cdots, e_{k} \}$ gives an orientation of
$(\pi_{-}|_{\partial L})^{-1} (S_{q}^{+})$, $\{ \mathcal{N}, e_{1},
\cdots, e_{n} \}$ gives an orientation of $L$, and $\{ e_{1},
\cdots, e_{n} \}$ gives an orientation of $\partial L$. When $\{
e_{k+1}, \cdots, e_{n} \}$ is positively oriented, $\{ e_{1},
\cdots, e_{k} \}$ gives the boundary orientation if and only if $\{
e_{1}, \cdots, e_{n} \}$ gives the boundary orientation. This is the
reason.

Thus $(\pi_{-}|_{\partial L})^{-1} (S_{q}^{+})$ has the boundary
orientation of $\partial \widehat{\mathcal{M}(p,q)}$ under this
identification if $\partial L$ is equipped with the boundary
orientation.

On the other hand, if we give $\partial L$ the product orientation,
i.e., we consider it as $\mathcal{M}(p,r) \times S_{r}^{-}$, then
$(\pi_{-}|_{\partial L})^{-1} (S_{q}^{+})$ will have the product
orientation of $\mathcal{M}(p,r) \times \mathcal{M}(r,q)$ under this
identification.

By Lemma \ref{orientation_3}, we have completed the proof of (1) of
Theorem \ref{orientation}.
\end{proof}

\begin{proof}[Proof of Lemma \ref{orientation_1}]
Since $\widetilde{\mathcal{M}}(p,r)$ is an embedded submanifold of
$V_{+}$, there exist a neighborhood $\Omega$  of $x_{2}$ and a
diffeomorphism $\alpha: \Omega \longrightarrow V_{+}$ such that
$V_{+} = K_{1} \times K_{2}$, $\alpha(\Omega \cap
\widetilde{\mathcal{M}}(p,r)) = K_{1} \times \{0\}$ and
$\alpha(x_{2}) = (0,0)$. Here $K_{1}$ and $K_{2}$ are two Hilbert
spaces. Define $\beta: B_{1}(\delta) \times \Omega \longrightarrow
B_{1}(\delta) \times V_{+}$ by $\beta(v_{1}, v_{2}) = (v_{1},
\alpha(v_{2}))$. Then $\beta$ is also a diffeomorphism. $\beta
(B_{1}(\delta) \times (\Omega \cap \widetilde{\mathcal{M}}(p,r))) =
B_{1}(\delta) \times K_{1} \times \{0\}$, $\beta (\{0\} \times
\Omega) = \{0\} \times V_{+}$ and $\beta(0, x_{2}) = (0,0,0)$.

Since $S_{p}^{-}$ is transversal to $\{0\} \times \Omega$, then
$\beta (S_{p}^{-} \cap (B_{1}(\delta) \times \Omega))$ is also
transversal to $\{0\} \times V_{+} = \beta (\{0\} \times \Omega)$.
Denote $\beta (S_{p}^{-} \cap (B_{1}(\delta) \times \Omega))$ by
$S$. Then
\begin{equation}\label{orientation_1_1}
T_{(0,0,0)} S + T_{(0,0,0)} (\{0\} \times V_{+}) = T_{(0,0,0)}
(B_{1}(\delta) \times V_{+}) = V_{-} \times V_{+}.
\end{equation}
Consider the map $\pi_{1}: B_{1}(\delta) \times V_{+}
\longrightarrow B_{1}(\delta) \times \{(0,0)\}$, where $\pi_{1}
(v_{1}, k_{1}, k_{2}) = (v_{1}, 0, 0)$. By (\ref{orientation_1_1}),
we get
\[
d \pi_{1}: T_{(0,0,0)} S \longrightarrow T_{(0,0,0)} (B_{1}(\delta)
\times \{(0,0)\}) = V_{-} \times \{(0,0)\}
\]
is surjective. In addition, since $\{0\} \times (\Omega \cap
\widetilde{\mathcal{M}}(p,r)) \subseteq S_{p}^{-} \cap
(B_{1}(\delta) \times \Omega)$, we have
\[
\{0\} \times K_{1} \times \{0\} = \beta (\{0\} \times (\Omega \cap
\widetilde{\mathcal{M}}(p,r))) \subseteq S.
\]
Thus
\begin{equation}\label{orientation_1_2}
\{0\} \times K_{1} \times \{0\} = T_{(0,0,0)} (\{0\} \times K_{1}
\times \{0\}) \subseteq T_{(0,0,0)} S.
\end{equation}
Consider the map $\pi_{2}: B_{1}(\delta) \times V_{+}
\longrightarrow B_{1}(\delta) \times K_{1} \times \{0\}$, where
$\pi_{2} (v_{1}, k_{1}, k_{2}) = (v_{1}, k_{1}, 0)$. By the
surjectivity of $d \pi_{1}$ on $S$ and (\ref{orientation_1_2}), we
know that
\[
d \pi_{2}: T_{(0,0,0)} S \longrightarrow T_{(0,0,0)} (B_{1}(\delta)
\times K_{1} \times \{0\}) = V_{-} \times K_{1} \times \{0\}
\]
is surjective.

Now we count the dimensions of $S$ and $B_{1}(\delta) \times K_{1}
\times \{0\}$.
\begin{eqnarray*}
\textrm{dim}(S) & = & \textrm{dim}(S_{p}^{-}) = \textrm{ind}(p) - 1
= \textrm{ind}(p) - \textrm{ind}(r) - 1 +
\textrm{ind} (r) \\
& = & \textrm{dim}(\mathcal{M}(p,r)) + \textrm{dim}(V_{-}) =
\textrm{dim}(K_{1} \times \{0\}) +
\textrm{dim}(B_{1}(\delta)) \\
& = & \textrm{dim}(B_{1}(\delta) \times K_{1} \times \{0\})
\end{eqnarray*}

By the Inverse Function Theorem, shrinking $\delta$ and $\Omega$  if
necessary, we have that $\pi_{2}$ gives a diffeomorphism from
$S=\beta (S_{p}^{-} \cap (B_{1}(\delta) \times \Omega))$ to
$B_{1}(\delta) \times K_{1} \times \{0\} = \beta (B_{1}(\delta)
\times (\Omega \cap \widetilde{\mathcal{M}}(p,r)))$. Also,
$(\pi_{2}|_{S})^{-1} (v_{1}, k_{1}, 0) = (v_{1}, \hat{\theta}(v_{1},
k_{1}, 0))$ for some $\hat{\theta}$. It's easy to see that $S \cap
(\{0\} \times V_{+}) = \{0\} \times K_{1} \times \{0\}$. Then
$\hat{\theta}(0, k_{1}, 0) = (k_{1}, 0)$.

Defining $\tilde{\theta} = \beta^{-1} \circ (\pi_{2}|_{S})^{-1}
\circ \beta$ on $B_{1}(\delta) \times (\Omega \cap
\widetilde{\mathcal{M}}(p,r))$, completes the proof.
\end{proof}

\begin{proof}[Proof of Lemma \ref{orientation_2}]
We may assume $\epsilon \leq 1$. Choose $\delta$ as in Lemma
\ref{orientation_1}. Choose $\Omega_{2}$ to be $\Omega$ in Lemma
\ref{orientation_1}. Consider $\varphi$ as a map defined in
$(-\delta, \delta) \times (\Omega_{2} \cap
\widetilde{\mathcal{M}}(p,r)) \times \tilde{S}_{r}^{-}$. By Lemma
\ref{orientation_1}, we have $\textrm{Im}(\varphi) \subseteq
L^{\circ}$ when $s>0$, $\textrm{Im}(\varphi) \cap L = \emptyset$
when $s<0$ and $\varphi (0, v_{2}, v_{1}) = ((0, v_{2}), (v_{1}, 0))
\in
\partial L$.

Now we compute $d \varphi$. First, we introduce some notation. Let
$\frac{\partial}{\partial s}$ be the positive unit tangent vector of
$(-\delta, \delta)$. Let $\frac{\partial}{\partial x_{1}}$ be a base
of $T_{x_{1}} \tilde{S}_{r}^{-} \subseteq V_{-}$, i.e.
\[
\frac{\partial}{\partial x_{1}} = \{ e_{1}, \cdots,
e_{\textrm{ind}(r) - 1}. \}
\]
Let $\frac{\partial}{\partial x_{2}}$ be a base of $T_{x_{2}}
\widetilde{\mathcal{M}}(p,r)) \subseteq V_{+}$. The notation $(d
\varphi) \frac{\partial}{\partial x_{1}}$ means
\[
(d \varphi) \frac{\partial}{\partial x_{1}} = \{ (d \varphi) e_{1},
\cdots, (d \varphi) e_{\textrm{ind}(r) - 1} \}.
\]
In the following calculation, omit $d \varphi
\frac{\partial}{\partial x_{2}}$ if
$\textrm{dim}(\widetilde{\mathcal{M}}(p,r)) = 0$ and omit $d \varphi
\frac{\partial}{\partial x_{1}}$ if $\textrm{dim}(\tilde{S}_{r}^{-})
= 0$. At $(s, x_{1}, x_{2})$
\begin{eqnarray}\label{orientation_2_1}
(d \varphi) \frac{\partial}{\partial s} & = & (x_{1}, (d \theta)
x_{1}, \epsilon^{-\frac{1}{2}} \|\theta\|^{-1} \langle \theta, (d
\theta) x_{1} \rangle
x_{1}, \epsilon^{\frac{1}{2}} \|\theta\|^{-1} h + s *) \\
(d \varphi) \frac{\partial}{\partial x_{2}} & = & \left(0, (d
\theta) \frac{\partial}{\partial x_{2}}, \epsilon^{-\frac{1}{2}}
\|\theta\|^{-1} \left \langle \theta, (d \theta)
\frac{\partial}{\partial x_{2}} \right \rangle
x_{1}, s *  \right), \nonumber \\
(d \varphi) \frac{\partial}{\partial x_{1}} & = & \left( s
\frac{\partial}{\partial x_{1}}, s (d \theta)
\frac{\partial}{\partial x_{1}}, \epsilon^{-\frac{1}{2}} \|\theta\|
\frac{\partial}{\partial x_{1}} + s *, s^{2}
* \right).  \nonumber
\end{eqnarray}
Here $*$ stands for some smooth functions which are not important.
Since $\theta(0,v_{2}) \equiv v_{2}$, we have $(d \theta) (0, x_{2})
\frac{\partial}{\partial x_{2}} = \frac{\partial}{\partial x_{2}}$.
And since $\frac{\partial}{\partial x_{2}}$ is contained in
$T_{x_{2}} \tilde{S}_{r}^{+}$ and is orthogonal to $x_{2}$, we have
$\langle \theta(0, x_{2}), (d \theta) (0, x_{2})
\frac{\partial}{\partial x_{2}} \rangle = 0$. In addition, $\| x_{2}
\| = \epsilon^{\frac{1}{2}}$. Thus
\[
(d \varphi) (0, x_{2}, x_{1}) \frac{\partial}{\partial s} = (x_{1},
d \theta (0,x_{2}) x_{1}, \epsilon^{-1} \langle x_{2}, d \theta (0,
x_{2}) x_{1} \rangle x_{1}, x_{2}),
\]
\begin{equation}\label{orientation_2_2}
(d \varphi) (0, x_{2}, x_{1}) \frac{\partial}{\partial x_{2}} =
\left( 0, \frac{\partial}{\partial x_{2}}, 0, 0 \right), \ \ (d
\varphi) ( 0, x_{2}, x_{1}) \frac{\partial}{\partial x_{1}} =
\left(0, 0, \frac{\partial}{\partial x_{1}}, 0 \right).
\end{equation}

Clearly, $d \varphi (0, x_{2}, x_{1}) \{ \frac{\partial}{\partial
s}, \frac{\partial}{\partial x_{2}}, \frac{\partial}{\partial x_{1}}
\}$ is linear independent. Since $\textrm{dim}(L) = \textrm{dim}
([0,$ $\delta) \times (\Omega_{2} \cap \widetilde{\mathcal{M}}(p,r))
\times \tilde{S}_{r}^{-})$, by the Inverse Function Theorem, we have
that this lemma is true.
\end{proof}

\begin{proof}[Proof of Lemma \ref{orientation_3}]
Let $((0,x_{2}),(x_{1},0))$ be an arbitrary point in $\partial L$.
We only need to prove the orientation difference is
$(-1)^{\textrm{ind}(p) - \textrm{ind}(r)}$ at this point.

Suppose $(0, \frac{\partial}{\partial x_{2}}$) and
$(\frac{\partial}{\partial x_{1}}, 0)$ are a positive basis of
$T_{(0,x_{2})} \mathcal{M}(p,r)$ and $T_{(x_{1},0)} S_{r}^{-}$
respectively. We use the locally collar embedding $\varphi$ in Lemma
\ref{orientation_2}. Fix $x_{2}$ and $x_{1}$, change $s$. By
(\ref{orientation_2_2}), $d \varphi (0, x_{2}, x_{1})
\frac{\partial}{\partial x_{2}} = (0, \frac{\partial}{\partial
x_{2}}, 0, 0)$ and $d \varphi (0, x_{2}, x_{1})
\frac{\partial}{\partial x_{1}} = (0, 0, \frac{\partial}{\partial
x_{1}}, 0)$. So $\{ d \varphi (0, x_{2}, x_{1})
\frac{\partial}{\partial x_{2}}, d \varphi (0, x_{2}, x_{1})
\frac{\partial}{\partial x_{1}} \}$ is a positive basis of
$\mathcal{M}(p,r) \times S_{r}^{-}$. When
$\textrm{dim}(\mathcal{M}(p,r))=0$ or $\textrm{dim}(S_{r}^{-})=0$,
the orientation of $T_{(0,x_{2})} \mathcal{M}(p,r)$ or
$T_{(x_{1},0)} S_{r}^{-}$ is a sign $\pm 1$, and $d \varphi (s,
x_{2}, x_{1}) \frac{\partial}{\partial x_{2}}$ or $d \varphi (s,
x_{2}, x_{1}) \frac{\partial}{\partial x_{1}}$ is replaced by this
sign.

Now, $- (d \varphi) (0, x_{2}, x_{1}) \frac{\partial}{\partial s}$
is an outward normal vector of $\partial L$. Thus, when $s=0$, $\{ d
\varphi \frac{\partial}{\partial x_{2}},$ $d \varphi
\frac{\partial}{\partial x_{1}} \}$ is a positive base of $\partial
L$ if and only if $\{ -d \varphi \frac{\partial}{\partial s}, d
\varphi \frac{\partial}{\partial x_{2}}, d \varphi
\frac{\partial}{\partial x_{1}} \}$ is a positive base of $L$. This
is also equivalent to the statement that, when $s \neq 0$, $\{ -d
\varphi \frac{\partial}{\partial s}, d \varphi
\frac{\partial}{\partial x_{2}}, d \varphi \frac{\partial}{\partial
x_{1}} \}$ is a positive base of $L$.

When $s \neq 0$, $\varphi(s, x_{2}, x_{1}) \in L^{\circ}$, and
$\pi_{+}: L^{\circ} \longrightarrow S_{p}^{-}$ preserves
orientation. Thus, by (\ref{orientation_2_1}), the above
consideration is equivalent to the statement that,
\begin{eqnarray*}
& & \left \{ - d \pi_{+} \cdot d \varphi \frac{\partial}{\partial
s}, d \pi_{+} \cdot d \varphi \frac{\partial}{\partial x_{2}}, d
\pi_{+}
\cdot d \varphi \frac{\partial}{\partial x_{1}} \right \}  \\
& = & \left \{ -(x_{1}, d \theta \cdot x_{1}), \left( 0, d \theta
\frac{\partial}{\partial x_{2}} \right), \left( s
\frac{\partial}{\partial x_{1}}, s \cdot d \theta
\frac{\partial}{\partial x_{1}} \right) \right \}
\end{eqnarray*}
is a positive base of $S_{p}^{-}$. We change this base to another
base
\begin{equation}\label{orientation_3_1}
\left \{ -(x_{1}, d \theta \cdot x_{1}), \left( 0, d \theta
\frac{\partial}{\partial x_{2}} \right), \left(
\frac{\partial}{\partial x_{1}}, d \theta \frac{\partial}{\partial
x_{1}} \right) \right \}.
\end{equation}
The new base (\ref{orientation_3_1}) has the same orientation as the
old one. Its advantage is that, when $s=0$, (\ref{orientation_3_1})
is still a base of $S_{p}^{-}$. The reason is as follows. When $s
\neq 0$, (\ref{orientation_3_1}) is in $T S_{p}^{-}$. Thus, by
continuity, it is still in $T S_{p}^{-}$ when $s=0$. In addition,
when $s=0$, $(0, d \theta \frac{\partial}{\partial x_{2}}) = (0,
\frac{\partial}{\partial x_{2}})$. As a base of $T_{x_{1}} V_{-}$
and $T_{x_{2}} \widetilde{\mathcal{M}}(p,r)$ respectively, both $\{
-x_{1}, \frac{\partial}{\partial x_{1}} \}$ and
$\frac{\partial}{\partial x_{2}}$ are linearly independent. So
(\ref{orientation_3_1}) remains linearly independent when $s=0$.

When $s$ varies in $[0, \delta)$, the orientation difference between
$\{ -(x_{1}, d \theta \cdot x_{1}), (0, d \theta
\frac{\partial}{\partial x_{2}}),$ $(\frac{\partial}{\partial
x_{1}}, d \theta \frac{\partial}{\partial x_{1}}) \}$ and
$S_{p}^{-}$ is fixed. So we only need to check the difference when
$s=0$. As a base of $T_{x_{2}} \mathcal{M}(p,r)$, $(0,
\frac{\partial}{\partial x_{2}})$ contains $\textrm{ind}(p) -
\textrm{ind}(r) - 1$ vectors. Denote the orientation of a base $\{
* \}$ by $\textrm{Or} \{ * \}$. Then, when $s=0$,
\begin{eqnarray*}
& & \textrm{Or} \left \{ -(x_{1}, d \theta \cdot x_{1}), \left( 0, d
\theta \frac{\partial}{\partial x_{2}} \right), \left(
\frac{\partial}{\partial x_{1}}, d \theta \frac{\partial}{\partial x_{1}} \right) \right \} \\
& = & \textrm{Or} \left \{ -(x_{1}, d \theta \cdot x_{1}), \left( 0,
\frac{\partial}{\partial x_{2}} \right), \left(
\frac{\partial}{\partial x_{1}},
d \theta \frac{\partial}{\partial x_{1}} \right) \right \} \\
& = & (-1)^{\textrm{ind}(p) - \textrm{ind}(r)} \textrm{Or} \left \{
\left( 0, \frac{\partial}{\partial x_{2}} \right), (x_{1}, d \theta
\cdot x_{1}), \left( \frac{\partial}{\partial x_{1}}, d \theta
\frac{\partial}{\partial x_{1}} \right) \right \}.
\end{eqnarray*}
Since $(x_{1}, 0) = -\nabla f (x_{1}, 0)$,
$(\frac{\partial}{\partial x_{1}}, 0)$ is a positive base of
$T_{(x_{1},0)} S_{r}^{+}$, then $\{ (x_{1},0),$
$(\frac{\partial}{\partial x_{1}},0) \}$ is a positive base of
$T_{(x_{1},0)} (V_{-} \times \{0\}) = V_{-} \times \{0\} = T_{r}
\mathcal{D}(r)$. Thus $\{ (x_{1}, d \theta \cdot x_{1}),
(\frac{\partial}{\partial x_{1}}, d \theta \frac{\partial}{\partial
x_{1}}) \}$ represents a positive base of the normal space $N_{(0,
x_{2})} (\mathcal{M}(p,r),$ $S_{p}^{-})$. Since $(0,
\frac{\partial}{\partial x_{2}})$ is a positive base of
$T_{(0,x_{2})} \mathcal{M}(p,r)$, we infer that $\{ (0,
\frac{\partial}{\partial x_{2}}), (x_{1}, d \theta \cdot x_{1}),
(\frac{\partial}{\partial x_{1}}, d \theta \frac{\partial}{\partial
x_{1}}) \}$ is a positive base of $T_{(0,x_{2})} S_{p}^{-}$.

As a result, $(-1)^{\textrm{ind}(p) - \textrm{ind}(r)} \textrm{Or}
\{ d \varphi (0, x_{2}, x_{1}) \frac{\partial}{\partial x_{2}}, d
\varphi (0, x_{2}, x_{1}) \frac{\partial}{\partial x_{1}} \}$
represents the orientation of $\partial L$. This completes the
proof.
\end{proof}

%--------------------------------------------------------------------------------------------------------------------
\subsection{Proof of (2) of Theorem \ref{orientation}}
The proof of (2) is similar to that of (1). In particular, they
share many details. We shall only give the outline and the key
calculation of this proof.

\begin{proof}
We only need to prove that $\partial (\mathcal{D}(p) \sqcup
\mathcal{M}(p,r) \times \mathcal{D}(r)) = \mathcal{M}(p,r) \times
\mathcal{D}(r)$ as oriented manifolds. Actually, we only need to
argue this in an open subset containing $\mathcal{M}(p,r) \times
\mathcal{D}(r)$ of $\mathcal{D}(p) \sqcup \mathcal{M}(p,r) \times
\mathcal{D}(r)$. Recall the evaluation map $e: \mathcal{D}(p) \sqcup
\mathcal{M}(p,r) \times \mathcal{D}(r) \longrightarrow M$ in (3) of
Theorem \ref{d(p)_manifold}. We have $e^{-1} \circ f^{-1} ((-\infty,
f(r)+\epsilon))$ is such an open subset. Moreover, we can simplify
this problem again. Let $M(r) = f^{-1}((f(r)-\epsilon,
f(r)+\epsilon))$. Consider the open subset $e^{-1}(M(r))$. For all
$x \in e^{-1} \circ f^{-1} ((-\infty, f(r)+\epsilon)) \cap
\mathcal{M}(p,r) \times \mathcal{D}(r)$, there exist $y \in
e^{-1}(M(r)) \cap \mathcal{M}(p,r) \times \mathcal{D}(r)$ and a flow
map $\psi$ in $\overline{\mathcal{D}(p)}$, such that $\psi(y) = x$
(see Lemma \ref{pull_back_gradient}). From $y$ to $x$, $d \psi$
preserves the orientations of $\mathcal{D}(p)$ and $\mathcal{M}(p,r)
\times \mathcal{D}(r)$ and the outward normal direction. Then $d
\psi$ preserves the orientation difference between $\partial
(\mathcal{D}(p) \sqcup \mathcal{M}(p,r) \times \mathcal{D}(r))$ and
$\mathcal{M}(p,r) \times \mathcal{D}(r)$. Thus we only need to show
this is true in $e^{-1}(M(r))$. Now denote $\mathcal{D}(p) \cap
M(r)$ by $D_{p}$, $\mathcal{D}(r) \cap M(r)$ by $D_{r}$ and
$e^{-1}(M(r))$ by $\widehat{D_{p}}$. Then $\widehat{D_{p}} = D_{p}
\sqcup \mathcal{M}(p,r) \times D_{r}$. We only need to show that
$\partial \widehat{D_{p}} = \mathcal{M}(p,r) \times D_{r}$ as
oriented manifolds.

We use the same notation of $M^{\pm}$, $S_{p}^{-}$ and $S_{r}^{+}$
as in the proof of (1). Also identify $S_{p}^{-} \cap S_{r}^{+}$
with $\mathcal{M}(p,r)$ and define $\widetilde{\mathcal{M}}(p,r)$ as
in the proof of (1). Define $\widetilde{D}_{r} = \{ v_{2} \mid (0,
v_{2}) \in D_{r} \}$. We also assume that there is only one critical
point $r$ in $M(r)$.

Define
\[
L = \{ (x,y) \in S_{p}^{-} \times M(r) \mid \text{$x$ and $y$ are
connected by a generalized flow line.} \}.
\]
Then $\partial L = \mathcal{M}(p,r) \times D_{r}$. And
\[
L^{\circ} = \{ (x,y) \in L \mid \text{$x$ and $y$ are connected by a
unbroken flow line.} \},
\]
$L$ is identified with $\widehat{D_{p}}$ because $(x,y) \in L$ is a
pair of points on a generalized flow line connecting $p$ and $y$.
Since $L \subseteq S_{p}^{-} \times M(r)$, we may consider the
natural projection $\pi: L \longrightarrow M(r)$. Moreover, $\pi$
identifies $L^{\circ}$ with $D_{p}$, and $\pi$ coincides with the
above identification between $L$ and $\widehat{D_{p}}$. The
orientation of $\widehat{D_{p}}$ gives $L$ an orientation, and $L$
gives $\partial L$ a boundary orientation. We only need to check the
difference between the boundary orientation and the product
orientation of $\partial L$.

Fix $((0,x_{2}), (x_{1},0)) \in \mathcal{M}(p,r) \times D_{r}$. Just
as Lemma \ref{orientation_2}, we give a locally collar neighborhood
parametrization $\varphi: [0,\delta) \times (\Omega_{2} \cap
\widetilde{\mathcal{M}}(p,r)) \times \widetilde{D}_{r}
\longrightarrow V_{-} \times V_{+} \times V_{-} \times V_{+}$ such
that
\begin{equation}\label{orientation_4_1}
\varphi (s, v_{2}, v_{1}) = (s v_{1}, \theta(s v_{1}, v_{2}), v_{1},
s \theta(s v_{1}, v_{2})),
\end{equation}
where $\theta$ is defined in Lemma \ref{orientation_1}. It's
necessary to point out that this argument includes the special case
of $\textrm{ind}(r) = 0$. In this case, $\widetilde{D}_{r} = \{0\}$,
$\varphi (s, v_{2}, v_{1}) = (0, v_{2}, 0, s v_{2})$ and $d \varphi
\frac{\partial}{\partial x_{1}}$ is the sign $\pm 1$ assigned to
$D_{r}$.

Suppose $(0, \frac{\partial}{\partial x_{2}}$) and
$(\frac{\partial}{\partial x_{1}}, 0)$ are positive basis of
$T_{(0,x_{2})} \mathcal{M}(p,r)$ and $T_{(x_{1},0)} D_{r}$
respectively. At $(s, x_{2}, x_{1})$, we have
\begin{equation}\label{orientation_4_3}
d \varphi \frac{\partial}{\partial s} = (x_{1}, d \theta \cdot
x_{1}, 0, \theta + s \cdot d \theta \cdot x_{1}),
\end{equation}
\[
d \varphi \frac{\partial}{\partial x_{2}} = \left( 0, d \theta
\frac{\partial}{\partial x_{2}}, 0, s \cdot d \theta
\frac{\partial}{\partial x_{2}} \right), \ d \varphi
\frac{\partial}{\partial x_{1}} = \left( s \frac{\partial}{\partial
x_{1}}, s \cdot d \theta  \frac{\partial}{\partial x_{1}},
\frac{\partial}{\partial x_{1}}, s^{2} \cdot d \theta
\frac{\partial}{\partial x_{1}} \right).
\]
We shall check that, when $s \in [0, \delta)$, $\{ -d \varphi
\frac{\partial}{\partial s}, d \varphi \frac{\partial}{\partial
x_{2}}, d \varphi \frac{\partial}{\partial x_{1}} \}$ coincides with
the orientation of $L$.

When $s \neq 0$, $\varphi(s, x_{2}, x_{1}) \in L^{\circ}$. By the
definition of the orientation of $L^{\circ}$, $\pi: L^{\circ}
\longrightarrow D_{p}$ preserves its orientation. Thus, we only need
to show that, when $s \neq 0$,
\begin{eqnarray*}
& & \left \{- d\pi \cdot d \varphi \frac{\partial}{\partial s}, d\pi
\cdot d \varphi \frac{\partial}{\partial x_{2}}, d\pi \cdot d
\varphi \frac{\partial}{\partial x_{1}} \right \} \\
& = & \left \{ (0, -\theta - s \cdot d \theta \cdot x_{1}), \left(
0, s \cdot d \theta \frac{\partial}{\partial x_{2}} \right), \left(
\frac{\partial}{\partial x_{1}}, s^{2} \cdot d \theta
\frac{\partial}{\partial x_{1}} \right) \right \}
\end{eqnarray*}
gives the orientation of $D_{p}$ at $\pi \varphi (s, x_{2}, x_{1})$.
By (\ref{orientation_4_1}), we know that $\pi \varphi (s, x_{2},
x_{1})$ $= (x_{1}, s \theta(sx_{1}, x_{2}))$ is connected with $(s
x_{1}, \theta(s x_{1}, x_{2})) \in S_{p}^{-}$ by an unbroken flow
line. Consider the flow map $\psi$ in $U$ such that $\psi (v_{1},
v_{2}) = (s^{-1} v_{1}, s v_{2})$. Then $\psi (s x_{1}, \theta(s
x_{1}, x_{2})) = (x_{1}, s \theta(sx_{1}, x_{2}))$ and $\psi$
preserves the orientation of $D_{p}$. Thus we only need to check
that
\begin{eqnarray*}
& & \left \{- d \psi^{-1} \cdot d\pi \cdot d \varphi
\frac{\partial}{\partial s}, d \psi^{-1} \cdot d\pi \cdot d \varphi
\frac{\partial}{\partial x_{2}}, d \psi^{-1} \cdot d\pi \cdot d
\varphi \frac{\partial}{\partial x_{1}} \right \} \\
& = & \left \{ (0, -s^{-1} \theta - d \theta \cdot x_{1}), \left( 0,
d \theta \frac{\partial}{\partial x_{2}} \right), \left( s
\frac{\partial}{\partial x_{1}}, s \cdot d \theta
\frac{\partial}{\partial x_{1}} \right) \right \}
\end{eqnarray*}
gives the orientation of $D_{p}$ at $(s x_{1}, \theta(s x_{1},
x_{2}))$. Change the above base to the orientation equivalent base
$\{ (0, -\theta - s \cdot d \theta \cdot x_{1}), (0, d \theta
\frac{\partial}{\partial x_{2}}), (\frac{\partial}{\partial x_{1}},
d \theta \frac{\partial}{\partial x_{1}}) \}$. When $s=0$, it
becomes
\begin{equation}\label{orientation_4_2}
\left \{ (0, - x_{2}), \left( 0, \frac{\partial}{\partial x_{2}}
\right), \left( \frac{\partial}{\partial x_{1}}, d \theta
\frac{\partial}{\partial x_{1}} \right) \right \}.
\end{equation}
Since $(\frac{\partial}{\partial x_{1}},0)$ is a positive base of
$V_{-} \times \{0\} = T_{r}D_{r}$, $(\frac{\partial}{\partial
x_{1}}, d \theta \frac{\partial}{\partial x_{1}})$ represents a
positive base of $N_{(0,x_{2})} (\mathcal{M}(p,r), S_{p}^{-})$. At
$(0, x_{2})$, $(0, -x_{2}) = - \nabla f$, and $(0,
\frac{\partial}{\partial x_{2}})$ is a positive base of
$T_{(0,x_{2})} \mathcal{M}(p,r)$. Thus (\ref{orientation_4_2}) gives
the orientation of $D_{r}$.
\end{proof}

\begin{remark}\label{orientation_remark_latour}
It seems that, in the finite dimensional case, the paper
\cite{latour} gets orientation relations by the same strategy as
this paper has. The following key fact is pointed out without
explanations in \cite[p. 155]{latour}. ``La vari\'{e}t\'{e}
$W^{s}(c,\varepsilon) \times L_{A}(c,d)$ est de codimension $0$ dans
le bord de $\overline{W}^{s} (d, A+\varepsilon)$ et la normale
sortante $n_{0}$ \`{a} $\overline{W}^{s}(d, A+\varepsilon)$ en
$(c,l) \in W^{s}(c) \times L_{A}(c,d)$ s'identifie au vecteur
tangent \`{a} $l$ orient\'{e}e par $- \xi$." (Here $\xi = \nabla
f$.) This is proved in the paper by moving $-d \varphi
\frac{\partial}{\partial s}$ (see (\ref{orientation_4_3})) to be
$(0, -x_{2}) = - \nabla f$ in (\ref{orientation_4_2}). Thus our work
may give the details omitted in \cite{latour}.
\end{remark}

%--------------------------------------------------------------------------------------------------------------------
\subsection{Proof of (3) of Theorem \ref{orientation}}
The proof of (3) is a mixture of those of (1) and (2).
\begin{proof}
We shall prove that $\partial (\mathcal{W}(p,q) \sqcup
\mathcal{M}(p,r) \times \mathcal{W}(r,q)) = \mathcal{M}(p,r) \times
\mathcal{W}(r,q)$ and $\partial (\mathcal{W}(p,q) \sqcup
\mathcal{W}(p,r) \times \mathcal{M}(r,q)) =
(-1)^{\textrm{ind}(p)-\textrm{ind}(r)+1} \mathcal{W}(p,r) \times
\mathcal{M}(r,q)$. Recall the evaluation map $e: \mathcal{W}(p,q)
\sqcup \mathcal{M}(p,r) \times \mathcal{W}(r,q) \longrightarrow M$
(or $\mathcal{W}(p,q) \sqcup \mathcal{W}(p,r) \times
\mathcal{M}(r,q) \longrightarrow M$) in (3) of Theorem
\ref{w(p,q)_manifold}. Define $M(r) = f^{-1}((f(r)-\epsilon,
f(r)+\epsilon))$, $M(r)^{+} = f^{-1}((f(r), f(r)+\epsilon))$ and
$M(r)^{-} = f^{-1}((f(r)-\epsilon, f(r)))$. We have four cases. Just
as the proofs of (1) and (2), we will define a manifold $L$ which
plays a important role all through this proof, where
\[
L = \{ (x,y) \mid \text{$x$ and $y$ are connected by a generalized
flow line.} \},
\]
and $(x,y)$ is contained in some different manifolds in each case.
Also, $x$ and $y$ will be connected by a unbroken flow line if and
only if $(x,y) \in L^{\circ}$.

Case (a). The boundary is $\mathcal{M}(p,r) \times \mathcal{W}(r,q)$
and $r \neq q$.

We reduce this problem to considering the case of
$e^{-1}(M(r)^{-})$. Denote $e^{-1}(M(r)^{-})$ by
$\widehat{W_{p,q}}$, $\mathcal{W}(r,q) \cap M(r)^{-}$ by $W_{r,q}$,
$\mathcal{D}(p) \cap M(r)^{-}$ by $D_{p}$ and $\mathcal{D}(r) \cap
M(r)^{-}$ by $D_{r}$. Clearly, as unoriented manifolds, $\partial
\widehat{W_{p,q}} = \mathcal{M}(p,r) \times W_{r,q}$.

Define $L \subseteq S_{p}^{-} \times M(r)^{-}$. The natural
projection $\pi_{2}: L \longrightarrow M(r)^{-}$ identifies
$L^{\circ}$ with $D_{p}$. $\partial L = \mathcal{M}(p,r) \times
D_{r}$. The orientation of $D_{p}$ gives $L$ an orientation. In the
proof of (2), it has been verified that the boundary orientation and
the product orientation of $\partial L$ are the same. We identify
$\pi_{2}^{-1} (\mathcal{A}(q))$ with $\widehat{W_{p,q}}$ and
identify $(\pi_{2}|_{\partial L})^{-1} (\mathcal{A}(q))$ with
$\partial \widehat{W_{p,q}}$. An argument similar to that in (1)
completes the proof.

Case (b). The boundary is $\mathcal{M}(p,q) \times
\mathcal{W}(q,q)$.

Replace $M(r)^{-}$ by $M(q)$ in Case (a). The same argument gives a
proof.

Case (c). The boundary is $\mathcal{W}(p,r) \times \mathcal{M}(r,q)$
and $p \neq r$.

Reduce to the case of $e^{-1}(M(r)^{+})$. Denote $e^{-1}(M(r)^{+})$
by $\widehat{W_{p,q}}$, $\mathcal{W}(p,r) \cap M(r)^{+}$ by
$W_{p,r}$ and $\mathcal{D}(p) \cap M(r)^{+}$ by $D_{p}$.

Define $L \subseteq D_{p} \times M^{-}$, where $M^{-} = f^{-1}(f(r)
- \epsilon)$. The projection $\pi_{1}: L \longrightarrow D_{p}$
identifies $L^{\circ}$ with $D_{p} - W_{p,r}$, and $\partial L =
W_{p,r} \times S_{r}^{-}$. Then $D_{p}$ gives $L$ an orientation.
Consider another projection $\pi_{2}: L \longrightarrow M^{-}$. Then
$\pi_{2}^{-1} (S_{q}^{+})$ can be identified with
$\widehat{W_{p,q}}$ and $(\pi_{2}|_{\partial L})^{-1} (S_{q}^{+})$
can be identified with $\partial \widehat{W_{p,q}}$. We reduce the
proof to checking the difference of two orientations of $\partial
L$.

Define $\widetilde{W}_{p,r} = \{ v_{2} \mid (0, v_{2}) \in W_{p,r}
\}$. Similar to Lemma \ref{orientation_1} and \ref{orientation_2},
there is a neighborhood $\Omega_{2}$ of $x_{2}$ in
$\widetilde{W}_{p,r}$ and a parametrization $\tilde{\theta}:
B_{1}(\delta) \times \Omega_{2} \longrightarrow D_{p}$ such that
$\tilde{\theta} (v_{1}, v_{2}) = (v_{1}, \theta(v_{1}, v_{2}))$ and
$\theta(0, v_{2}) = v_{2}$. We also have a local collar embedding
$\varphi: [0,\delta) \times \Omega_{2} \times \widetilde{S}_{r}^{-}
\longrightarrow V_{-} \times V_{+} \times V_{-} \times V_{+}$ such
that
\begin{eqnarray*}
\varphi(s, v_{2}, v_{1}) & = & \left( s v_{1}, \theta(s v_{1},
v_{2}), (2 \epsilon)^{-\frac{1}{2}} (\epsilon + (\epsilon^{2} + 4
s^{2} \epsilon \|\theta(s v_{1},
v_{2})\|^{2})^{\frac{1}{2}})^{\frac{1}{2}}
v_{1}, \right. \\
& & \left. s (2 \epsilon)^{\frac{1}{2}} (\epsilon + (\epsilon^{2} +
4 s^{2} \epsilon \|\theta(s v_{1},
v_{2})\|^{2})^{\frac{1}{2}})^{-\frac{1}{2}} \theta(s v_{1}, v_{2})
\right).
\end{eqnarray*}

Just as the proof of (1), we reduce the proof to checking the
orientation of $\{ - d \pi_{1} \cdot d \varphi
\frac{\partial}{\partial s}, d \pi_{1} \cdot d \varphi
\frac{\partial}{\partial x_{2}}, d \pi_{1} \cdot d \varphi
\frac{\partial}{\partial x_{1}} \}$ and then that of $\{ -(x_{1}, d
\theta \cdot x_{1}), (0, \frac{\partial}{\partial x_{2}}),
(\frac{\partial}{\partial x_{1}}, d \theta \frac{\partial}{\partial
x_{1}}) \}$ in $T_{(0, x_{2})} D_{p}$. Here, $\{ (0,
\frac{\partial}{\partial x_{2}}) \}$ is a positive base of $T_{(0,
x_{2})}$ $W_{p,r}$. It contains $\textrm{ind}(p) - \textrm{ind}(r)$
vectors. Thus the orientations are
\begin{eqnarray*}
& & \textrm{Or} \left \{ -(x_{1}, d \theta \cdot x_{1}), \left( 0,
\frac{\partial}{\partial x_{2}} \right), \left(
\frac{\partial}{\partial x_{1}}, d \theta \frac{\partial}{\partial x_{1}} \right) \right \} \\
& = & (-1)^{\textrm{ind}(p) - \textrm{ind}(r) + 1} \textrm{Or} \left
\{ \left( 0, \frac{\partial}{\partial x_{2}} \right), \left( x_{1},
d \theta \cdot x_{1} \right), \left( \frac{\partial}{\partial
x_{1}}, d \theta \frac{\partial}{\partial x_{1}} \right) \right \}.
\end{eqnarray*}
Since $\{ (0, \frac{\partial}{\partial x_{2}}), (x_{1}, d \theta
\cdot x_{1}), (\frac{\partial}{\partial x_{1}}, d \theta
\frac{\partial}{\partial x_{1}}) \}$ is positive, the proof is
complete.

Case (d). The boundary is $\mathcal{W}(p,p) \times
\mathcal{M}(p,q)$.

Reduce to the case of $e^{-1}(M(p))$. Denote $e^{-1}(M(p))$ by
$\widehat{W_{p,q}}$ and $\mathcal{D}(p) \cap M(p)$ by $D_{p}$. Then
$\partial \widehat{W_{p,q}} = \mathcal{W}(p,p) \times
\mathcal{M}(p,q) = \{p\} \times \mathcal{M}(p,q)$.

Define $L \subseteq D_{p} \times M^{-}$, where $M^{-} =
f^{-1}(f(p)-\epsilon)$. Then $\pi_{1}: L \longrightarrow D_{p}$
identifies $L^{\circ}$ with $D_{p} - \{p\}$, and $\partial L =
\mathcal{W}(p,p) \times S_{p}^{-}$. Moreover, $D_{p}$ gives $L$ an
orientation. Consider $\pi_{2}: L \longrightarrow M^{-}$. Then
$\pi_{2}^{-1}(S_{q}^{+})$ can be identified with $\widehat{W_{p,q}}$
and $(\pi_{2}|_{\partial L})^{-1}(S_{q}^{+})$ can be identified with
$\partial \widehat{W_{p,q}}$. We reduce the proof to checking the
two orientations of $\partial L$.

Consider the collar embedding $\varphi: [0,\sqrt{2}) \times
\widetilde{S}_{p}^{-} \longrightarrow V_{-} \times V_{+} \times
V_{-} \times V_{+}$ such that $\varphi(s, v_{1}) = (s v_{1}, 0,
v_{1}, 0)$. Since $\mathcal{W}(p,p)$ has orientation $+1$, we only
need to check the orientation difference between $\{ - d \varphi
\frac{\partial}{\partial s}, d \varphi \frac{\partial}{\partial
x_{1}} \}$ and $L$. When $s=1$, $\textrm{Or} \{ - d \pi_{1} \cdot d
\varphi \frac{\partial}{\partial s}, d \pi_{1} \cdot d \varphi
\frac{\partial}{\partial x_{1}}  \} = -  \textrm{Or} \{ - \nabla f,
(\frac{\partial}{\partial x_{1}}, 0) \}$ is the negative orientation
of $T_{(x_{1},0)} D_{p}$. Thus $\partial \widehat{W_{p,q}} = -
\mathcal{W}(p,p) \times \mathcal{M}(p,q)$.
\end{proof}

\begin{remark}\label{cup_product}
The papers \cite{austin_braam} and \cite{viterbo} compute the cup
product of $H^{*}(M;R)$ via Morse Theory. Both
\cite[(2.2)]{austin_braam} and  \cite[lem.\ 2 and 3]{viterbo}
neglect signs. Theorem \ref{orientation}, (3), can tell us the the
signs if we do care about them. The following is an explanation of
\cite[lem.\ 3]{viterbo}. We shall use notation different from that
in \cite{viterbo}. Our $\mathcal{W}(p,q)$ and $\# \mathcal{M}(p,q)$
are $\mathcal{M}(p,q)$ and $n(p,q)$ in \cite{viterbo} respectively.
A real coefficients Thom-Smale cochain complex is defined in
\cite{viterbo} as $C^{*} = \bigoplus_{n}
\bigoplus_{\textrm{ind}(p)=n} R[p]$ with coboundary operator
\[
  \delta q = \sum_{\textrm{ind}(p) = \textrm{ind}(q) + 1} \# \mathcal{M}(p,q) p,
\]
where $\# \mathcal{M}(p,q)$ is defined in Theorem
\ref{boundary_operator}. Let $\omega$ be a differential form, in
\cite{viterbo}, a cup product action of $\omega$ on $C^{*}$ is
defined as
\[
  \pi(\omega) q = \sum_{p} \left( \int_{\mathcal{W}(p,q)} \omega \right) p.
\]

The paper \cite[lem.\ 3]{viterbo} states that $\pi(d \omega) =
\delta \pi (\omega) \pm \pi (\omega) \delta$. Actually, (3) of
Theorem \ref{orientation} tells us
\begin{equation}\label{cup_product_1}
\pi(d \omega) = \delta \pi (\omega) + (-1)^{|\omega| + 1} \pi
(\omega) \delta.
\end{equation}
If $\alpha$ and $\beta$ are two singular cochains, then $\delta
\alpha \cup \beta = \delta (\alpha \cup \beta) + (-1)^{|\alpha| + 1}
\alpha \cup \delta \beta$. By comparison with this,
(\ref{cup_product_1}) is reasonable. The proof of
(\ref{cup_product_1}) is as follows.

\begin{eqnarray*}
& & \pi(d \omega) q = \sum_{p} \left( \int_{\mathcal{W}(p,q)}
d \omega \right) p \\
& = & \sum_{p} \left( \int_{\overline{\mathcal{W}(p,q)}} e^{*} d
\omega \right) p = \sum_{p} \left( \int_{\partial^{1}
\overline{\mathcal{W}(p,q)}} e^{*} \omega \right)p \\
& = & \sum_{p} \left( \sum_{r} \int_{\mathcal{M}(p,r) \times
\mathcal{W}(r,q)} e^{*} \omega + \sum_{r} (-1)^{\textrm{ind}(p) -
\textrm{ind}(r) + 1} \int_{\mathcal{W}(p,r) \times \mathcal{M}(r,q)}
e^{*} \omega \right)p.
\end{eqnarray*}
Here $e$ is defined in (3) of Theorem \ref{w(p,q)_manifold}. When
$\textrm{dim}(\mathcal{W}(r,q)) < |\omega|$ (or
$\textrm{dim}(\mathcal{W}(p,r))$ $< |\omega|$), $e^{*} \omega = 0$
on $\mathcal{M}(p,r) \times \mathcal{W}(r,q)$ (or $\mathcal{W}(p,r)
\times \mathcal{M}(r,q)$). Thus
\begin{eqnarray*}
\pi(d \omega) q & = & \sum_{p} \left( \sum_{\textrm{ind}(r) =
\textrm{ind}(p)
-1} \# \mathcal{M}(p,r) \int_{\mathcal{W}(r,q)} \omega \right. \\
& & \left. + \sum_{\textrm{ind}(r) = \textrm{ind}(q) + 1}
(-1)^{\textrm{ind}(p) - \textrm{ind}(q)} \#
\mathcal{M}(r,q) \int_{\mathcal{W}(p,r)} \omega  \right)p \\
& = & \delta \pi (\omega) q + (-1)^{\textrm{ind}(p) -
\textrm{ind}(q)} \pi (\omega) \delta q.
\end{eqnarray*}
This completes the proof since $\textrm{ind}(p) - \textrm{ind}(q) =
|\omega| + 1$.
\end{remark}
%--------------------------------------------------------------------------------------------------------------------
%--------------------------------------------------------------------------------------------------------------------
\section{CW Structure}\label{section_CW_structure}

%--------------------------------------------------------------------------------------------------------------------
\subsection{Proof of Theorem \ref{disk}}
We present an elementary proof here. A non-elementary one is
sketched in Remark \ref{disk_non_elementary}.

Recall the evaluation map $e: \overline{\mathcal{D}(p)}
\longrightarrow M$ in (3) of Theorem \ref{d(p)_manifold}. We shall
``pull back" the vector field $- \nabla f$ on $M$ to
$\overline{\mathcal{D}(p)}$ via $e$. First, we need to explain the
definition of the pull back. We know $\overline{\mathcal{D}(p)} =
\bigsqcup_{I} \mathcal{D}_{I}$, where $I$ are critical sequences
with head $p$. The restriction of $e$ on $\mathcal{D}_{I} =
\mathcal{M}_{I} \times \mathcal{D}(r_{k})$ is the projection
$\mathcal{M}_{I} \times \mathcal{D}(r_{k}) \longrightarrow
\mathcal{D}(r_{k})$. For all $(\alpha, x) \in \mathcal{M}_{I} \times
\mathcal{D}(r_{k})$, $\{0\} \times T_{x} \mathcal{D}(r_{k})
\subseteq T_{\alpha} \mathcal{M}_{I} \times T_{x} \mathcal{D}(r_{k})
= T_{(\alpha, x)} (\mathcal{M}_{I} \times \mathcal{D}(r_{k}))$ and
the derivative of $e$ gives an isomorphism $d e: \{0\} \times T_{x}
\mathcal{D}(r_{k}) \longrightarrow T_{x} \mathcal{D}(r_{k})$. Thus
there is a unique vector $(0, - \nabla f) \in \{0\} \times T_{x}
\mathcal{D}(r_{k})$ such that $de(0, - \nabla f) = - \nabla f$. Then
$(0, - \nabla f(x)) \in T_{(\alpha, x)} (\mathcal{M}_{I} \times
\mathcal{D}(r_{k}))$ is the pull back of $- \nabla f(x)$.

\begin{lemma}\label{pull_back_gradient}
There is a smooth vector field $X$ on $\overline{\mathcal{D}(p)}$
such that $\forall (\alpha, z) \in \mathcal{M}_{I} \times
\mathcal{D}(r_{k})$, $X(\alpha, z) \in \{0\} \times T_{z}
\mathcal{D}(r_{k})$ and $de (X) = - \nabla f$.
\end{lemma}
\begin{proof}
Let $X$ be the pull back of $- \nabla f$ as explained above. We only
need to prove that $X$ is smooth.

Suppose the critical values in $(- \infty, f(p)]$ are exactly $f(p)
= c_{0} > c_{1} > \cdots > c_{l}$. Let $U(i) = e^{-1} \circ f^{-1}
((c_{i+1}, c_{i-1}))$, where $c_{-1} = + \infty$ and $c_{l+1} = -
\infty$. By Theorem \ref{d(p)_manifold}, each $U(i)$ is open and
$\bigcup_{i} U(i) = \overline{\mathcal{D}(p)}$, and we only need to
prove that $X$ is smooth in each $U(i)$. By Lemma
\ref{d(p)_embedding}, there is a smooth embedding $E(i): U(i)
\longrightarrow \prod_{j=0}^{i-1} f^{-1}(a_{j}) \times M(i)$, where
$a_{j} \in (c_{j+1}, c_{j})$ is a regular value and $M(i) = f^{-1}
((c_{i+1}, c_{i-1}))$. Define a vector field $\widehat{X} = (0,
\cdots, 0, - \nabla f) \in \prod_{j=0}^{i-1} T f^{-1}(a_{j}) \times
T M(i)$ on $\prod_{j=0}^{i-1} f^{-1}(a_{j}) \times M(i)$. Clearly,
$\widehat{X}$ is smooth. For brevity, denote $E(i)$ by $E$. We shall
prove that the restriction of $\widehat{X}$ on $E (U(i))$ is $X$.

Each $(\alpha, z) \in (\mathcal{M}_{I} \times \mathcal{D}(r_{k}))
\cap U(i)$ represents a pair $(\Gamma, z)$, where $\Gamma$ is a
generalized flow line connecting $p$ and $z$ (see
(\ref{comapctify_d})). Suppose $\Gamma = (\gamma_{0}, \cdots,
\gamma_{n})$, where $\gamma_{0} \equiv p$ and $\gamma_{n}(0) = z$.
Suppose the intersection of $\Gamma$ with $f^{-1}(a_{j})$ is
$z_{j}$. Then $\xi(t) = (z_{0}, \cdots, z_{i-1}, \gamma_{n}(t))$ is
a curve in $E (U(i)) \subseteq \prod_{j=0}^{i-1} f^{-1}(a_{j})
\times M(i)$ such that
\[
\xi'(0) = (0, \cdots, 0, - \nabla f) = \widehat{X}, \qquad de \cdot
\xi'(0) = - \nabla f.
\]
Moreover, since $\xi(t) \subseteq E( \{ \alpha \} \times
\mathcal{D}(r_{k}) )$, we infer $\xi'(0) \in d E (\{0\} \times T_{z}
\mathcal{D}(r_{k}))$. Identify $U(i)$ with $E (U(i))$, then
$\widehat{X} = \xi'(0) = X$ at $(\alpha, z)$. This completes the
proof.
\end{proof}

In the following, we use the terminology of \cite{douady}. It's easy
to see that Definition \ref{tangent_sector} is equivalent to the
\textit{secteur tangent} in \cite[p. 3]{douady}.
\begin{definition}\label{tangent_sector}
Suppose $L$ is a manifold with corners. For all $x \in L$,
\[
A_{x}L = \{ v \in T_{x}L \mid \text{$v = \gamma'(0)$ for some smooth
curve $\gamma:$ $[0,\epsilon) \longrightarrow L$.} \}
\]
is the tangent sector of $L$ at $x$.
\end{definition}

\begin{definition}\label{strict_outward}
Suppose $L$ is a manifold with corners, $\partial^{k} L$ is the
$k$-stratum ($k>0$) of $L$, $x \in \partial^{k} L$ and $v \in T_{x}
L$. $v$ is in the corner if $v \in T_{x} \partial^{k} L$. $v$ is
outward if $v \notin A_{x} L$. $v$ is strictly outward if $- v$ is
in the interior of $A_{x} L$.
\end{definition}

Clearly, strictly outward implies outward. We know that $A_{x} L$ is
linear isomorphic to $[0, +\infty)^{k} \times R^{n-k}$. Under this
isomorphism, $v$ is in the corner if and only if $v \in \{0\}^{k}
\times R^{n-k}$; $v$ is strictly outward if and only if $v \in (-
\infty, 0)^{k} \times R^{n-k}$. This does not depend on the
isomorphisms. It's easy to see the above vector field $X$ is in the
corner. We present the following easy lemma without proof.

\begin{lemma}\label{strict_outward_sum}
If both $v_{1}$ and $v_{2}$ are strictly outward, so are $v_{1} +
v_{2}$ and $l v_{1}$ for $l > 0$. If $v_{1}$ is strictly outward and
$v_{2}$ is in the corner, then $v_{1} + v_{2}$ is strictly outward.
\end{lemma}

The proof of the following lemma is in the Appendix.
\begin{lemma}\label{smoothness_norm}
Suppose $L$ is a manifold with corners, and $g: L \longrightarrow H$
is a smooth map where $H$ is a Hilbert space. If there exists a
smooth map $\tilde{g}: L \longrightarrow S(H)$ such that $g(x) =
\|g(x)\| \tilde{g}(x)$, then $\|g(x)\|$ is also smooth, where $S(H)$
is the unit sphere of $H$.
\end{lemma}

Let $\tilde{f} = f \circ e$ defined on $\overline{\mathcal{D}(p)}$
be the pull back of $f$, then $X \cdot \tilde{f} = - \| (\nabla f) e
\|^{2} \leq 0$.
\begin{lemma}\label{critical_boundary}
Suppose $x \in \overline{\mathcal{D}(p)}$ be such that $e(x)$ is a
critical point. Let $U_{x}$ be a neighborhood of $x$. Then there is
a smooth vector field $Y_{x}$ on $\overline{\mathcal{D}(p)}$ such
that its support $\text{\rm supp}(Y_{x}) \subseteq U_{x}$, $Y_{x}(x)
\neq 0$ and $Y_{x} \tilde{f} \leq 0$. In addition, for all $y \in
\partial \overline{\mathcal{D}(p)}$, $Y_{x}(y)$ is strictly outward
if $Y_{x}(y) \neq 0$.
\end{lemma}
\begin{proof}
Suppose $e(x) = r_{k}$ for some critical point $r_{k}$ and $x =
(\alpha, r_{k}) \in \mathcal{M}_{I} \times \mathcal{D}(r_{k})$,
where $I = \{ p, r_{1}, \cdots, r_{k} \}$. By Lemma
\ref{embed_d(p)_normal}, there exist a neighborhood $W_{1}$ of
$\alpha$ in $\mathcal{M}_{I}$, a neighborhood $W_{2}$ of $r_{k}$ in
$\mathcal{D}(r_{k})$, an $\epsilon > 0$ and a smooth embedding
$\varphi: W_{1} \times W_{2} \times [0, \epsilon)^{k}
\longrightarrow \overline{\mathcal{D}(p)}$ such that $\textrm{Im}
\varphi \subseteq U_{x}$, and $\varphi$ satisfies the stratum
condition in Lemma \ref{embed_d(p)_normal}.

By local triviality of the metric, choose a neighborhood $U$ of
$r_{k}$ as (\ref{localization_map}) such that
(\ref{localization_function}) and (\ref{localization_dymamics})
hold. We identify $U$ with $B$ by $h$ in (\ref{localization_map}).
We may assume $e(\textrm{Im} \varphi) \subseteq U$, and $W_{2}$ is a
neighborhood of $0$ in $V_{-}$. Identify $r_{k} \in
\mathcal{D}(r_{k})$ with $0 \in V_{-}$. The key part of the proof is
to show $\varphi$ can be modified so that
\begin{equation}\label{critical_boundary_1}
\tilde{f} \circ \varphi (\tilde{\alpha}, z, \rho_{I}, \sigma) =
f(r_{k}) - \frac{1}{2} \langle z, z \rangle + \frac{1}{2}
\sigma^{2},
\end{equation}
where $\tilde{\alpha} \in W_{1}$, $z \in W_{2}$, $\rho_{I} =
(\rho_{1}, \cdots, \rho_{k-1}) \in [0,\epsilon)^{k-1}$ and $\sigma
\in [0,\epsilon)$.

Denote $e \circ \varphi (\tilde{\alpha}, z, \rho_{I}, \sigma) =
(e_{1} (\tilde{\alpha}, z, \rho_{I}, \sigma), e_{2} (\tilde{\alpha},
z, \rho_{I}, \sigma)) \in V_{-} \times V_{+}$. Consider the map
$\theta: W_{1} \times W_{2} \times [0, \epsilon)^{k} \longrightarrow
W_{1} \times W_{2} \times [0, \epsilon)^{k}$ defined by
\[
\theta (\tilde{\alpha}, z, \rho_{I}, \sigma) = (\tilde{\alpha},
e_{1} (\tilde{\alpha}, z, \rho_{I}, \sigma), \rho_{I}, \|e_{2}
(\tilde{\alpha}, z, \rho_{I}, \sigma)\|).
\]

Firstly, we prove $\theta$ is smooth. It suffices to show
$\|e_{2}\|$ is smooth. Since $e_{2}$ is smooth, by Lemma
\ref{smoothness_norm}, we only need to find a smooth $\tilde{g}$
such that $e_{2} = \|e_{2}\| \tilde{g}$. By (\ref{comapctify_d}), an
element in $\overline{\mathcal{D}(p)}$ represents a pair $(\Gamma,
z)$, where $\Gamma$ is a generalized flow line connecting $p$ and $z
\in M$. Let $c = f(r_{k})$. Define $E: \overline{\mathcal{D}(p)}
\cap e^{-1} \circ f^{-1}((c - \epsilon, c + \epsilon))
\longrightarrow f^{-1}( c + \frac{\epsilon}{2}) \times M$ to be the
map $E(\Gamma, z) = (s(\Gamma), z)$, where $s(\Gamma)$ is the
intersection of $\Gamma$ with $f^{-1}( c + \frac{\epsilon}{2})$. By
Lemma \ref{d(p)_embedding}, $E$ is smooth. Furthermore, $E
\varphi(\tilde{\alpha}, z, \rho_{I}, \sigma) = ((\eta_{1},
\eta_{2}), (e_{1}, e_{2})) \in V_{-} \times V_{+} \times V_{-}
\times V_{+}$. By the stratum condition in Lemma
\ref{embed_d(p)_normal}, $e \varphi(\tilde{\alpha}, z, \rho_{I},
\sigma) \in \mathcal{D}(r_{k})$ or $e_{2}=0$ if and only if
$\sigma=0$. Thus, when $\sigma > 0$, $e_{2} \neq 0$ and $(e_{1},
e_{2})$ is connected with $(\eta_{1}, \eta_{2})$ by a unbroken flow
line. Thus $(e_{1}, e_{2}) = (\lambda^{-1} \eta_{1}, \lambda
\eta_{2})$ for some $\lambda > 0$ and $e_{2} / \|e_{2}\| = \eta_{2}
/ \|\eta_{2}\|$. However, $\eta_{2} \neq 0$ even if $\sigma = 0$.
Thus $\eta_{2} / \|\eta_{2}\|$ is smooth for all $\sigma \in
[0,\epsilon)$. Let $\tilde{g}(\tilde{\alpha}, z, \rho_{I}, \sigma) =
\eta_{2} / \|\eta_{2}\|$, then $e_{2} = \|e_{2}\| \tilde{g}$ for all
$\sigma \in [0,\epsilon)$. Thus $\|e_{2}\|$ is smooth.

Secondly, we prove that $\frac{\partial}{\partial \sigma} \|e_{2}\|
\neq 0$ at $(\alpha, 0, 0, 0)$. By the stratum condition, $d \varphi
\frac{\partial}{\partial \sigma}$ represents an inward normal vector
in $N_{(\alpha, r_{k})} ( \mathcal{M}_{I} \times \mathcal{D}(r_{k}),
\mathcal{M}_{J} \times \overline{\mathcal{D}(r_{k-1})} )$, where $J
= \{ p, r_{1}, \cdots, r_{k-1} \}$. Thus by Lemma \ref{d(p)_normal},
$0 \neq d e_{2} \frac{\partial}{\partial \sigma} \in V_{-}$. Denote
$d e_{2} \frac{\partial}{\partial \sigma}$ by $w$. Since
$e_{2}(\alpha, 0, 0, 0) = 0$, we see $e_{2}(\alpha, 0, 0, \sigma) =
\sigma w + O(\sigma^{2})$, and
\[
\frac{\partial}{\partial \sigma}|_{\sigma = 0} \|e_{2}\| =
\lim_{\sigma \rightarrow 0+} \frac{\| \sigma w + O(\sigma^{2})
\|}{\sigma} = \|w\| \neq 0.
\]

Thirdly, the Jacobian of $\theta$ at $(\alpha, 0, 0, 0)$ is
\[
\left(
\begin{array}{cccc}
\frac{\partial}{\partial \tilde{\alpha}} & 0 & 0 & 0 \\
0 & \frac{\partial}{\partial z} & d e_{1} \frac{\partial}{\partial
\rho_{I}} & d e_{2} \frac{\partial}{\partial \sigma} \\
0 & 0 & \frac{\partial}{\partial \rho_{I}} & 0 \\
0 & 0 & 0 & \frac{\partial}{\partial \sigma} \|e_{2}\|
\end{array}
\right).
\]
Since $\frac{\partial}{\partial \sigma} \|e_{2}\| \neq 0$, $d
\theta$ is nonsingular at $(\alpha, 0 , 0, 0)$.

Since $\| e_{2} \|$ is smooth, $\frac{\partial}{\partial
\sigma}|_{\sigma = 0} \|e_{2}\| \neq 0$, and $\| e_{2} \|$ vanishes
if and only if $\sigma = 0$, we can extend $\| e_{2} \|$ to be
defined on $W_{1} \times W_{2} \times (- \epsilon, \epsilon)^{k}$
such that $\| e_{2} \| < 0$ when $\sigma < 0$. By the Inverse
Function Theorem, shrinking $W_{1}$, $W_{2}$ and $\epsilon$
suitably, a smooth $\theta^{-1}$ can be defined in $W_{1} \times
W_{2} \times [0, \epsilon)^{k}$. Modify $\varphi$ to be $\varphi
\circ \theta^{-1}$ to get a smooth embedding $\varphi: W_{1} \times
W_{2} \times [0, \epsilon)^{k} \longrightarrow
\overline{\mathcal{D}(p)}$ such that $e \circ \varphi
(\tilde{\alpha}, z, \rho_{I}, \sigma) = (z, e_{2})$ and $\|e_{2}\| =
\sigma$. This gives (\ref{critical_boundary_1}).

Consider the vector field $\widetilde{Y} = \sum_{i=1}^{k-1}
(\rho_{i} - \epsilon) \frac{\partial}{\partial \rho_{i}} + (\sigma -
\epsilon) \frac{\partial}{\partial \sigma}$ in $W_{1} \times W_{2}
\times [0, \epsilon)^{k}$. It's strictly outward at corners,
$\widetilde{Y}(\varphi^{-1}(x)) \neq 0$ and $\widetilde{Y}(\tilde{f}
\circ \varphi) = (\sigma - \epsilon) \sigma \leq 0$.

By Lemma \ref{strict_outward_sum}, using the partition of the unity,
we can move $\widetilde{Y}$ to $\overline{\mathcal{D}(p)}$. This
defines the desired smooth vector field $Y_{x}$.
\end{proof}

\begin{lemma}\label{regular_boundary}
Suppose $x \in \overline{\mathcal{D}(p)}$ is such that $e(x)$ is a
regular point. Let $U_{x}$ be a neighborhood of $x$. Then there is a
smooth vector field $Y_{x}$ on $\overline{\mathcal{D}(p)}$ such that
its support $\text{\rm supp}(Y_{x}) \subseteq U_{x}$, $Y_{x}(x) \neq
0$ and $Y_{x} \tilde{f} = 0$. In addition, $\forall y \in
\partial \overline{\mathcal{D}(p)}$, $Y_{x}(y)$ is strictly outward
if $Y_{x}(y) \neq 0$.
\end{lemma}
\begin{proof}
Suppose $x \in \mathcal{M}_{I} \times \mathcal{D}(r_{k})$. By Lemma
\ref{embed_d(p)_normal}, there is a smooth embedding $\varphi: W
\times [0,\epsilon)^{k} \longrightarrow \overline{\mathcal{D}(p)}$
such that $\textrm{Im} \varphi \subseteq U_{x}$ where $W$ is a
neighborhood of $x$ in $\mathcal{M}_{I} \times \mathcal{D}(r_{k})$.
Since $e(x)$ is a regular point, $X \tilde{f}(x) = - \| \nabla
f(e(x)) \|^{2} < 0$. Shrinking $W$ and $\epsilon$ suitably, we may
assume $X \tilde{f} < 0$ in $\textrm{Im} \varphi$.

Denote the coordinates of $[0,\epsilon)^{k}$ by $(\rho_{1}, \cdots,
\rho_{k})$. Then $\sum_{i=1}^{k} (\rho_{i} - \epsilon)
\frac{\partial}{\partial \rho_{i}}$ defines a vector field on $W
\times [0,\epsilon)^{k}$ which is strictly outward at corners. Move
this one to $\textrm{Im} \varphi$ to get a strictly outward vector
field $Y_{1}$ on $\textrm{Im} \varphi$. Let $Y_{2} = Y_{1} -
\frac{Y_{1} \tilde{f}}{X \tilde{f}} X$. Then $Y_{2} \tilde{f} = 0$.
Since $Y_{1}$ is strictly outward, and $X$ is in the corner, we get,
by Lemma \ref{strict_outward_sum}, $Y_{2}$ is strictly outward and
$Y_{2}(x) \neq 0$. Using a partition of the unity, we get $Y_{x}$.
\end{proof}

As mentioned in Introduction, the following key lemma fulfills
Milnor's suggestion of adding a vector field to $X$.

\begin{lemma}\label{modified_vector}
Suppose $K \subseteq \mathcal{D}(p) \subseteq
\overline{\mathcal{D}(p)}$, $K$ is closed and $p$ is an interior
point of $K$. Then there is a smooth vector field $\widetilde{X}$ on
$\overline{\mathcal{D}(p)}$ such that $\widetilde{X} \tilde{f} \leq
X \tilde{f} = (-\nabla f) f$, $\widetilde{X}$ equals $X$ and
$-\nabla f$ on $K$, and $\widetilde{X}$ is strictly outward on
$\partial \overline{\mathcal{D}(p)}$.
\end{lemma}
\begin{proof}
Since $K$ is closed, $\overline{\mathcal{D}(p)} - K$ is open. Since
$K \subseteq \mathcal{D}(p)$, then $\overline{\mathcal{D}(p)} - K
\supseteq \partial \overline{\mathcal{D}(p)}$. Thus $\forall x \in
\partial \overline{\mathcal{D}(p)}$, by Lemmas
\ref{critical_boundary} and \ref{regular_boundary}, there is a
vector field $Y_{x}$ such that $\text{\rm supp}(Y_{x}) \subseteq
\overline{\mathcal{D}(p)} - K$ and satisfies the conclusions of
those lemmas. Define $W_{x} = \{y|Y_{x}(y) \neq 0 \}$, we have
$W_{x}$ is a neighborhood of $x$. Since $\partial
\overline{\mathcal{D}(p)}$ is compact, it can be covered by finite
many $W_{x_{i}}$ ($i=1, \cdots, n$). Let $Y = \sum_{i=1}^{n}
Y_{x_{i}}$. Since $Y_{x_{i}} \tilde{f} \leq 0$, we get $Y \tilde{f}
\leq 0$. Since $Y_{x_{i}}$ vanishes on $K$, so does $Y$. Also since
$\{ W_{x_{i}} \mid i=1, \cdots, n \}$ covers $\partial
\overline{\mathcal{D}(p)}$, and $Y_{x_{i}}$ is strictly outward if
it's nonzero, by Lemma \ref{strict_outward_sum}, we have that $Y$ is
strictly outward. Recall that $X$ is in the corner on $\partial
\overline{\mathcal{D}(p)}$. We complete the proof by defining
$\widetilde{X} = X + Y$.
\end{proof}

\begin{lemma}\label{to_boundary}
Let $\phi_{t}(x)$ be the flow line of $\widetilde{X}$ with initial
value $x$ and $x \neq p$. Then $\phi_{t}(x)$ reaches $\partial
\overline{\mathcal{D}(p)}$ at a unique time $0 \leq \omega(x) < +
\infty$. Furthermore, $\omega(x)$ is continuous with respect to $x$
in $\overline{\mathcal{D}(p)} - \{p\}$.
\end{lemma}
\begin{proof}
Above all, we prove the following claim: If $\phi_{t}(x)$ cannot
reach $\partial \overline{\mathcal{D}(p)}$ when $t \geq 0$, then
$\phi_{t}(x)$ exists for $t \in [0, +\infty)$.

If not, the maximal positive flow of $\phi_{t}(x)$ can only be
defined in $[0,s]$ or $[0,s)$, where $s < + \infty$. If the domain
is $[0,s]$, then $\phi_{s}(x) \in \partial
\overline{\mathcal{D}(p)}$. This is a contradiction. If the domain
is $[0,s)$, by the compactness of $\overline{\mathcal{D}(p)}$,
$\phi_{t}(x)$ has a cluster point $y_{0}$ when $t \rightarrow s$.
There are two cases. Case (1): $y_{0} \in \mathcal{D}(p)$. In this
case, there is a neighborhood $U_{y_{0}}$ of $y_{0}$ such that there
exists $\delta
> 0$ such that, for all $y \in U_{y_{0}}$, and for all $t \in (-\delta, +\delta)$,
$\phi_{t}(y)$ exists. Thus $\phi_{t}(x)$ can be defined in $[0, s +
\delta)$. This is a contradiction. Case (2): $y_{0} \in \partial
\overline{\mathcal{D}(p)}$. In this case, a neighborhood $U_{y_{0}}$
of $y_{0}$ is diffeomorphic to an open subset of $[0,+\infty)^{k}
\times R^{n-k}$ for some $k$ and $n$. The vector field in
$U_{y_{0}}$ can be smoothly extended to an open subset of $R^{n}$.
Then we may consider $y_{0}$ as an interior point. This converts the
argument to the first case. We can define $\phi_{t}(x)$ for $t \in
[0,s]$ with $\phi_{s}(x) = y_{0}$. This is also a contradiction.
This gives the claim.

Secondly, we prove that $\phi_{t}(x)$ reaches $\partial
\overline{\mathcal{D}(p)}$ at some time $0 \leq \omega(x) < +
\infty$ by contradiction.

Suppose $\phi_{t}(x)$ doesn't reach $\partial
\overline{\mathcal{D}(p)}$. By the claim, $\phi_{t}(x)$ exists for
$t \in [0, +\infty)$. By the assumption, $m = \inf_{M} f > -
\infty$. For all $y \in \overline{\mathcal{D}(p)}$, $\tilde{f}(y)
\leq \tilde{f}(p) = f(p)$. For all $T \geq 0$,
\begin{equation}\label{to_boundary_1}
\int_{0}^{T} \widetilde{X} \tilde{f}(\phi_{t}(x)) dt =
\tilde{f}(\phi_{T}(x)) - \tilde{f}(\phi_{0}(x)) \geq m - f(p) > -
\infty.
\end{equation}
Since $\widetilde{X} \tilde{f} \leq X \tilde{f} \leq 0$, then there
exists $\{t_{n}\} \subseteq [0, + \infty)$, $t_{n} \rightarrow +
\infty$ and $\widetilde{X} \tilde{f}(\phi_{t_{n}}(x)) \rightarrow
0$. Since $\overline{\mathcal{D}(p)}$ is compact, we may assume
$\phi_{t_{n}}(x) \rightarrow y_{0}$. Then $0 = \widetilde{X}
\tilde{f}(y_{0}) \leq X \tilde{f}(y_{0}) \leq 0$. Since $X
\tilde{f}(y_{0}) = - \| \nabla f (e(y_{0})) \|^{2}$, we see that
$e(y_{0})$ is a critical point. Thus $y_{0} \in \partial
\overline{\mathcal{D}(p)}$. Choose a neighborhood $U_{y_{0}}$ of
$y_{0}$ which is diffeomorphic to $[0,\epsilon)^{k} \times
B(0,\epsilon)$, where $B(0,\epsilon) = \{ v \in R^{n-k} \mid \|v\| <
\epsilon \}$ and $y_{0}$ is identified with $0 \in [0,\epsilon)^{k}
\times B(0,\epsilon)$. Identify $U_{y_{0}}$ with $[0,\epsilon)^{k}
\times B(0,\epsilon)$. We may assume $\widetilde{X}$ can be extended
smoothly to $(- \epsilon, \epsilon)^{k} \times B(0,\epsilon)$.
Denote the flow of the extended vector field by $\varphi_{t}$. Then
$\varphi_{t}(y_{0}) = \varphi_{t}(0) = t \widetilde{X}(0) +
O(t^{2})$. Since $\widetilde{X}(0)$ is outward, there exists
$\delta_{1} > 0$, such that for all $\delta \in (0,\delta_{1}]$,
$\varphi_{\delta}(0) \in (- \epsilon, \epsilon)^{k} \times
B(0,\epsilon) - [0,\epsilon)^{k} \times B(0,\epsilon)$. Fixing
$\delta$, there exists $\epsilon_{1} > 0$, for all $y \in
[0,\epsilon_{1})^{k} \times B(0,\epsilon_{1})$, $\varphi_{t}(y)$
exists for $t \in [-\delta, \delta]$ and $\varphi_{\delta}(y) \in (-
\epsilon, \epsilon)^{k} \times B(0,\epsilon) - [0,\epsilon)^{k}
\times B(0,\epsilon)$. Since $(0,\epsilon)^{k} \times B(0,\epsilon)$
and $(- \epsilon, \epsilon)^{k} \times B(0,\epsilon) -
[0,\epsilon)^{k} \times B(0,\epsilon)$ are disconnected, we have
$\varphi_{t_{0}}(y) \in [0,\epsilon)^{k} \times B(0,\epsilon) -
(0,\epsilon)^{k} \times B(0,\epsilon)$ at some time $t_{0} \in [0,
\delta)$. Since $\phi_{t_{n}}(x) \in [0,\epsilon_{1})^{k} \times
B(0,\epsilon_{1})$ for some $t_{n}$, we have $\phi_{t_{n}+t_{0}}(x)
\in
\partial \overline{\mathcal{D}(p)}$ for some $t_{0} \in [0,\delta)$.
This gives a contradiction.

Finally, we prove that $\omega(x)$ is unique and continuous.

Since $\widetilde{X}$ is outward, $\phi_{t}(x)$ does not exist after
it reaches $\partial \overline{\mathcal{D}(p)}$. Thus $\omega(x)$ is
unique. Denote $y_{0} = \phi_{\omega(x_{0})} (x_{0}) \in \partial
\overline{\mathcal{D}(p)}$, by the argument at the end of the second
step, we have, $\forall \delta > 0$, there is a neighborhood
$U_{y_{0}}$ of $y_{0}$ such that, for all $y \in U_{y_{0}}$,
$\phi_{t_{0}}(y) \in \partial \overline{\mathcal{D}(p)}$ for some
$t_{0} \in [0,\delta)$. Then there exist a neighborhood $U_{x_{0}}$
of $x_{0}$ and $\delta_{2} > 0$ such that, for all $x \in
U_{x_{0}}$, $\phi_{\omega(x_{0}) - \delta_{2}}(x)$ exists and is in
$U_{y_{0}}$. Thus $\omega(x) \leq \omega(x_{0}) + \delta$. Since
$\omega(x) \geq 0$, and $\omega(x) = 0$ when $x \in
\partial \overline{\mathcal{D}(p)}$, we get $\omega(x)$ is
continuous at $x_{0} \in \partial \overline{\mathcal{D}(p)}$. If
$x_{0} \in \mathcal{D}(p)$, then for all $\delta > 0$, there exists
$\delta_{2} \in (0, \delta)$, such that $\phi_{\omega(x_{0}) -
\delta_{2}}(x_{0})$ exists and is in $\mathcal{D}(p)$. Also, there
exists $U_{x_{0}}$ such that, for all $x \in U_{x_{0}}$,
$\phi_{\omega(x_{0}) - \delta_{2}}(x)$ exists and is in
$\mathcal{D}(p)$. Thus $\omega(x) \geq \omega(x_{0}) - \delta$. We
have now proved $\omega(x)$ is continuous in general.
\end{proof}

Actually, the above lemma only requires $\widetilde{X}$ to be
outward. However, the following one requires $\widetilde{X}$ to be
strictly outward.

\begin{lemma}\label{to_center}
Let $\phi_{t}(x)$ be the flow line of $\widetilde{X}$ with initial
value $x$. Then $\phi_{t}(x)$ exists for $t \in (-\infty, 0]$ and
$\displaystyle \lim_{t \rightarrow - \infty} \phi_{t}(x) = p$.
\end{lemma}
\begin{proof}
Firstly, we prove that $\phi_{t}(x)$ exists for $t \in (-\infty, 0]$
by contradiction. If not, the maximal negative flow can only be
defined for $[s,0]$ or $(s,0]$, where $s > -\infty$.

Suppose the domain is $[s,0]$. If $\phi_{s}(x) \in \mathcal{D}(p)$,
then $\phi_{t}(x)$ can be defined in $(s - \delta, 0]$ for some
$\delta > 0$. This is a contradiction. Suppose $\phi_{s}(x) = x_{0}
\in \partial \overline{\mathcal{D}(p)}$. Like the proof of Lemma
\ref{to_boundary}, a neighborhood of $x_{0}$ is identified with
$[0,\epsilon)^{k} \times B(0,\epsilon)$ and $x_{0}$ is identified
with $0$. Extend the vector field in $[0,\epsilon)^{k} \times
B(0,\epsilon)$ smoothly to be defined in $(- \epsilon, \epsilon)^{k}
\times B(0,\epsilon)$. Denote the flow of the extended vector field
by $\varphi_{t}$. Since $\widetilde{X}(0) = \widetilde{X}(x_{0})$ is
strictly outward, then $-\widetilde{X}(0) \in (0, +\infty)^{k}
\times R^{n-k}$. Since $\varphi_{t}(x_{0}) = \varphi_{t}(0) = t
\widetilde{X}(0) + O(t^{2})$, there exists $\delta > 0$ such that,
for all $t \in [-\delta,0]$, we have $\varphi_{t}(0) \in
[0,\epsilon)^{k} \times B(0,\epsilon)$. Thus $\phi_{t}(x)$ exists
for $t \in [s - \delta, 0]$. This gives a contradiction.

Suppose the domain is $(s,0]$. Using the same argument as in the
proof of Lemma \ref{to_boundary}, we can extend the domain to be
$[s,0]$. This gives a contradiction.

As a result, we proved the first assertion.

Secondly, we prove by contradiction that $\phi_{t}(x)$ has no
cluster point in $\partial \overline{\mathcal{D}(p)}$ when $t
\rightarrow - \infty$.

Suppose $\phi_{t}(x)$ has a cluster point $x_{0} \in \partial
\overline{\mathcal{D}(p)}$. By the continuity of $\omega(x)$ in
Lemma \ref{to_boundary}, there exists a neighborhood $U_{x_{0}}$ of
$x_{0}$ such that, for all $x \in U_{x_{0}}$, we have $\omega(x) \in
[0,1)$. Since $x_{0}$ is a cluster point, there exist $T < -1$, and
$\phi_{T}(x) \in U_{x_{0}}$. Thus $\phi_{T + t_{0}}(x) \in \partial
\overline{\mathcal{D}(p)}$ for some $t_{0} \in [0,1)$. Then
$\phi_{t}(x)$ does not exist when $t > T + t_{0}$. In particular,
$\phi_{t}(x)$ does not exist when $t = 0$. This gives a
contradiction.

Thirdly, we prove by contradiction that $\phi_{t}(x)$ has no cluster
point in $\mathcal{D}(p) - \{p\}$ when $t \rightarrow - \infty$.

Suppose $x_{0} \in \mathcal{D}(p) - \{p\}$ is a cluster point.
Clearly, $\widetilde{X} \tilde{f}(x_{0}) \leq X f(x_{0}) = -
\|\nabla f (e(x_{0}))\|^{2} = A < 0$. Thus there exists a
neighborhood $U_{x_{0}}$ of $x_{0}$, a $\delta > 0$, for all $x \in
U_{x_{0}}$, such that $\phi_{t}(x)$ exists for $t \in [-\delta,
\delta]$ and $\widetilde{X} \tilde{f}(\phi_{t}(x)) \leq \frac{A}{2}$
in this interval. Since $x_{0}$ is a cluster point, there exists $\{
t_{n} \} \subseteq (-\infty, 0]$ such that $t_{n+1} < t_{n} -
\delta$ and $\phi_{t_{n}}(x) \in U_{x_{0}}$. Then
\[
\int_{-\infty}^{0} \widetilde{X} \tilde{f}(\phi_{t}(x)) dt \leq
\sum_{n=1}^{\infty} \int_{t_{n} - \delta}^{t_{n}} \widetilde{X}
\tilde{f}(\phi_{t}(x)) dt \leq \sum_{n=1}^{\infty} \int_{t_{n} -
\delta}^{t_{n}} \frac{A}{2} = - \infty.
\]
On the other hand, similar to (\ref{to_boundary_1}), we have for all
$T < 0$, $\int_{T}^{0} \widetilde{X} \tilde{f}(\phi_{t}(x)) dt \geq
\tilde{f}(x) - \tilde{f}(p) > -\infty$. This gives a contradiction.

Finally, since $\overline{\mathcal{D}(p)}$ is compact, $\forall
\{t_{n}\} \subseteq (-\infty, 0]$, there must be a cluster point of
$\phi_{t_{n}}(x)$. Thus $\phi_{t}(x) \rightarrow p$ when $t
\rightarrow - \infty$.
\end{proof}

Now we are ready to prove Theorem \ref{disk}. The idea of this proof
is as follows. Choose a closed neighborhood $K$ of $p$ in
$\overline{\mathcal{D}(p)}$ which is diffeomorphic to
$D^{\textrm{ind}(p)}$. The flow line $\phi_{t}$ of the above
$\widetilde{X}$ expands $K$ homeomorphically onto
$\overline{\mathcal{D}(p)}$. We also explain this idea by the
previous example on $T^{2}$. The flow generated by $X$ on
$\overline{\mathcal{D}(p)}$ is as the right part of Figure
\ref{dp_figure}. The flow generated by $\widetilde{X}$ is
illustrated by Figure \ref{flow_figure}.

\begin{figure}[!htbp]
\centering
\includegraphics[scale=0.3]{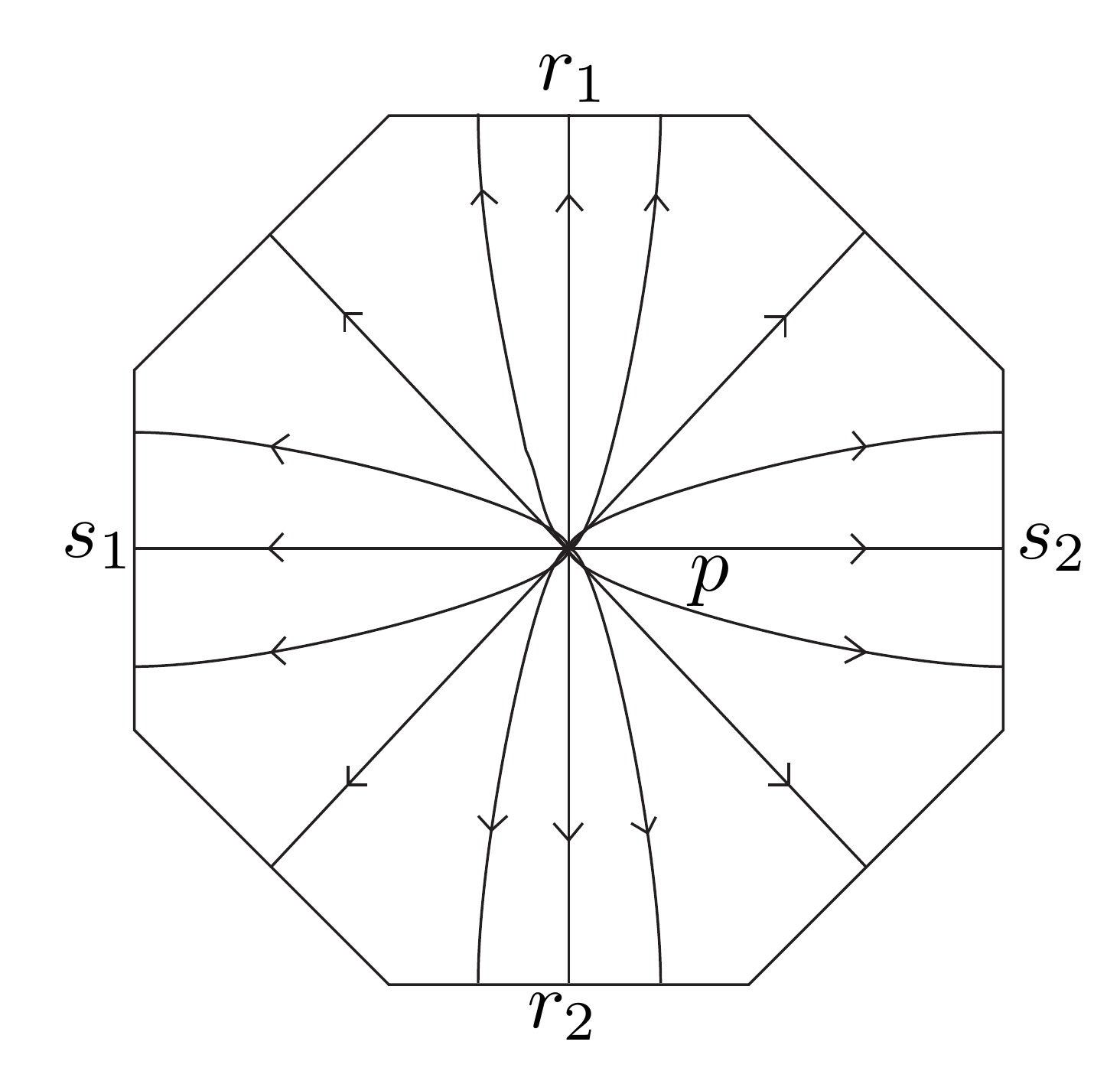} \caption{Flow Generated by $\widetilde{X}$}
\label{flow_figure}
\end{figure}

\begin{proof}[Proof of Theorem \ref{disk}]
Choose a closed neighborhood $K$ of $p$ in
$\overline{\mathcal{D}(p)}$ satisfying the following two properties:
(1). $K \subseteq \mathcal{D}(p)$. (2). There is a diffeomorphism
$\theta: D(\epsilon) \longrightarrow K$ such that $\theta(0) = p$,
$\tilde{f} \circ \theta(v) = f(p) - \frac{1}{2} \langle v, v
\rangle$ and $((d \theta)^{-1} X)(v) = v$, where $D(\epsilon) = \{ v
\in R^{\textrm{ind}(p)} \mid \|v\| \leq \epsilon \}$.

We only need to construct a homeomorphism $\Psi: (D(\epsilon),
S(\epsilon)) \longrightarrow (\overline{\mathcal{D}(p)}, \partial
\overline{\mathcal{D}(p)})$, where $S(\epsilon) = \partial
D(\epsilon)$.

By Lemmas \ref{modified_vector}, \ref{to_boundary} and
\ref{to_center}, there is a vector field $\widetilde{X}$ on
$\overline{\mathcal{D}(p)}$ satisfying the following four
properties: (1). We have $\widetilde{X} = X$ in $K$. (2). We have
$\widetilde{X} \tilde{f} < 0$ in $\mathcal{D}(p)-\{p\}$. (3). The
flow $\phi_{t}(x)$ generated by $\widetilde{X}$ reaches the boundary
at a unique time $\omega(x) \in [0,+\infty)$ when $x \neq p$, and
$\omega(x)$ is continuous in $\overline{\mathcal{D}(p)} - \{p\}$.
(4). For all $x$, $\phi_{t}(x) \rightarrow p$ when $t \rightarrow -
\infty$.

Denote $\varphi_{t}$ the flow generated by the vector field $Z(v)=v$
on $D(\epsilon)$. Then $\theta (\varphi_{t}(v)) = \phi_{t}
(\theta(v))$.

Define $\beta(s)$ in $[0,\epsilon]$ to be
\[
\beta(s) =
\begin{cases}
0 & t \in [0, \frac{\epsilon}{2}], \\
\frac{2t - \epsilon}{\epsilon} & t \in [\frac{\epsilon}{2},
\epsilon].
\end{cases}
\]

Define $\Psi: D(\epsilon) \longrightarrow \overline{\mathcal{D}(p)}$
to be
\[
\Psi(v) =
\begin{cases}
\theta(v) & \|v\| \in [0, \frac{\epsilon}{2}], \\
\phi[\omega[\theta(\frac{v}{\|v\|} \epsilon)] \beta(\|v\|),
\theta(v)] & \|v\| \in [\frac{\epsilon}{2}, \epsilon].
\end{cases}
\]
Here we use the notation $\phi(t,x) = \phi_{t}(x)$.

Firstly, $\Psi$ is continuous, $\Psi(S(\epsilon)) \subseteq \partial
\overline{\mathcal{D}(p)}$ and $\Psi^{-1}(\partial
\overline{\mathcal{D}(p)}) \subseteq S(\epsilon)$.

Secondly, we prove that $\Psi$ is injective. Consider the orbits of
the flows. The orbits in $D(\epsilon)$ are $\{0\}$ and $\{ s v \mid
\|v\| = \epsilon, s \in (0,1] \}$. We have $\Psi(0) = p$ and
$\Psi(sv) = \phi(l(s,v), \theta(v))$, where $l(s,v) =
\omega(\theta(v)) \beta(s \epsilon) + \log s$ and $\|v\| =
\epsilon$. When $\|v\| \equiv \epsilon$, $\tilde{f} (\theta(v))
\equiv \frac{1}{2} f(p) - \frac{1}{2} \epsilon^{2}$, by the above
property (2) of $\widetilde{X}$, we have $\Psi$ maps distinct orbits
to distinct orbits. Since $l(s,v)$ is a strictly increasing function
with respect to $s$, by the above property (2) of $\widetilde{X}$
again, we have $\Psi$ is injective.

Thirdly, we prove that $\Psi$ is surjective. Clearly, $\Psi(0)=p$.
For all $x \in \overline{\mathcal{D}(p)}$ and $x \neq p$, since
$\phi_{t}(x) \rightarrow p$ when $t \rightarrow - \infty$, we have
that there exist $t_{0} \in R$ and $v_{0} \in D(\epsilon)$ such that
$\phi_{-t_{0}}(x) \in K$, $\| v_{0} \| = \epsilon$, and $v_{0} =
\theta^{-1}(\phi_{-t_{0}}(x))$. Then $t_{0} \leq
\omega(\theta(v_{0}))$. Since $l(s,v_{0}) \rightarrow - \infty$ when
$s \rightarrow 0$ and $l(s,v)$ is continuous, the range of
$l(s,v_{0})$ is $(-\infty, \omega(\theta(v_{0}))]$. Then there
exists $s_{0}$ such that $l(s_{0}, v_{0}) = t_{0}$. Thus $\Psi(s_{0}
v_{0}) = x$. Therefore, $\Psi$ is surjective.

Finally, $\Psi$ is a map from a compact space to a Hausdorff space,
so $\Psi$ is a homeomorphism.
\end{proof}

\begin{remark}\label{disk_non_elementary}
The following is a quick but non-elementary proof of Theorem
\ref{disk}, which is based on the Poincar\'{e} Conjecture in all
dimensions. Clearly, $\overline{\mathcal{D}(p)}$ is a compact
topological manifold with boundary whose interior is an open disk.
We can prove that $\partial \overline{\mathcal{D}(p)}$ is a homotopy
sphere. By the Poincar\'{e} Conjecture, $\partial
\overline{\mathcal{D}(p)}$ is a topological sphere. Consider a
collar embedding $\partial \overline{\mathcal{D}(p)} \times [0,1]
\longrightarrow \overline{\mathcal{D}(p)}$ which identifies
$\partial \overline{\mathcal{D}(p)} \times \{0\} $ with $\partial
\overline{\mathcal{D}(p)}$. By the Generalized Schoenflies Theorem
(see \cite[thm. 5]{brown}), we can prove that
$\overline{\mathcal{D}(p)} -
\partial \overline{\mathcal{D}(p)} \times [0,\frac{1}{2})$ is a closed disk.
This completes the proof.
\end{remark}

%--------------------------------------------------------------------------------------------------
\subsection{Proof of Theorem \ref{CW}}
As mentioned in Subsection \ref{subsection_cw_structure}, the CW
complex structure of $K^{a}$ immediately results from Theorems
\ref{d(p)_manifold} and \ref{disk}. We only need to prove that, when
$f$ is proper, $M^{a}$ has the desired CW decomposition. By Theorem
\ref{disk}, we can always construct a CW decomposition from a good
vector field. The key part of this proof is to find a good vector
field for $M^{a}$ (see Lemma \ref{symmetry_dynamic}). This is
heavily based on Milnor's dealing with gradient-like dynamics in
\cite{milnor2}.

\begin{definition}\label{gradient-like}
Suppose $f$ is a Morse function on a Hilbert manifold. A vector
field $X$ is a gradient-like vector field of $f$ if $X= \nabla f$
near each critical point of $f$ and $Xf>0$ at each regular point of
$f$.
\end{definition}

\begin{remark}
Some papers in the literature include the local triviality of $X$
into the definition of a gradient-like vector field. We follow the
style of \cite{smale1} and exclude it.
\end{remark}

Up until now, we haven't assumed that $M^{a}$ is compact and we have
considered only negative gradient dynamics. In this subsection, we
take $M^{a}$ to be compact because we take $f$ to be proper. The
results proved before this subsection still hold for negative
gradient-like dynamics when the underlying manifold is compact.
There are two reasons. Both are sufficient. Firstly, on the compact
manifold, Smale points out in \cite[remark after thm.\ B]{smale1}
that all gradient-like dynamics are actually gradient dynamics (this
is even true on a Hilbert manifold, see Lemma
\ref{gradient-like_gradient}), and $(M,f)$ is a CF pair
automatically. Secondly, we can formally replace ``gradient" by
``gradient-like" in the above proofs when $M$ is compact.

The proof of the following lemma is in the Appendix.

\begin{lemma}\label{gradient-like_gradient}
If $X$ is a gradient-like vector field of a Morse function $f$ on a
Hilbert manifold $M$, then there is a metric on $M$ such that this
metric equals the original one associated with $X$ near each
critical point of $f$ and $\nabla f = X$ for this metric.
\end{lemma}

Since $a$ is a regular value of $f$ and $f$ is proper, $M^{a}$ is a
compact manifold with boundary $f^{-1}(a)$. There is a smooth collar
embedding $\varphi: [0, \epsilon_{0}) \times \partial M^{a}
\longrightarrow M^{a}$ such that $f \circ \varphi (s,x) = a - s$.
Clearly, all critical points of $f$ are in $M^{a} - \textrm{Im}
\varphi$. Double $M^{a}$ to be a compact manifold $2M^{a}$ without
boundary such that the above $\varphi$ can be extended in the
obvious way to a smooth embedding $\varphi: (-\epsilon_{0},
\epsilon_{0}) \times \partial M^{a} \longrightarrow 2M^{a}$.

For convenience, we identify $(-\epsilon_{0}, \epsilon_{0}) \times
\partial M^{a}$ with $\textrm{Im} \varphi$ from now on.

There is an evident $Z_{2}$-symmetry group acting on $2M^{a}$. For
all $x \in M^{a} \subseteq 2M^{a}$, denote $\bar{x} \in 2M^{a}$ the
copy of $x$. Define $\sigma: 2M^{a} \longrightarrow 2M^{a}$ by
$\sigma(x) = \bar{x}$ and $\sigma(\bar{x}) = x$. Then
\begin{equation}
Z_{2} = \{ \textrm{Id}, \sigma \}
\end{equation}
is the group. By the smooth structure of $2M^{a}$, $Z_{2}$ acts
smoothly. The set of fixed points of $Z_{2}$ is $\textrm{Fix}(Z_{2};
2M^{a}) = \partial M^{a}$.

We omit the proof of the following, which is straightforward.

\begin{lemma}\label{symmetry_function}
There exists a Morse Function $F$ on $2M^{a}$ satisfying the
following properties. (1). It is invariant under the $Z_{2}$ action.
(2). It equals $f$ on $M^{a} - \text{\rm Im} \varphi$. (3). We have
$F(s,x) = a - \frac{1}{2} s^{2} + g(x)$ in $(-\delta, \delta) \times
\partial M^{a}$ for some $\delta \in (0, \epsilon_{0})$, and $g$ (and then $F|_{\partial M^{a}}$) is
a Morse function on $\partial M^{a}$. (4). The critical points of
$F$ are exactly the critical points of $f$ (which are in $M^{a} -
\text{\rm Im} \varphi$) together with their images under the $Z_{2}$
action, and the critical points of $g$. (5) The function values of
$F$ on $\partial M^{a}$ are greater than the function values at
critical points off $\partial M^{a}$.
\end{lemma}

We can define a metric $G$ on $2M^{a}$ satisfying the following
properties. (1). It is invariant under the $Z_{2}$ action. (2). It
equals the original metric on $M^{a} - \textrm{Im} \varphi$. (3). It
is a product metric on $(-\delta, \delta) \times \partial M^{a}$,
where $(-\delta, \delta)$ is given the standard metric. (4). It is
locally trivial.

\begin{lemma}\label{symmetry_dynamic}
There is a negative gradient-like vector field $\xi$ of $F$ on
$2M^{a}$ satisfying the following properties. (1). The vector field
$\xi$ is invariant under the $Z_{2}$ action. (2). It equals $-
\nabla f$ on $M^{a} - \textrm{Im} \varphi$. (3). It satisfies local
triviality and transversality. (4). For all $x \in \partial M^{a}$,
$\xi(x) \in T_{x} \partial M^{a}$, and $\xi|_{\partial M^{a}}$ is a
negative gradient-like vector field of $F|_{\partial M^{a}}$ on
$\partial M^{a}$ satisfying local triviality and transversality.
\end{lemma}
\begin{proof}
We shall modify $- \nabla F$ to be $\xi$. The proof follows closely
those of \cite[thm.\ 4.4, lem.\ 4.6 and thm.\ 5.2]{milnor2} plus
arguing in the $Z_{2}$ invariant setting. The book \cite{milnor2}
uses gradient-like vector fields, we use negative ones.

Clearly, if $\xi$ is $Z_{2}$ invariant, then $\xi(x) \in T_{x}
\partial M^{a}$ for all $x \in \partial M^{a}$. Since both $F$ and
the metric on $2M^{a}$ are $Z_{2}$ invariant, so is $- \nabla F$. By
the constructions of $F$ and the metric, $- \nabla F$ and $- \nabla
F|_{\partial M^{a}}$ satisfy everything but transversality.

Suppose the critical points on $\partial M^{a}$ have function values
$c_{1} < \cdots < c_{l}$. Suppose $c_{0}$ is the maximum of function
values on critical points off $\partial M^{a}$. By (5) of Lemma
\ref{symmetry_function}, $c_{0} < c_{1}$. By induction on $k$, we
shall modify the vector field $\xi$ on $M^{a_{k},b_{k}}$ for some
$a_{k}$, $b_{k} \in (c_{k-1}, c_{k})$ such that the vector field on
$M^{c_{k}}$ satisfies the conclusion (in $M^{c_{k}}$, we don't
consider $\mathcal{D}(p) \cap M^{c_{k}}$ for $p \notin M^{c_{k}}$),
and the vector field globally satisfies everything but
transversality.

Firstly, by (2) and (4) of Lemma \ref{symmetry_function} and the
construction of the metric, the vector field on $M^{c_{0}}$
satisfies the conclusion automatically.

Secondly, supposing we have finished the construction for
$M^{c_{k-1}}$, we shall modify $\xi$ for $M^{c_{k}}$ by the method
in \cite{milnor2}. Suppose the critical points with function value
$c_{k}$ are exactly $p_{i}$ ($i=1, \cdots, n$). Denote the
descending and ascending manifolds of $p$  in $\partial M^{a}$ with
respect to $\xi|_{\partial M^{a}}$ by $\widetilde{\mathcal{D}(p)}$
and $\widetilde{\mathcal{A}(p)}$ respectively.

By (3) of Lemma \ref{symmetry_function} and local triviality of
$\xi$, there is a neighborhood $U_{i}$ of $p_{i}$ such that $U_{i}$
has a coordinate chart $(s, v_{1}, v_{2})$, $s^{2} < 4\epsilon$,
$\|v_{1}\|^{2} < 4\epsilon$, $\|v_{2}\|^{2} < 4\epsilon$, $F(s,
v_{1}, v_{2}) = c_{k} - \frac{1}{2} s^{2} - \frac{1}{2}
\|v_{1}\|^{2} + \frac{1}{2} \|v_{2}\|^{2}$, the metric on $U_{i}$ is
standard, and the action of $\sigma$ is $\sigma(s, v_{1}, v_{2}) =
(-s, v_{1}, v_{2})$. Here $\epsilon$ is uniform for all $i$. We may
assume $U_{i}$ are disjoint for different $i$. Then
$\mathcal{D}(p_{i}) \cap U_{i} = \{(s, v_{1}, 0)\}$ and
$\widetilde{\mathcal{D}(p_{i})} \cap U_{i} = \{(0, v_{1}, 0)\}$.
Denote $S_{i}^{-} = \mathcal{D}(p_{i}) \cap F^{-1}(c_{k} - \epsilon)
= \{(s, v_{1}, 0) \mid s^{2} + \|v_{1}\|^{2} = 2 \epsilon \}$ and
$\widetilde{S_{i}^{-}} = \widetilde{\mathcal{D}(p_{i})} \cap
F^{-1}(c_{k} - \epsilon) = \{(0, v_{1}, 0) \mid \|v_{1}\|^{2} = 2
\epsilon \}$. Let $B_{2} = \{ v_{2} \mid \|v_{2}\|^{2} < \epsilon
\}$. Then we have a map $\alpha_{i}: S_{i}^{-} \times B_{2}
\longrightarrow F^{-1}(c_{k} - \epsilon) \cap U_{i}$ defined by
\[
\alpha_{i} (s, v_{1}, 0, v_{2}) = ((\|v_{2}\|^{2} + 2
\epsilon)^{\frac{1}{2}} (2\epsilon)^{-\frac{1}{2}} s, (\|v_{2}\|^{2}
+ 2 \epsilon)^{\frac{1}{2}} (2\epsilon)^{-\frac{1}{2}} v_{1},
v_{2}).
\]
Clearly, $\alpha_{i}$ is a diffeomorphism to its image, and
$\alpha_{i}(\widetilde{S_{i}^{-}} \times B_{2}) \subseteq \partial
M^{a}$. There is a $v_{2,i} \in B_{2}$ such that, for all critical
points $q \in M^{c_{k-1}}$, $\alpha_{i}: S_{i}^{-} \times
\{v_{2,i}\} \longrightarrow F^{-1}(c_{k} - \epsilon)$ is transverse
to $\mathcal{A}(q) \cap F^{-1}(c_{k} - \epsilon)$, and $\alpha_{i}:
\widetilde{S_{i}^{-}} \times \{v_{2,i}\} \longrightarrow
F^{-1}(c_{k} - \epsilon) \cap \partial M^{a}$ is transverse to
$\widetilde{\mathcal{A}(q)} \cap F^{-1}(c_{k} - \epsilon)$. Define
$\alpha_{t,i}: S_{i}^{-} \longrightarrow F^{-1}(c_{k} - \epsilon)$
by $\alpha_{t,i}(s, v_{1}, 0) = \alpha_{i} (s, v_{1}, 0, t v_{2,i})$
for $t \in [0,1]$. When $t$ varies in $[0,1]$, $\alpha_{t,i}:
S_{i}^{-} \longrightarrow F^{-1}(c_{k} - \epsilon)$ is an isotopy of
embeddings, and its restriction to $\widetilde{S_{i}^{-}}$ is also
an isotopy of embeddings $\alpha_{t,i}: \widetilde{S_{i}^{-}}
\longrightarrow F^{-1}(c_{k} - \epsilon) \cap
\partial M^{a}$. Moreover, $\alpha_{t,i}$ is $Z_{2}$ equivariant.
Following \cite{milnor2}, we can extend $\alpha_{t,i}$ to be a
$Z_{2}$ equivariant isotopy of $F^{-1}(c_{k} - \epsilon)$, which we
still denote by $\alpha_{t,i}$, such that $\alpha_{0,i}$ is the
identity, and $\alpha_{t,i}$ is the identity outside of $U_{i}$ for
all $t$.

Since $U_{i}$ are disjoint for all $i$, composing these isotopies
$\alpha_{t,i}$, we get a $Z_{2}$ equivariant isotopy $\alpha_{t}$ of
$F^{-1}(c_{k} - \epsilon)$ such that $\alpha_{t}(U_{i}) = U_{i}$ and
$\alpha_{t}|_{U_{i}} = \alpha_{t,i}|_{U_{i}}$. We have $\alpha_{0}$
is the identity, and for all critical points $q \in M^{c_{k-1}}$,
$\alpha_{1}: S_{i}^{-} \longrightarrow F^{-1}(c_{k} - \epsilon)$ is
transverse to $\mathcal{A}(q) \cap F^{-1}(c_{k} - \epsilon)$, and
$\alpha_{1}: \widetilde{S_{i}^{-}} \longrightarrow F^{-1}(c_{k} -
\epsilon) \cap \partial M^{a}$ is transverse to
$\widetilde{\mathcal{A}(q)} \cap F^{-1}(c_{k} - \epsilon)$.

By this isotopy $\alpha_{t}$ and its $Z_{2}$ equivariance, following
\cite{milnor2}, we can modify $\xi$ in $M^{c_{k} - \epsilon, c_{k} -
\frac{1}{2} \epsilon}$ such that the new $\xi$ is still $Z_{2}$
invariant, and $\mathcal{D}(p_{i})$
($\widetilde{\mathcal{D}(p_{i})}$) is transverse to $\mathcal{A}(q)$
($\widetilde{\mathcal{A}(q)}$) for all $p_{i}$ and all critical
points $q \in M^{c_{k-1}}$. Since $\xi$ only changed in $M^{c_{k} -
\epsilon, c_{k} - \frac{1}{2} \epsilon}$, we have $\xi$ and
$\xi|_{\partial M^{a}}$ are still locally trivial, and on
$M^{c_{k-1}}$ nothing has changed. Thus we get a desired $\xi$ for
$M^{c_{k}}$.

The above two steps complete the induction.
\end{proof}

By Lemma \ref{symmetry_dynamic}, $\xi$ and $\xi|_{\partial M^{a}}$
give $2M^{a}$ and $\partial M^{a}$ a CW decomposition respectively.
We shall consider the relation between these two decompositions. Use
the same notations as in the proof of Lemma \ref{symmetry_dynamic},
denote the descending and ascending manifolds of $\xi|_{\partial
M^{a}}$ by $\widetilde{\mathcal{D}(p)}$ and
$\widetilde{\mathcal{A}(p)}$. It's easy to see that $\mathcal{D}(p)
\cap \mathcal{A}(q) = \widetilde{\mathcal{D}(p)} \cap
\widetilde{\mathcal{A}(q)}$ when $p$, $q \in \partial M^{a}$. Thus
the moduli spaces $\mathcal{M}(p,q)$ of $\xi$ and $\xi|_{\partial
M^{a}}$ are the same. Then $\overline{\mathcal{D}(p)} =
\bigsqcup_{I} \mathcal{M}_{I} \times \mathcal{D}(r_{k})$ and
$\overline{\widetilde{\mathcal{D}(p)}} = \bigsqcup_{I \subseteq
\partial M^{a}} \mathcal{M}_{I} \times
\widetilde{\mathcal{D}(r_{k})}$. Since
$\widetilde{\mathcal{D}(r_{k})} \subseteq \mathcal{D}(r_{k})$, there
is a natural embedding $\theta:
\overline{\widetilde{\mathcal{D}(p)}} \hookrightarrow
\overline{\mathcal{D}(p)}$. In addition, suppose $\Gamma$ is a
generalized flow line connecting $p$ and $x$, then $\sigma \Gamma$
is a generalized flow line connecting $\sigma p = p$ and $\sigma x$.
Thus there is a $Z_{2}$ action on $\overline{\mathcal{D}(p)}$.

\begin{lemma}\label{action_d(p)}
Suppose $p \in \partial M^{a}$. Then $\theta:
\overline{\widetilde{\mathcal{D}(p)}} \hookrightarrow
\overline{\mathcal{D}(p)}$ is a smooth embedding. The action of
$Z_{2}$ on $\overline{\mathcal{D}(p)}$ is smooth and $\text{\rm Im}
\theta = \text{\rm Fix}(Z_{2}; \overline{\mathcal{D}(p)})$. In
addition, $\widetilde{e} = e \theta$, where $\widetilde{e}$ is the
characteristic map $\widetilde{e}:
\overline{\widetilde{\mathcal{D}(p)}} \longrightarrow \partial
M^{a}$ and $e$ is the characteristic map $e:
\overline{\mathcal{D}(p)} \longrightarrow 2M^{a}$.
\end{lemma}
\begin{proof}
Except for smoothness, this lemma is obviously true. We only need to
prove smoothness. This is a local property.

Suppose the critical values in $(-\infty, f(p)]$ are $c_{l} < \cdots
< c_{0}$. Denote $M(i) = F^{-1} ((c_{i+1}, c_{i-1}))$, $U(i) =
e^{-1}(M(i))$ and $\widetilde{U}(i) = \widetilde{e}^{-1}(M(i))$.
Choose $a_{i} \in (c_{i-1}, c_{i+1})$, by Lemma
\ref{d(p)_embedding}, we have the following commutative diagram, and
both $E(i)$ and $\widetilde{E}(i)$ are smooth embeddings. Thus
$\theta$ is a smooth embedding.
\[
\xymatrix{
  \widetilde{U}(i) \ar[d]_{\theta} \ar[r]^-{\widetilde{E}(i)} &  \prod_{j=0}^{i-1} F^{-1}(a_{j}) \times M(i)      \\
  U(i) \ar[ur]_{E(i)}                     }
\]

Since $F$ is $Z_{2}$ invariant, there is a smooth $Z_{2}$ action on
$\prod_{j=0}^{i-1} F^{-1}(a_{j}) \times M(i)$ and $E(i)$ is $Z_{2}$
equivariant. Thus the action of $Z_{2}$ on $U(i)$ is smooth.
\end{proof}

\begin{proof}[Proof of Theorem \ref{CW}]
For briefness, we shall not distinguish between a CW complex and its
underlying space in this proof.

The function $F$ in Lemma \ref{symmetry_function} and the vector
fields $\xi$ and $\xi|_{\partial M^{a}}$ in Lemma
\ref{symmetry_dynamic} give two CW decompositions. They are $2M^{a}
= \bigsqcup_{p} \mathcal{D}(p)$ with characteristic maps $e:
\overline{\mathcal{D}(p)} \longrightarrow 2M^{a}$ and $\partial
M^{a} = \bigsqcup_{p \in
\partial M^{a}} \widetilde{\mathcal{D}(p)}$ with characteristic maps $\widetilde{e}:
\overline{\widetilde{\mathcal{D}(p)}} \longrightarrow \partial
M^{a}$. The decomposition of $2M^{a}$ is $Z_{2}$ invariant, $K^{a} =
\bigsqcup_{p \in M^{a} - \partial M^{a}} \mathcal{D}(p)$ is a
subcomplex of $2M^{a}$, and $\bigsqcup_{p \in 2M^{a} - M^{a}}
\mathcal{D}(p) = \sigma(K^{a}) \subseteq 2M^{a} - M^{a}$. However,
there is still no CW structure on $M^{a}$. We shall expand $K^{a}$
to $M^{a}$ by a sequence of elementary expansions (compare \cite[p.
14]{m.cohen}), which gives $M^{a}$ a CW structure.

For clarity, denote the characteristic map for
$\overline{\mathcal{D}(p)}$ by $e_{p}$. Suppose $p \in \partial
M^{a}$, and denote $e_{p}^{-1}(M^{a})$ by $\frac{1}{2}
\overline{\mathcal{D}(p)}$.

Construct a vector field $\widetilde{X}$ on
$\overline{\mathcal{D}(p)}$ as Lemma \ref{modified_vector}, i.e.,
$\widetilde{X} (F \circ e_{p}) \leq \xi F$, $\widetilde{X}$ equals
$\xi$ near $p$ in $\mathcal{D}(p)$, and $\widetilde{X}$ is strictly
outward on $\partial \overline{\mathcal{D}(p)}$. By Lemma
\ref{action_d(p)}, $\sigma \widetilde{X}$ has the same property as
$\widetilde{X}$ does. By Lemma \ref{strict_outward_sum}, and
replacing $\widetilde{X}$ by $\frac{1}{2}(\widetilde{X} + \sigma
\widetilde{X})$ if necessary, we may assume $\widetilde{X}$ is
$Z_{2}$ invariant. By the $Z_{2}$ invariance of $F$, Lemma
\ref{action_d(p)} and the proof of Theorem \ref{disk}, the $Z_{2}$
equivariant flow generated by $\widetilde{X}$ gives a homeomorphism
\[
\Psi: \left( \frac{1}{2} D^{\textrm{ind}(p)}, D^{\textrm{ind}(p)-1}
\right) \longrightarrow \left( \frac{1}{2}
\overline{\mathcal{D}(p)}, \overline{\widetilde{\mathcal{D}(p)}}
\right),
\]
where $\frac{1}{2} D^{\textrm{ind}(p)} = \{ (s, v_{1}) \in [0,
+\infty) \times V_{-} \mid s^{2} + \| v_{1} \|^{2} \leq \epsilon
\}$, $D^{\textrm{ind}(p)-1} = \{ (0, v_{1}) \in \{0\} \times V_{-}
\mid \| v_{1} \|^{2} \leq \epsilon \}$, and $V_{-} \times \{0\}$ is
the descending subspace of $T_{p} \partial M^{a}$.

Denote the $k$ skeletons of $2M^{a}$ and $\partial M^{a}$ by $L_{k}$
and $\widetilde{L}_{k}$ respectively. If $\textrm{ind}(p) = n$, then
$e_{p}(\partial \overline{\widetilde{\mathcal{D}(p)}}) \subseteq
\widetilde{L}_{n-2}$, $e_{p}(\partial (\frac{1}{2}
\overline{\mathcal{D}(p)})) \subseteq (L_{n-1} \cap M^{a}) \cup
\widetilde{L}_{n-1}$ and $e_{p}(\partial (\frac{1}{2}
\overline{\mathcal{D}(p)}) - \widetilde{\mathcal{D}(p)}) \subseteq
L_{n-1} \cap M^{a}$.

Expand $K^{a}$ by attaching cell pairs $e_{p}: (\frac{1}{2}
\overline{\mathcal{D}(p)}, \overline{\widetilde{\mathcal{D}(p)}})
\longrightarrow (M^{a}, \partial M^{a})$ for critical points $p \in
\partial M^{a}$ by induction on $\textrm{ind}(p)$. Then $K^{a}$
expands by elementary expansions to a CW complex $N$ such that
$K^{a}$ and $\partial M^{a}$ are its subcomplexes. Clearly, $N
\subseteq M^{a}$. In addition, if $x \in M^{a} - K^{a}$, then $x \in
\mathcal{D}(p)$ for some $p \in
\partial M^{a}$ because $\mathcal{D}(q) \subseteq 2M^{a} - M^{a}$
when $q \notin M^{a}$. Since $\frac{1}{2} \overline{\mathcal{D}(p)}
= e_{p}^{-1}(M^{a})$, then $x \in e_{p}(\frac{1}{2}
\overline{\mathcal{D}(p)}) \subseteq N$. Thus $N = M^{a}$ as sets.
Finally, $N$ and $M^{a}$ share the same topology since $N$ is a
finite complex.
\end{proof}

%--------------------------------------------------------------------------------------------------
\subsection{Proof of Theorem \ref{boundary_operator}}
\begin{proof}
In this proof, all critical points have function values less than
$a$.

Suppose $\textrm{ind}(p)=k$. $[\overline{\mathcal{D}(p)}]$ is a base
of $H_{k}(\overline{\mathcal{D}(p)}, \partial
\overline{\mathcal{D}(p)})$. It's well known that $\partial
[\overline{\mathcal{D}(p)}]$ is the image of
$[\overline{\mathcal{D}(p)}]$ under the following composition of
homomorphisms
\begin{eqnarray*}
H_{k}(\overline{\mathcal{D}(p)}, \partial \overline{\mathcal{D}(p)})
\longrightarrow H_{k-1}(\partial \overline{\mathcal{D}(p)})
\longrightarrow \widetilde{H}_{k-1}(\partial
\overline{\mathcal{D}(p)} / e^{-1}(K^{a}_{k-2})) \\
\longrightarrow \widetilde{H}_{k-1}(K^{a}_{k-1} / K^{a}_{k-2}) =
H_{k-1}(K^{a}_{k-1}, K^{a}_{k-2}),
\end{eqnarray*}
where $K^{a}_{n} = \bigsqcup_{\textrm{ind}(q) \leq n}
\mathcal{D}(q)$, and $\partial \overline{\mathcal{D}(p)} =
\bigsqcup_{i} \partial^{i} \overline{\mathcal{D}(p)}$ is the full
boundary of $\overline{\mathcal{D}(p)}$. The first homomorphism
follows from the homology long exact sequence, the second one
follows from the quotient map $\partial \overline{\mathcal{D}(p)}
\longrightarrow
\partial \overline{\mathcal{D}(p)} / e^{-1}(K^{a}_{k-2})$, and the
third one follows from the map $\partial \overline{\mathcal{D}(p)} /
e^{-1}(K^{a}_{k-2}) \longrightarrow K^{a}_{k-1} / K^{a}_{k-2}$
induced by $e$. Denote the first homomorphism by $\varphi_{1}$ and
the composition of the first two by $\varphi_{2}$. The composition
of all of them is the boundary operator $\partial$.

We have $e^{-1}(K^{a}_{k-2}) = \bigsqcup_{\textrm{ind}(q)<k-1}
\mathcal{M}(p,q) \times \mathcal{D}(q) \sqcup \bigsqcup_{|I|>0}
\mathcal{D}_{I}$. Thus there is the following wadge of spheres with
dimension $k-1$
\[
\partial \overline{\mathcal{D}(p)} / e^{-1}(K^{a}_{k-2}) =
\bigvee_{\textrm{ind}(q)=k-1} \  \bigvee_{x \in \mathcal{M}(p,q)}
\{x\} \times \overline{\mathcal{D}(q)} / \partial(\{x\} \times
\overline{\mathcal{D}(q)}),
\]
where the base points of spheres are $\partial(\{x\} \times
\overline{\mathcal{D}(q)}) / \partial(\{x\} \times
\overline{\mathcal{D}(q)})$.

Clearly, $\overline{\mathcal{D}(p)}$ is a topological manifold with
boundary $\partial \overline{\mathcal{D}(p)}$, and
$[\overline{\mathcal{D}(p)}]$ represents an orientation of
$\overline{\mathcal{D}(p)}$. So
$\varphi_{1}([\overline{\mathcal{D}(p)}])$ represents the boundary
orientation of $\partial \overline{\mathcal{D}(p)}$ induced from
$[\overline{\mathcal{D}(p)}]$. Give $\{x\} \times
\overline{\mathcal{D}(q)}$ the orientation
$[\overline{\mathcal{D}(q)}]$ of $\overline{\mathcal{D}(q)}$ by the
natural identification. Denote by $[\{x\} \times
\overline{\mathcal{D}(q)}]$ the element in $
\widetilde{H}_{k-1}(\{x\} \times \overline{\mathcal{D}(q)} /
\partial(\{x\} \times \overline{\mathcal{D}(q)})) \subseteq
\widetilde{H}_{k-1}(\partial \overline{\mathcal{D}(p)} /
e^{-1}(K^{a}_{k-2}))$ which represents this orientation. Then by (2)
of Theorem \ref{orientation}, we have
\[
\varphi_{2}([\overline{\mathcal{D}(p)}]) =
\sum_{\textrm{ind}{q}=k-1} \  \sum _{x \in \mathcal{M}(p,q)}
\varepsilon(x) [\{x\} \times \overline{\mathcal{D}(q)}],
\]
where $\varepsilon(x)$ is the orientation $\pm 1$ at $x \in
\mathcal{M}(p,q)$. Thus
\[
\partial [\overline{\mathcal{D}(p)}] = \sum_{\textrm{ind}(q)=k-1} \  \sum_{x \in
\mathcal{M}(p,q)} \varepsilon(x) [\overline{\mathcal{D}(q)}] =
\sum_{\textrm{ind}(q) = \textrm{ind}(p)-1} \# \mathcal{M}(p,q)
[\overline{\mathcal{D}(q)}].
\]
\end{proof}

%--------------------------------------------------------------------------------------------------
%--------------------------------------------------------------------------------------------------
\section*{Appendix}
\begin{proof}[Proof of Lemma \ref{smoothness_norm}]
Define $\varphi: [0, +\infty) \times S(H) \longrightarrow H \times
S(H)$ by $\varphi(\lambda, v) = (\lambda v, v)$. Then
\[
d \varphi \frac{\partial}{\partial \lambda} = (v, 0), \qquad d
\varphi \frac{\partial}{\partial v} = \left( \lambda
\frac{\partial}{\partial v}, \frac{\partial}{\partial v} \right).
\]
Thus $d \varphi$ is nonsingular everywhere.

Define $\theta: H \times S(H) \longrightarrow [0, +\infty) \times
S(H)$ by $\theta (v_{1}, v_{2}) = (\| v_{1} \|, v_{2})$. Then
$\theta$ is continuous and $\theta \varphi = \textrm{Id}$. Thus
$\varphi$ is a smooth embedding. Then $\varphi^{-1}:
\textrm{Im}\varphi \longrightarrow [0, +\infty) \times S(H)$ is also
smooth.

Clearly, $\forall x \in L$, $(g(x), \tilde{g}(x)) \in \textrm{Im}
\varphi$, and $\varphi^{-1} (g(x), \tilde{g}(x)) = (\|g(x)\|,
\tilde{g}(x))$. Since $\varphi^{-1}$, $g(x)$ and $\tilde{g}(x)$ are
smooth, then so is $\|g(x)\|$.
\end{proof}

\begin{proof}[Proof of Lemma \ref{gradient-like_gradient}]
Define a matrix
\[
A_{1} = \left(
\begin{array}{cc}
\cos \theta & \sin \theta \\
\sin \theta & \frac{1}{\cos \theta}
\end{array}
\right)
\]
for $\theta \in [0, \frac{\pi}{2})$. Then $A_{1}$ is a symmetric
positive matrix and $A_{1} (1,0)^{T} = (\cos \theta, \sin
\theta)^{T}$, where $(*,*)^{T}$ is the transpose of $(*,*)$.

Suppose $V_{1}$ and $V_{2}$ are two vectors in a Hilbert space such
that $\langle V_{1}, V_{2} \rangle > 0$, i.e., the angle between
$V_{1}$ and $V_{2}$ is less than $\frac{\pi}{2}$. We define a
symmetric positive operator $A(V_{1},V_{2})$ such that
$A(V_{1},V_{2})V_{1}=V_{2}$ as follows.

If $V_{1}$ and $V_{2}$ are colinear, then define $A(V_{1},V_{2}) =
\frac{\|V_{2}\|}{\|V_{1}\|} \textrm{Id}$. If $V_{1}$ and $V_{2}$ are
not colinear, then they span a plane $V_{1} \wedge V_{2}$. First, we
define an operator $A_{2}(e_{1},e_{2})$ for $e_{1} =
\frac{V_{1}}{\|V_{1}\|}$ and $e_{2} = \frac{V_{2}}{\|V_{2}\|}$. In
$(V_{1} \wedge V_{2})^{\bot}$, $A_{2}(e_{1},e_{2})$ is the identity.
In $V_{1} \wedge V_{2}$, $A_{2}(e_{1},e_{2})$ is the above $A_{1}$
mapping $e_{1}$ to $e_{2}$. Define $A(V_{1},V_{2}) =
\frac{\|V_{2}\|}{\|V_{1}\|} A_{2}(e_{1}, e_{2})$.

Thus, in general, $A(V_{1},V_{2}) = \frac{\|V_{2}\|}{\|V_{1}\|}
A_{2}(\frac{V_{1}}{\|V_{1}\|}, \frac{V_{2}}{\|V_{2}\|})$ for
$\langle V_{1}, V_{2} \rangle > 0$. Here, for $\|e_{1}\| = \|e_{2}\|
= 1$ and $\langle e_{1}, e_{2} \rangle > 0$, we have
\begin{eqnarray*}
A_{2}(e_{1}, e_{2}) Y & = & Y + \frac{\langle e_{1}, e_{2} \rangle
\langle e_{2}, Y \rangle - (1 + \langle e_{1}, e_{2} \rangle +
\langle e_{1}, e_{2} \rangle^{2})\langle e_{1}, Y \rangle}{1 +
\langle e_{1}, e_{2} \rangle} e_{1} \\
& & + \frac{\langle e_{1}, e_{2} \rangle^{2} \langle e_{1}, Y
\rangle + \langle e_{2}, Y \rangle}{\langle e_{1}, e_{2} \rangle (1
+ \langle e_{1}, e_{2} \rangle)} e_{2}.
\end{eqnarray*}
Then $A(V_{1},V_{2})$ smoothly depends on $V_{1}$ and $V_{2}$, and
$A(V_{1},V_{1}) = \textrm{Id}$.

Let $G_{1}$ be the metric associated with $X$. Denote the gradient
vector field of $f$ with respect to $G_{1}$ by $\nabla_{G_{1}} f$.
Then $\nabla_{G_{1}} f$ equals $X$ near each critical point, and
$\langle \nabla_{G_{1}} f, X \rangle = Xf > 0$ at each regular
point. Define the operator $A(X, \nabla_{G_{1}} f)$ as above at each
regular point. Define $A(X, \nabla_{G_{1}} f) = \textrm{Id}$ at each
critical point. Then $A(X, \nabla_{G_{1}} f)$ is smooth on $M$.
Define a new metric $G_{2}$ such that $\langle *, * \rangle_{G_{2}}
= \langle A(X, \nabla_{G_{1}} f) *, * \rangle_{G_{1}}$. Then
$\nabla_{G_{2}} f = X$.
\end{proof}

%--------------------------------------------------------------------------------------------------
%--------------------------------------------------------------------------------------------------
\section*{Acknowledgements}
I wish to thank Profs. Octav Cornea, Matthias Schwarz and Katrin
Wehrheim, who told me which work had already been done. I wish to
thank the referees of this paper who referred me to
\cite{abbondandolo_majer1} and \cite{abbondandolo_majer2}. I wish to
thank Prof. Ralph Cohen for suggesting to me this topic, his patient
teaching, and his important referencing. I wish to thank Prof. Dan
Burghelea for his explanation of \cite{burghelea_haller} and for his
generosity in sending me part of
\cite{burghelea_friedlander_kappeler}. I'm indebted to Prof. John
Milnor for his advice and suggesting \cite{kalmbach2}, which
motivated me to prove Theorem \ref{disk}. All of Section
\ref{section_CW_structure} is a result of his influence. Most
importantly, I am very grateful to my PhD advisor Prof. John Klein.
It was he who helped me communicate with the above professors. I am
also indebted to him for his direction, his patient educating, and
his continuous encouragement.

%--------------------------------------------------------------------------------------------------
%--------------------------------------------------------------------------------------------------

\end{document}